\documentclass{article}
\usepackage[utf8]{inputenc}

\usepackage{enumerate}
\usepackage{dsfont}
\usepackage{amssymb}
\usepackage{amsmath}
\usepackage{mathtools}
\usepackage{mathrsfs}
\usepackage{amsthm}
\usepackage{graphicx}
\usepackage{wrapfig}
\usepackage{lipsum}
\usepackage{float}
\usepackage{subcaption}
\usepackage[dvipsnames]{xcolor}
\usepackage{autobreak}
\usepackage{hyperref}
\usepackage{color}
\usepackage{todonotes}
\usepackage{soul}
\usepackage{authblk}
\usepackage{geometry}
\usepackage{chngcntr}
\usepackage{thmtools}
\usepackage{thm-restate}
\usepackage{bm}
\usepackage{accents}
\usepackage{cancel}
\usepackage{algorithm}
\usepackage{algpseudocode}
\usepackage{yhmath}
\usepackage{relsize}
\newtheorem{theorem}{Theorem}[section]
\newtheorem{definition}[theorem]{Definition}
\newtheorem{lemma}[theorem]{Lemma}

\newtheorem{remark}[theorem]{Remark}

\newtheorem{proposition}[theorem]{Proposition}

\newtheorem{condition}[theorem]{Condition}

\counterwithin{theorem}{section}
\counterwithin{equation}{section}

\newcommand{\sss}{\scriptscriptstyle}

\newcommand*\epidn{\mathscr{E}^{(\rho)}_n(t)}
\newcommand*\susn{\mathcal{S}^{(\rho)}_n(t)}
\newcommand*\infn{\mathcal{I}^{(\rho)}_n(t)}
\newcommand*\recn{\mathcal{R}^{(\rho)}_n(t)}

\newcommand*\epidnr{\mathscr{E}^{(\rho)}_{n,r}(t)}
\newcommand*\susnr{\mathcal{S}^{(\rho)}_{n,r}(t)}
\newcommand*\infnr{\mathcal{I}^{(\rho)}_{n,r}(t)}
\newcommand*\recnr{\mathcal{R}^{(\rho)}_{n,r}(t)}

\newcommand*\sus{\mathcal{S}^{(\rho)}(t)}
\newcommand*\infect{\mathcal{I}^{(\rho)}(t)}
\newcommand*\rec{\mathcal{R}^{(\rho)}(t)}

\newcommand*\bgrg{\mathrm{BGRG}_n(\bm{w})}
\newcommand*\drig{\mathrm{DRIG}_n(\bm{w})}

\newcommand*\bgrgs{\mathrm{BGRG}_n^{s}(\bm{w})}
\newcommand*\drigs{\mathrm{DRIG}_n^{s}(\bm{w})}

\newcommand*\markedunion{G_n^{\sss T,(\mathrm{ON, OFF})}}
\newcommand*\markedunionroot{\left(G_n^{\sss T,(\mathrm{ON, OFF})},o_n\right)}
\newcommand*\markedunionv{\left(G_n^{\sss T,(\mathrm{ON, OFF})},v\right)}
\newcommand*\markedg{G^{\sss T,(\mathrm{ON, OFF})}}
\newcommand*\marklim{\left(G^{\sss T,(\mathrm{ON, OFF})},o\right)}
\newcommand*\marklimbar{\left(\Bar{G}^{\sss T,(\mathrm{ON}, \mathrm{OFF})},\Bar{o}\right)}
\newcommand*\hmarkedroot{\left(H^{\sss(\mathrm{ON, OFF})},o\right)}

\newcommand*\markeduniontilde{\Tilde{G}_n^{\sss T,(\mathrm{ON, OFF})}}

\newcommand*\piaon{\pi^a_{\text{\rm{ON}}}}
\newcommand*\piaoff{\pi^a_{\text{\rm{OFF}}}}

\newcommand*\laon{\lambda^a_{\text{\rm{ON}}}}
\newcommand*\laoff{\lambda^a_{\text{\rm{OFF}}}}

\newcommand*\sigeon[2]{\sigma^{\sss #1}_{\sss #2,\text{\rm{ON}}}}
\newcommand*\sigeoff[2]{\sigma^{\sss #1}_{\sss #2,\text{\rm{OFF}}}}

\newcommand*\teon[2]{t^{\sss #1}_{\sss #2,\text{\rm{ON}}}}
\newcommand*\teoff[2]{t^{\sss #1}_{\sss #2,\text{\rm{OFF}}}}

\newcommand*\Var{\mathrm{Var}}

\newcommand{\longversion}[1]{}

\title{Dynamic random intersection graph:\\
Dynamic local convergence and giant structure}

\author[1]{Marta Milewska}
\author[2]{Remco van der Hofstad}
\author[1,2]{Bert Zwart}

\date{}

\affil[1]{Centrum Wiskunde \& Informatica (CWI), P.O. Box 94079, 1090 GB Amsterdam, The Netherlands}
\affil[2]{Department of Mathematics and Computer Science, Eindhoven University of Technology, 5600 MB Eindhoven, The Netherlands}

\title{SIR on locally converging dynamic random graphs}

\begin{document}

\maketitle

\begin{abstract}
In this paper, we study the trajectory of a classic SIR epidemic on a family of dynamic random graphs of fixed size, whose set of edges continuously evolves over time. We set general infection and recovery times, and start the epidemic from a positive, yet small, proportion of vertices. We show that in such a case, the spread of an infectious disease around a typical individual can be approximated by the spread of the disease in a local neighbourhood of a uniformly chosen vertex.

We formalize this by studying general dynamic random graphs that converge dynamically locally in probability and demonstrate that the epidemic on these graphs converges to the epidemic on their dynamic local limit graphs. We  provide a detailed treatment of the theory of dynamic local convergence, which remains a relatively new topic in the study of random graphs. One main conclusion of our paper is that a \emph{specific} form of dynamic local convergence is required for our results to hold.
\end{abstract}


\section{Introduction}
The first mathematical model of epidemic evolution is accredited to David Bernoulli and his analysis of the smallpox outbreak, presented at the Royal Academy of Sciences in  Paris in 1760 with the aim of advocating a preventive inoculation against the disease. His work was published in \cite{Bernoulli_1766} in 1766. Afterwards, the most significant contributions to the mathematical modelling of infectious diseases were made in the 20th century \cite{Kendall_1956, McKendrick_1927, Ross_1916, Ross_1917}. 

Mathematical models of epidemics in the above works differ in details and assumptions; however, they have a common denominator: they are \emph{compartmental}. Compartmental models assign each individual to one of the distinguished categories (i.e., compartments), which reflect their status with respect to the disease, such as `infected' or `recovered'. Since their emergence, these models have been continuously gaining popularity and nowadays they constitute a fundamental apparatus for predicting disease evolution.

Compartmental models are widely used to explore various aspects of epidemics, such as the total number of infections, the duration, or the rate of growth \cite{Eubank_2004, Larson_2007, Lloyd-Smith_2009}. They are also valuable tools in assessing public health policies, such as vaccination strategies and lockdown measures \cite{Acemoglu_2023, Bastani_2021, Birge_2022, Mamani_2013}. Moving beyond the realm of infectious diseases, compartmental models have applications in social sciences, for instance, to analyze the spread of information or opinions \cite{Bass_1969}. Mathematically, these approaches often rely on deterministic or stochastic partial differential equations (PDEs) \cite{Bartlett_1949, Britton_2019, Dimitrov_2010}, but due to the simplifications inherent in such systems, they may not capture the full complexity of disease dynamics.

Random graph models offer an alternative, allowing for more flexibility and randomness in both the disease progression and the network structure \cite{Croccolo_2020, Fransson_2019, Manshadi_2020}. However, most research using such models considers \emph{static} networks, which do not account for the continuous evolution of interpersonal contacts over time. In reality, social connections change frequently and vary in strength: some, like workplace or household interactions, last longer, while others, such as those formed during commutes, are more transient and unpredictable. For a more realistic depiction, dynamic network models have been developed \cite{Altmann_1995, Ball_2017, Britton_2016, Lashari_2018}. These models capture the temporal variability of connections but are inherently more complex to analyze, as they require tracking both the epidemic’s progression and the evolving population structure. Therefore, the mathematical analysis of these dynamic models is often highly specific to each situation, depending on the nature of the model and the context of the epidemic.

In this paper, we take a more general approach. We consider a dynamic network without a demographic component, but with constantly evolving connections between the network members. We take a general approach, refraining from specifying the graph and its dynamic in detail, and consider a Susceptible-Infected-Recovered (SIR) epidemic started from a positive proportion of infected vertices with general infection and recovery times.
We require the dynamic to shape the graph in a way that results in convergence of the proportion of vertices whose neighbourhoods have some particular dynamic structure. We then base our methodology on the consequences of this convergence. The power of such a technique consists in its applicability to a large number of dynamic random graphs at the same time.

The convergence of the proportion of vertices with a particular neighbourhood structure is referred to as \emph{local convergence}, a concept that is well-established and thoroughly studied in the context of static graphs. However, for dynamic random graphs, this notion is relatively recent, having been introduced in \cite{dynweaklimit2023} and independently in our work \cite{Milewska2023}, where we explored it within the framework of a specific dynamic graph model. In this paper, we present the concept in a more general setting and examine its nuances by introducing a stronger form of local convergence for dynamic random graphs. Our primary result demonstrates the convergence of an epidemic process on the dynamic graph to the corresponding process on its dynamic local limit, under the assumption of this stronger form of dynamic local convergence. This result propagates the use of sparse graph limits in epidemiological modelling implying that a disease outbreak is to a large extent a \emph{local} property of the graph. This approach builds upon the results and proofs presented in \cite{Alimohammadi2024}, extending their findings from the context of locally converging static graphs to dynamic random graphs.\\

\noindent\textbf{Main innovations of this paper:}\\

\noindent\textbf{Extension of \cite{Alimohammadi2024} to a dynamic setting.} In this paper, we extend the results of \cite{Alimohammadi2024}, where the authors establish a connection between epidemic processes on random static graphs that converge locally, and the corresponding processes on the local limits of these graphs. Specifically, we generalize this framework to dynamic graphs, characterized by a fixed vertex set and an evolving edge set, thereby adapting the analysis to a dynamic context.\\

\noindent\textbf{Description of dynamic local convergence.} Dynamic local convergence, that is, the local convergence of dynamic graphs, is a relatively new concept that was first explored independently in \cite{dynweaklimit2023} and \cite{Milewska2023} and further analysed in \cite{Rath_2024}. However, these works define the concept differently. In this paper, we revisit our approach where the dynamic graph is treated as a stochastic process in time, and its convergence is derived by applying the well-established theory of convergence of processes with discontinuities. We briefly describe the theoretical background of this framework and then specify it to the space of rooted dynamic graphs, offering a comprehensive characterization of dynamic local convergence applicable to a broad class of dynamic random graphs, without limiting the analysis to any particular model. This approach provides a systematic, self-contained presentation of dynamic local convergence.\\

\noindent\textbf{Unraveling complications of local convergence of dynamic graphs.} While extending the results of \cite{Alimohammadi2024} to the dynamic setting, we observed that the relationship between the graph and processes evolving on it becomes significantly more intricate in the dynamic case compared to the static one. Specifically, the information obtained from dynamic local convergence, as introduced in \cite{Milewska2023}, which captures the structure of the connected component of a uniformly chosen vertex at successive time instances, does not suffice for analyzing certain processes acting on the network, such as epidemic spreading. As a result, a stronger form of local convergence for dynamic graphs is necessary to rigorously establish the connection between an epidemic process on a dynamic graph and its counterpart on the limit. We introduce this stronger form and demonstrate its suitability for the epidemic's description. This insight also highlights an intriguing open problem in the field of dynamic graph convergence: understanding the relationship between these two forms of dynamic local convergence.\\

\noindent\textbf{Simulations on dynamic graphs.} Our main result allows for transitioning from the setting of dynamic graphs to their dynamic graph limits when analyzing epidemic processes. In many cases, including several examples presented in this paper, these dynamic graph limits take the form of dynamic rooted trees. This transition significantly simplifies the simulation of epidemics on dynamic structures, as it reduces both the computational complexity and the volume of data that needs to be analyzed.\\

\noindent\textbf{Organisation of this paper.} The paper is organised as follows: In Sections \ref{sec_dyn_graph}-\ref{sec_simulations_overview} we provide an introduction, summarizing our contributions and introducing the key concepts of dynamic random graphs and the SIR epidemic model that we will investigate. We also give an informal description of local time-marked union convergence, which is the version of dynamic local convergence most suited for epidemics. We further present our main result, and conclude with a brief overview of our simulation study and a discussion. Section \ref{sec_locconv_theory} delves into dynamic local convergence, beginning with the necessary background on the metric space of rooted graphs and covering both local weak convergence and local convergence in probability for dynamic random graphs. Additionally, we introduce the concept of local time-marked union convergence and discuss the backward process that generates epidemic marks. In Section \ref{sec_proofs_overview}, we present an overview of the proofs, including auxiliary results and the proof of the main theorem on convergence for the epidemic process. Sections \ref{sec_all_proofs_prelim} and \ref{sec_all_proofs_lim} are dedicated to detailed proofs of the auxiliary results, covering bounds on the moments and local approximations. Finally, Section \ref{sec_simulations_details} provides background on the simulation study outlined in Section \ref{sec_simulations_overview}, with a focus on static and dynamic limiting results for Erd\H{o}s-R{\'e}nyi random graphs and random intersection graphs.

\subsection{Dynamic random graphs} \label{sec_dyn_graph}
We investigate dynamic random graphs, where the set of vertices remains fixed but the set of edges evolves over time. We formalise it in the following definition:

\begin{definition}[Dynamic random graph] \label{def_dyn_g}
A dynamic random graph process is given by $\big(G_n^s\big)_{s\in[0,T]} = \big((\mathcal{V}_n, \mathcal{E}^s_n)\big)_{s\in[0,T]}$ with a fixed vertex set $\mathcal{V}_n = [n] = \{1, \ldots, n\}$, and a dynamic edge set $\big(\mathcal{E}^s_n\big)_{s\in[0,T]}$, with paths in $\mathcal{D}\left([0,T], \mathcal{E}(K_n)\right)$ — the space of c{\`a}dl{\`a}g functions $f:[0,T] \mapsto \mathcal{E}(K_n)$, where $K_n$ is the complete graph on the vertex set $[n]$, and $\mathcal{E}(K_n)$ denotes its edge set.
\end{definition}
Consequently, we say that an edge $e \in \mathcal{E}_n^s$ is said to be `ON' at time $s \in [0,T]$ if $e \in \mathcal{E}_n^s$, meaning that it is present in the graph at time $s$. Conversely, the edge $e$ is said to be `OFF' at time $s$ if $e \notin \mathcal{E}_n^s$, meaning that it is absent from the graph at time $s$.  We will also write `$e$ switches ON’ to denote the event of changing status from being in an OFF state to being in an ON state, and analogously for the reversed roles of ON and OFF.

We leave the dynamics unspecified for now, keeping them at the most general level. However, our main result (as well as most of the auxiliary results presented in Section \ref{sec_proofs_overview}) will only apply to dynamic random graphs that \textit{converge locally in probability} (as we explain in Section \ref{sec_locconv_theory}), which will naturally impose restrictions on the types of dynamics permitted.
 
\subsection{SIR epidemic model} \label{SIR_setup}
We focus on a Susceptible-Infected-Removed (SIR) model where individuals can be susceptible, infected, or removed. Vertices in the graph will represent the individuals, and edges between these vertices indicate potential paths for the spread of the epidemic. An infected vertex transmits the disease to its neighbours independently at random times drawn from some arbitrary continuous distribution denoted by $D_I$. Every edge ever ON in $[0,T]$, i.e., every edge $\{u,v\}\in \bigcup_{s\in[0,T]}\mathcal{E}^s_n$, is associated with a unique copy $C\left(\{u,v\}\right)$ drawn from the distribution $D_I$. If $u$ or $v$ becomes infected, then $C\left(\{u,v\}\right) \sim D_I$ is compared with the ON periods of the edge $e$ to determine whether an infected vertex can transmit the disease to its susceptible neighbour. Specifically, the transmission time $C\left(\{u,v\}\right) \sim D_I$ is initialized at the infection time of the infected endpoint, provided this infection time falls within an ON period of the edge. Otherwise, $C\left(\{u,v\}\right) \sim D_I$ is initialized at the next activation time of the edge $e$. Then, the actual transmission time computed as the sum of the initialization time and transmission time $C\left(\{u,v\}\right)\sim D_I$ must fall within (any, not necessarily the first) ON period of the edge $e$ for the infection to be transmitted. For more details, see Section \ref{sec_backward_process}. An infected vertex independently recovers at a time drawn from some arbitrary distribution denoted by $D_R$. Initially, every vertex is infected independently with some fixed probability $\rho>0$.\\
We represent the epidemic process at time $t\geq0$ by the vector
\begin{align}
    \epidn = \big(\susn,\infn,\recn \big),
\end{align}
with $\susn,\infn,\recn$ denoting the (random) proportions of susceptible, infected and recovered vertices in $\big(G_n^s\big)_{s\in[0,T]}$ at time $t$, respectively.\\

\noindent\textbf{Model interpretation.} In dynamic graphs, the contact process is naturally modeled by the $\mathrm{ON}$ and $\mathrm{OFF}$ dynamics of the edges. However, to prevent the scenario where any contact between an infected and susceptible vertex instantly results in transmission, and to facilitate comparisons with static cases, a random variable $C\left(\{u,v\}\right)$ drawn from the distribution $D_I$ is assigned to each edge. This variable is initialized at the first $\mathrm{ON}$ period following the infection time, representing a latency period that reflects the inherent randomness of the transmission process.

The assignment of independent copies of $C\left(\{u,v\}\right)$ to edges also accounts for heterogeneity in transmission patterns. For models where edges only activate once (which is the most likely scenario for some of the dynamic random graph models discussed in Section \ref{sec_examples_dyn_rg}), $C\left(\{u,v\}\right)$ can be interpreted as the amount of time required for an infected individual to transmit the disease during a single contact. In cases where multiple $\mathrm{ON}$ periods occur, the direct interpretation of $C\left(\{u,v\}\right)$ as the time required for transmission during contact becomes less straightforward, as the model does not depend on the length or frequency of intervening $\mathrm{OFF}$ periods. Nevertheless, $C\left(\{u,v\}\right)$ continues to capture the stochastic nature of transmission events, which are influenced by factors such as viral load and immune response rather than the continuity of contact. Pathogen infectivity can persist even after breaks in contact, allowing for transmission during subsequent interactions. This approach remains mathematically tractable, preserves the heterogeneity of transmission rates, and provides a foundational framework for modelling disease dynamics on dynamic networks.

\subsection{Informal description of local time-marked union convergence} \label{sec_informal}
After detailing the setup of the SIR model on dynamic graphs, we intuitively describe the concepts required for understanding the main result. Throughout this paper, we heavily rely on the property of dynamic random graphs called \emph{dynamic local convergence}, as introduced in \cite{Milewska2023} and inspired by \cite{Aldous2004, Benjamini2001}. This property ensures that the dynamic neighbourhood structures of uniformly chosen vertices converge to a limiting probability distribution on dynamic rooted graphs. The vertices are called roots, and are denoted by $o_n$ in the original graph and $o$ in the limit. The space of all rooted graphs is denoted $\mathcal{G}_{\star}$. The dynamic convergence we describe guarantees that the shape of connected components around $o_n$ at subsequent time instances throughout $[0, T]$ can be characterized by a limiting dynamic graph. However, as mentioned, in the dynamic setting this convergence does not fully capture the information needed for epidemic spreading.

For example, imagine that $o_n$ is connected to vertex $u_1$ at time $[1,2]$, and $u_1$ was connected to an initially infected vertex $u_2$ during $[0,1)$. Although $u_2$ and $o_n$ are never directly connected, $u_2$ can still infect $o_n$ through $u_1$. This indirect connection is missed if we only look at the connected components of $o_n$ at specific times. Therefore, a stronger notion of local convergence is required to account for such indirect paths.

To address this, we introduce the notion of a \emph{union graph}, which consists of all edges that were ever active during the time interval $[0,T]$. More specifically, an edge $e$ is present in the union graph precisely when there exists a time $s$ such that $e\in\mathcal{E}_n^s$. This ensures that all vertices ever connected to $o_n$ during $[0,T]$, either directly or indirectly through its neighbors, are included. At each edge in the union graph, we attach a sequence of ON and OFF times, representing the (potentially long) vector of time pairs indicating when the edge switches ON and OFF. These sequences, referred to as \emph{marks}, retain precise temporal information about the dynamics of the graph. We then require the local convergence of the union graph, along with its marks, ensuring that the graph structure around a vertex, including the sequence of ON and OFF times of its incident edges, converges to a limiting (also marked) structure. This new notion, termed \emph{local time-marked union convergence}, provides comprehensive control over the evolving graph process and encapsulates all necessary information for applications such as epidemic modeling. We formalize the above concepts in Section \ref{sec_locconv_theory}, where further details can be found.

\subsection{Main result}
Having informally explained the above concepts, we now present our main theorem - the dynamic analogue of Theorem 2.5 from \cite{Alimohammadi2024}. Denote convergence in probability by `$\stackrel{\textbf{P}}{\longrightarrow}$'. The following theorem links an SIR epidemic process on a dynamic graph with an epidemic process on the local limit of its corresponding time-marked union graph:

\begin{theorem}[Dynamic convergence of the epidemic processes] \label{thm_dyn_conv_epid}
Let $(G_n^s)_{s\in[0,T]}$ be a dynamic graph sequence converging in probability in a local time-marked union sense to the limiting time-marked union graph on $\mathcal{G}_{\star}$. Consider an SIR epidemic process on $(G_n^s)_{s\in[0,T]}$, in which every vertex is initially infected independently with probability $\rho>0$, and susceptible otherwise. Then, for any $t\in[0,T]$, there are deterministic functions $\mathscr{E}(t) = \big(s(t),i(t),r(t)\big)$ such that, for any $t\in[0,T]$,
\begin{align}
    \epidn \stackrel{\textbf{P}}{\longrightarrow}\mathscr{E}(t),
\end{align}
with
\begin{align} \label{main_thm_limits}
    s(t) = \mu\big(o\in \sus\big), \hspace{0.5cm} i(t) = \mu\big(o\in \infect\big) \hspace{0.5cm}\text{and}\hspace{0.5cm} r(t) = \mu\big(o\in \rec\big), 
\end{align}
where $\mu$ on the right-hand side refers to the law of the limiting time-marked union graph, $\big(\sus,\infect,\rec\big)$ are sets of susceptible, infected and recovered vertices in an SIR epidemic on the limiting dynamic graph and every vertex is in $\mathcal{I}^{(\rho)}(0)$ independently with probability $\rho$.
\end{theorem}
We give a more explicit description of the limiting variables in (\ref{main_thm_limits}) in Section \ref{sec_backward_process}, in terms of the backwards epidemic process. Theorem \ref{thm_dyn_conv_epid} extends an equivalent finding about the epidemic on a static graph converging locally from \cite{Alimohammadi2024} to the dynamic setting, showing that also in a dynamic graph, an SIR epidemic is practically a local property of the graph. We hope that this result can facilitate the investigation of epidemics on dynamic random graphs, by transferring the problem from the detailed world of real graph dynamics to a realm of limiting graphs, which tend to be more regular and approachable.

\subsection{Some dynamic random graph models} \label{sec_simulations_overview}
In this section, we present examples of dynamic random graphs that converge locally in time-marked union sense and provide an overview of the results from our simulation study. This study leverages Theorem \ref{thm_dyn_conv_epid} to explore various aspects of epidemic dynamics on random graphs, with a focus on how dynamics influence the spread of the epidemic. We begin by introducing dynamic extensions of a well-known Erd\H{o}s-R{\'e}nyi random graph model and its generalization in the form of dynamic intersection random graph. Finally, we describe a dynamic version of the configuration model, which incorporates edge rewiring.

\subsubsection{Examples} \label{sec_examples_dyn_rg}
We now define several specific dynamic random graph models. Dynamic random graphs have been explored in various other works, such as \cite{Altmann_1995, AvenadHvdH2018, Mandjes_2019, Mandjes_2024, Sousi_2020, Trapman2016_dynER}. However, these studies do not specifically focus on their local convergence. In particular, two of the cited studies \cite{Altmann_1995, Trapman2016_dynER} investigated natural dynamic formulations of the Erd\H{o}s-R{\'e}nyi random graph. Here, we define our version in an analogous manner, with slight adjustments to the rates of the exponential variables:

\begin{definition}[Dynamic Erd\H{o}s-R{\'e}nyi random graph] \label{dyn_er_def}
Fix $\gamma>0$ and a positive integer $n$. The dynamic Erd\H{o}s-R{\'e}nyi graph $\big(\mathrm{ER}^s_n(\gamma/n) \big)_{s\in[0,T]}$, denoted shortly also by $\mathrm{DER}_n(\gamma/n)$, is a stochastic process with the number of vertices fixed at $n$ and the set of edges evolving according to the following dynamics: At time $s=0$, we draw a realisation of the static $\mathrm{ER}_n(\gamma/n)$, i.e., every edge is present independently with probability $\gamma/n$ and absent otherwise. For $s>0$, independently for each vertex pair, if no edge is present then it is added after an $Exp(\gamma/(n-\gamma))$-distributed time; if an edge is present, then the edge is removed after an $Exp(1)$-distributed time.
\end{definition}

This dynamic Erd\H{o}s-R{\'e}nyi model introduces temporal evolution by allowing edges to be added and removed dynamically. In practice, this process can capture settings where interactions between nodes change over time, while maintaining the edge density characteristic of the static Erd\H{o}s-R{\'e}nyi graph. Next, we consider an alternative variation of the dynamic Erd\H{o}s-R{\'e}nyi model that simplifies the temporal dynamics.

\begin{definition} [Alternative dynamic Erd\H{o}s-R{\'e}nyi random graph] \label{def_alt_er}
At time $s=0$ we draw a realisation of the static $\mathrm{ER}_n(\gamma/n)$, i.e., every edge is present independently with probability $\gamma/n$ and absent otherwise. For $s>0$, we let all active edges switch $\mathrm{ON}$ and $\mathrm{OFF}$ at rate 1.
\end{definition}

In this alternative model, each edge present from the start switches ON and OFF independently at a constant rate. This modification retains the randomness of edge presence but changes the timing of edge updates compared to the dynamic Erd\H{o}s-R{\'e}nyi random graph defined earlier.\\

The dynamic Erd\H{o}s-R{\'e}nyi random graphs provide a simple example of edge dynamics, where edges independently switch ON and OFF over time. However, many real-world networks exhibit more complex structures, such as community-based connectivity. To capture this, we now introduce the dynamic random intersection graph, which incorporates group-based interactions with stochastic dynamics.

\begin{definition}[Dynamic random intersection graph] \label{def_dyn_drig}
A dynamic random intersection graph is constructed as follows: Let \([n] = \{1, \dots, n\}\) denote the set of vertices, and \([n]_k\) the set of all subsets of size \(k \geq 2\) from \([n]\), which represent groups. Each group \(a \in [n]_k\) alternates between $\mathrm{ON}$ and $\mathrm{OFF}$ states independently, following a continuous-time Markov process. The $\mathrm{ON}$ and $\mathrm{OFF}$ holding times are exponentially distributed with rates \(\lambda^a_{\text{\rm{ON}}} = 1\) and \(\lambda^a_{\text{\rm{OFF}}} = \frac{f(|a|)\prod_{i \in a} w_i}{\ell_n^{|a|-1}}\), where  $|a|$ is the size of the group \(a \in \cup_{k \geq 2} [n]_k\), \(f: \mathbb{N}_{\geq 2} \to \mathbb{R}_{+}\) determines the group-size dependency, \(w_i\) is the weight of vertex \(i\), and \(\ell_n = \sum_{i \in [n]} w_i\) is the total weight. At any time \(s\), vertices \(i, j \in [n]\) are connected by an edge if they belong to at least one group \(a\) that is ON, forming dynamic cliques as group states evolve.
\end{definition}

The dynamic random intersection graph offers a flexible framework to model networks with evolving community structures. Unlike most other models, it naturally builds in \emph{clustering} in the form of many triangles being present in the network. Next, we extend our consideration to a dynamic version of the configuration model, which incorporates edge rewiring but preserves the degree sequence over time.

\begin{definition}[Configuration model with rewiring of the edges] \label{def_dyn_CM} At time \( s = 0 \), a realization of the static configuration model \(\mathrm{CM}_n(\bm{d})\) is generated. Each vertex \( v \in [n] \) is assigned \( d_v \) half-edges, where the degree sequence \(\bm{d} = (d_1, d_2, \ldots, d_n)\) satisfies that \(\ell_n=\sum_{v \in [n]} d_v\) is even. These half-edges are paired uniformly at random to form edges, resulting in a multigraph that may contain self-loops and multiple edges. For \( s > 0 \), edges are dynamically rewired. Specifically, at rate $\alpha \ell_n$, we sample two edges uniformly at random. These edges are subsequently broken into four half-edges, which are then re-paired uniformly at random.
\end{definition}

This dynamic configuration model maintains a consistent degree sequence while introducing temporal variability in the graph structure. The dynamic rewiring process allows for the possibility of re-forming the original edges, creating new connections, introducing new self-loops, or removing existing ones. Such a model is particularly suited to applications where node connectivity patterns remain fixed in terms of degrees, but connections themselves fluctuate dynamically. See Section \ref{sec_simulations_details} for more details about the model and an argumentation of their local time-marked union convergence. With these models defined, we now turn to analyzing implications of temporal changes for epidemic dynamics and evaluating the accuracy of local approximations derived from Theorem \ref{thm_dyn_conv_epid}.\\

\subsubsection{Simulation study} \label{sec_sim_on_dyn_rg}
In the remainder of this section, we employ the dynamic Erd\H{o}s-R{\'e}nyi random graph from Definition \ref{dyn_er_def} to simulate and evaluate the correctness of the local approximation established in Theorem \ref{thm_dyn_conv_epid}. Numerical validation through simulations provides confidence in the approximation's applicability beyond its theoretical foundation. Furthermore, simulations enable a direct comparison of epidemic dynamics on static and dynamic graphs with the same stationary distribution, offering insights into how graph dynamics influence key epidemic characteristics, such as the epidemic curve. While theoretical work, such as \cite{Altmann_1995}, has previously explored this question for an SIR model by comparing basic reproduction numbers ($R_0$) on static and dynamic Erd\H{o}s-R{\'e}nyi random graphs, our approach provides a more detailed characterization of the impact of dynamics, including how the epidemic unfolds over time. This simulation-based analysis is particularly valuable for understanding scenarios where analytical solutions are challenging or infeasible.\\

\noindent\textbf{Test of the local approximation.} In Figure \ref{fig:plots_2}, we test the accuracy of our main result from Theorem \ref{thm_dyn_conv_epid} by simulating the SIR epidemic on a dynamic Erd\H{o}s-R{\'e}nyi random graph (see Definition \ref{dyn_er_def}) and its time-marked union local limit. We assume that $D_I$ and $D_R$ follow exponential distributions with rates $I$ and $R$, respectively. We observe that the local limit provides an accurate approximation of the epidemic progression on the dynamic graph. This result is significant because it confirms that we can replace dynamic graphs with their dynamic local limit, frequently represented by dynamic trees, in further simulations, reducing computational complexity and cost.\\

\begin{figure}[h] 
    \centering
    \begin{subfigure}[b]{0.3\textwidth}    
         \centering 
         \includegraphics[width=\textwidth]{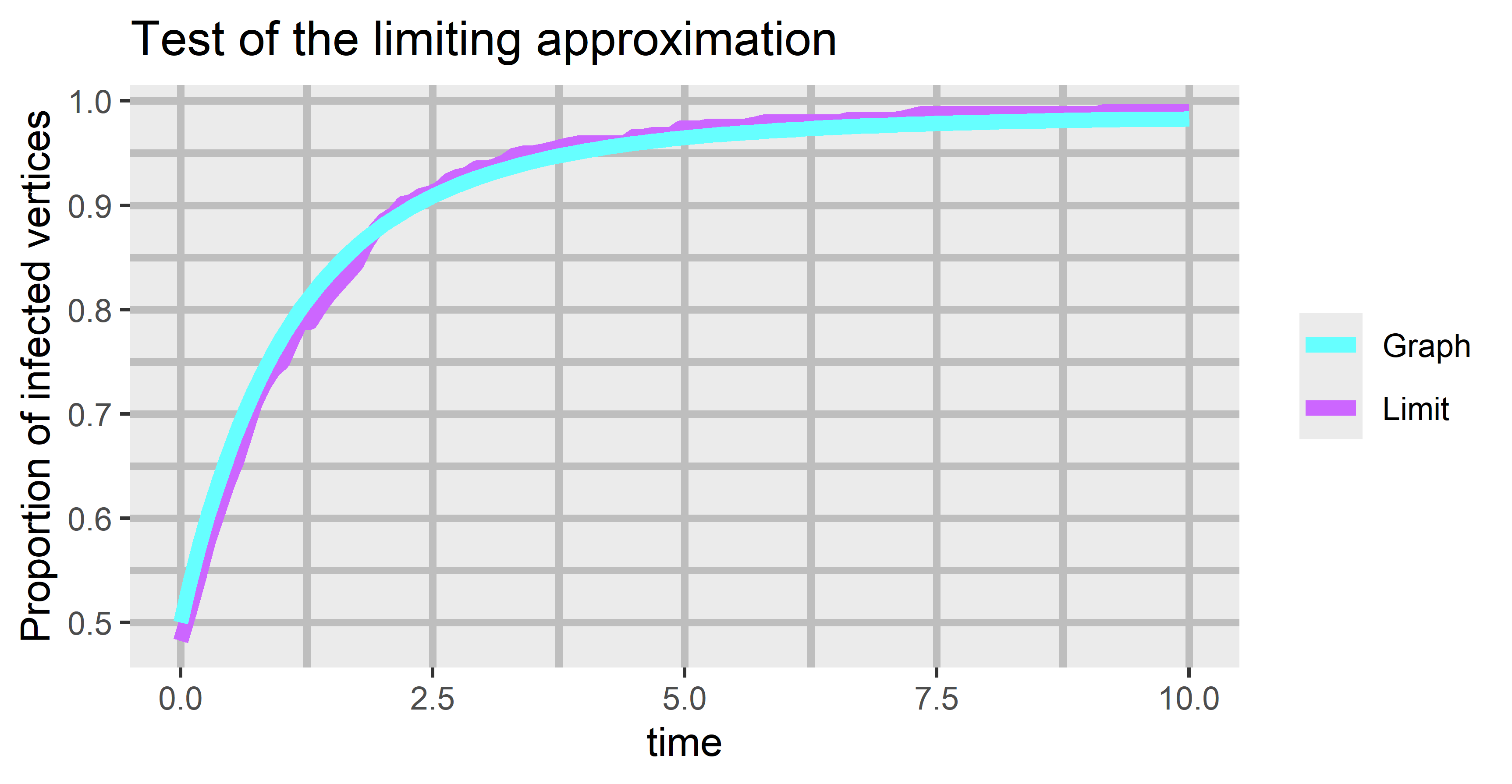}
             \caption[]%
             {{\small $\rho = 0.5, I=2, R=3$}}    
     \end{subfigure}
     \hfill
     \begin{subfigure}[b]{0.3\textwidth}   
         \centering 
         \includegraphics[width=\textwidth]{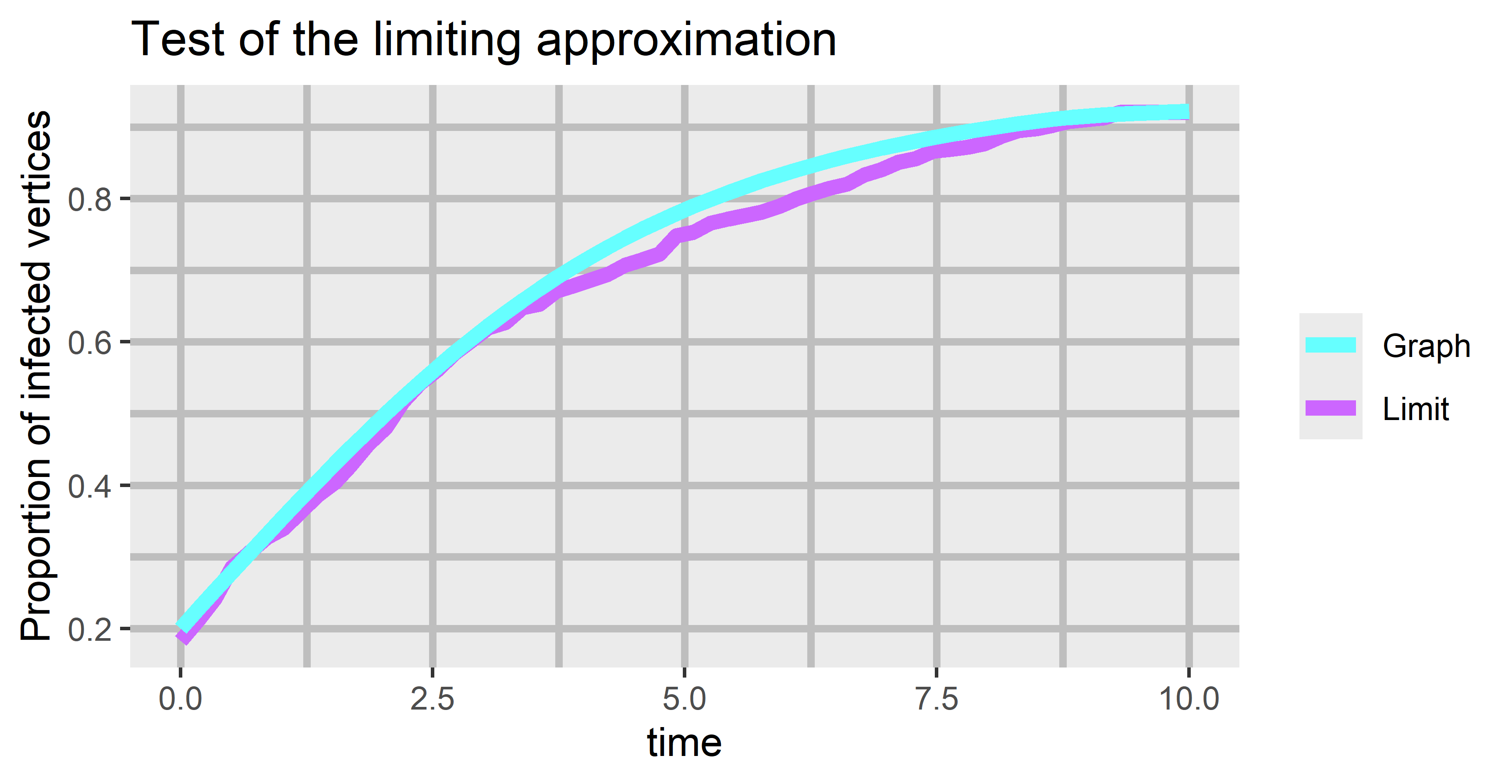}
             \caption[]%
             {{\small $\rho = 0.2, I=2, R=3$}}
     \end{subfigure}
     \hfill
     \begin{subfigure}[b]{0.3\textwidth}   
         \centering 
         \includegraphics[width=\textwidth]{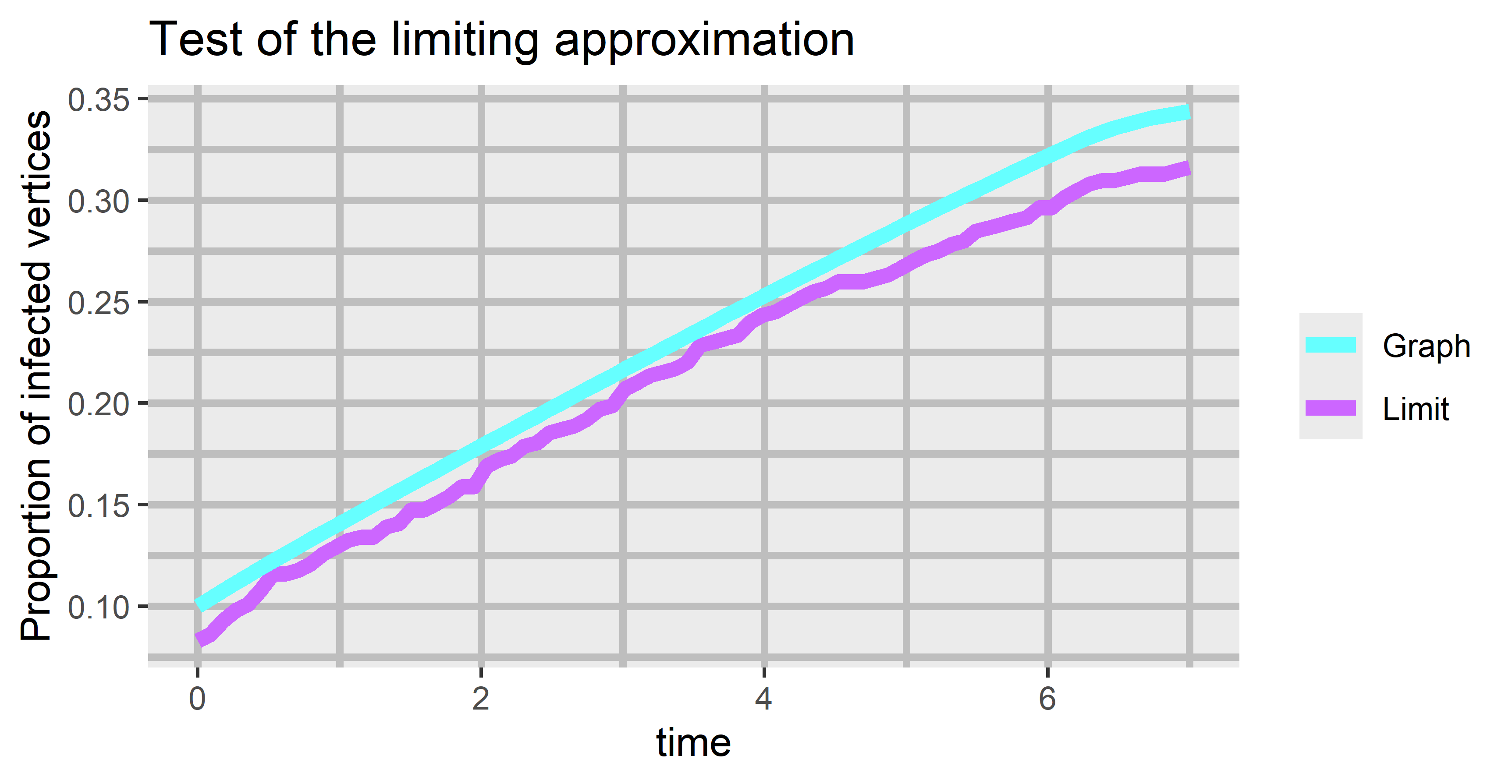}
             \caption[]%
             {{\small $\rho = 0.1, I=5, R=4$}}
     \end{subfigure}
     \caption{Comparison of the SIR epidemic time progression on the dynamic Erd\H{o}s-R{\'e}nyi random graph with $\gamma=3$ and $n=25 000$ and its dynamic local limit. The parameters $I$ and $R$ represent the rates of $D_I$ and $D_R$, respectively, which are assumed to follow exponential distributions. We have performed $500$ runs of simulations in each case, for various values of $\rho$, $I$ and $R$, explained under the plots.}
     \label{fig:plots_2}
\end{figure}

\noindent\textbf{Impact of dynamics on the epidemic curve.} We proceed to compare epidemic curves on static and dynamic random graphs. It is often hypothesised that graph dynamics accelerate the epidemic time progression (see Section \ref{sec_simulations_details}) and our simulation output strongly supports this conjecture. We present plots of predictions of SIR epidemic curves on static and dynamic graphs with the same stationary degree distribution for two static dynamic graph models: the static and dynamic Erd\H{o}s-R{\'e}nyi random graphs (Figure \ref{fig:plots_1}) and the static and dynamic random intersection graphs (Figure \ref{fig:plots_3}). We refer to Section \ref{sec_simulations_details} for more details on the models and their local limits.\\

\begin{figure}[h]
    \centering
    \begin{subfigure}[b]{0.3\textwidth}   
         \centering 
         \includegraphics[width=\textwidth]{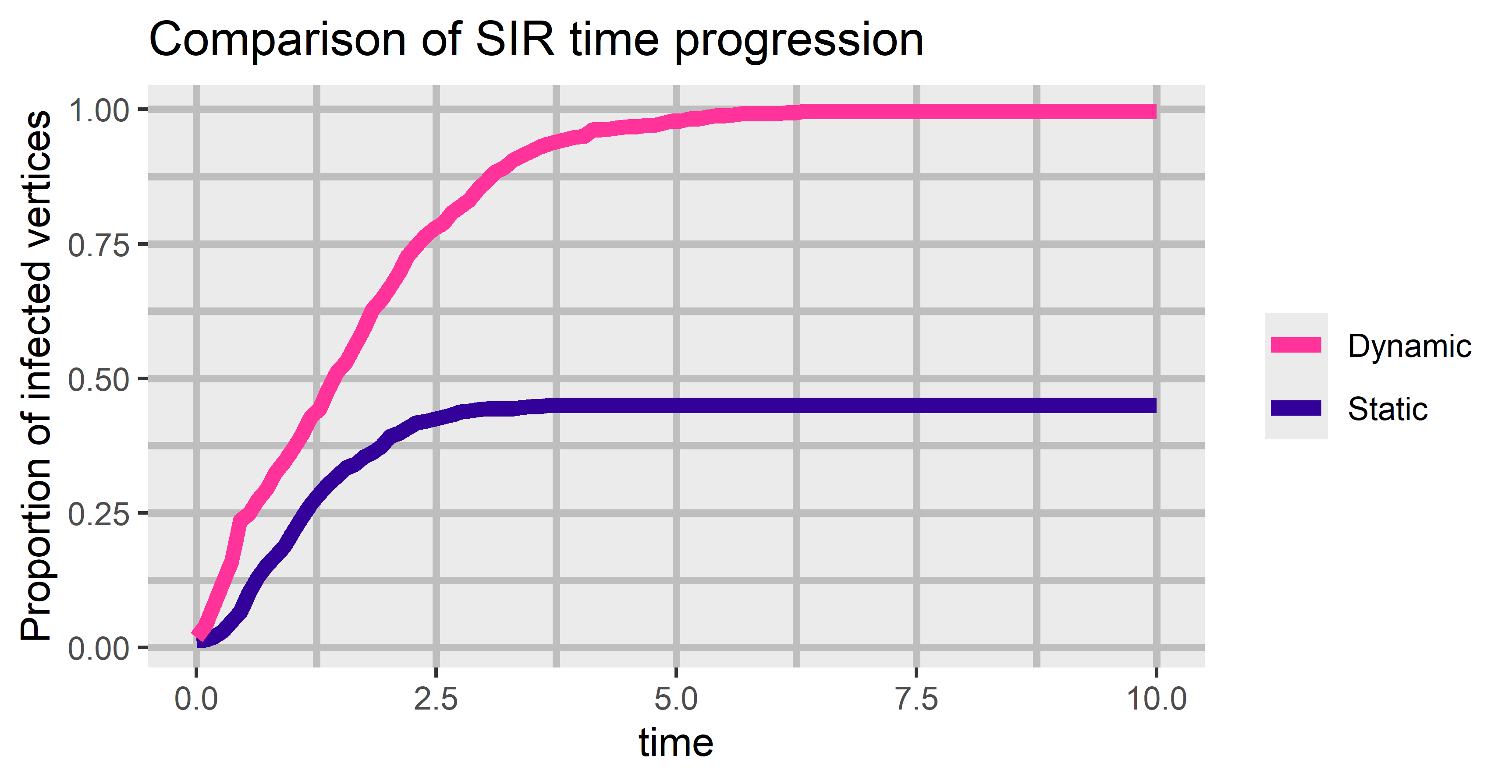}
             \caption[]%
             {{\small $I=0.5, R=3$}}    
     \end{subfigure}
     \hfill
     \begin{subfigure}[b]{0.3\textwidth}   
         \centering 
         \includegraphics[width=\textwidth]{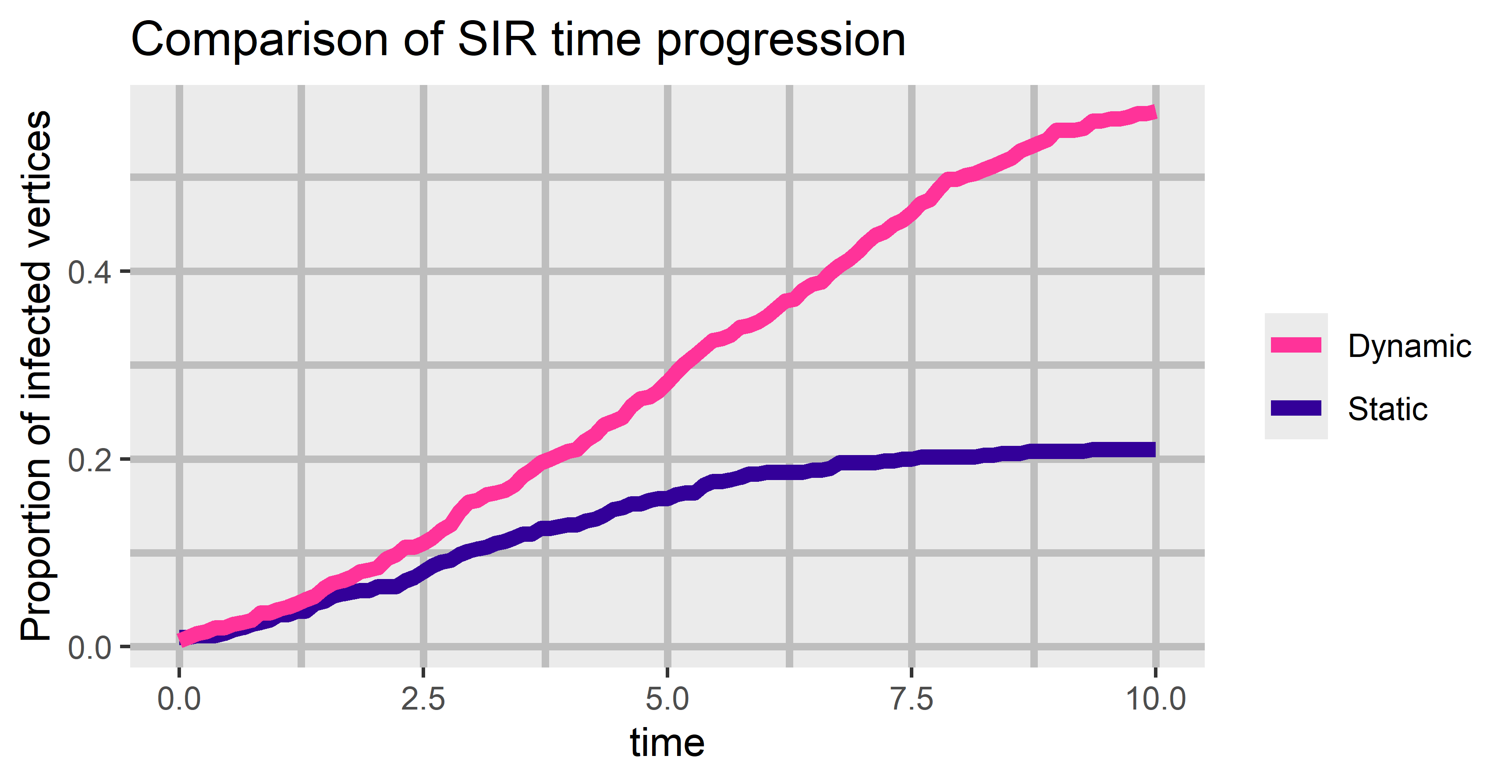}
             \caption[]%
             {{\small $I=2, R=3$}}
     \end{subfigure}
     \hfill
     \begin{subfigure}[b]{0.3\textwidth}   
         \centering 
         \includegraphics[width=\textwidth]{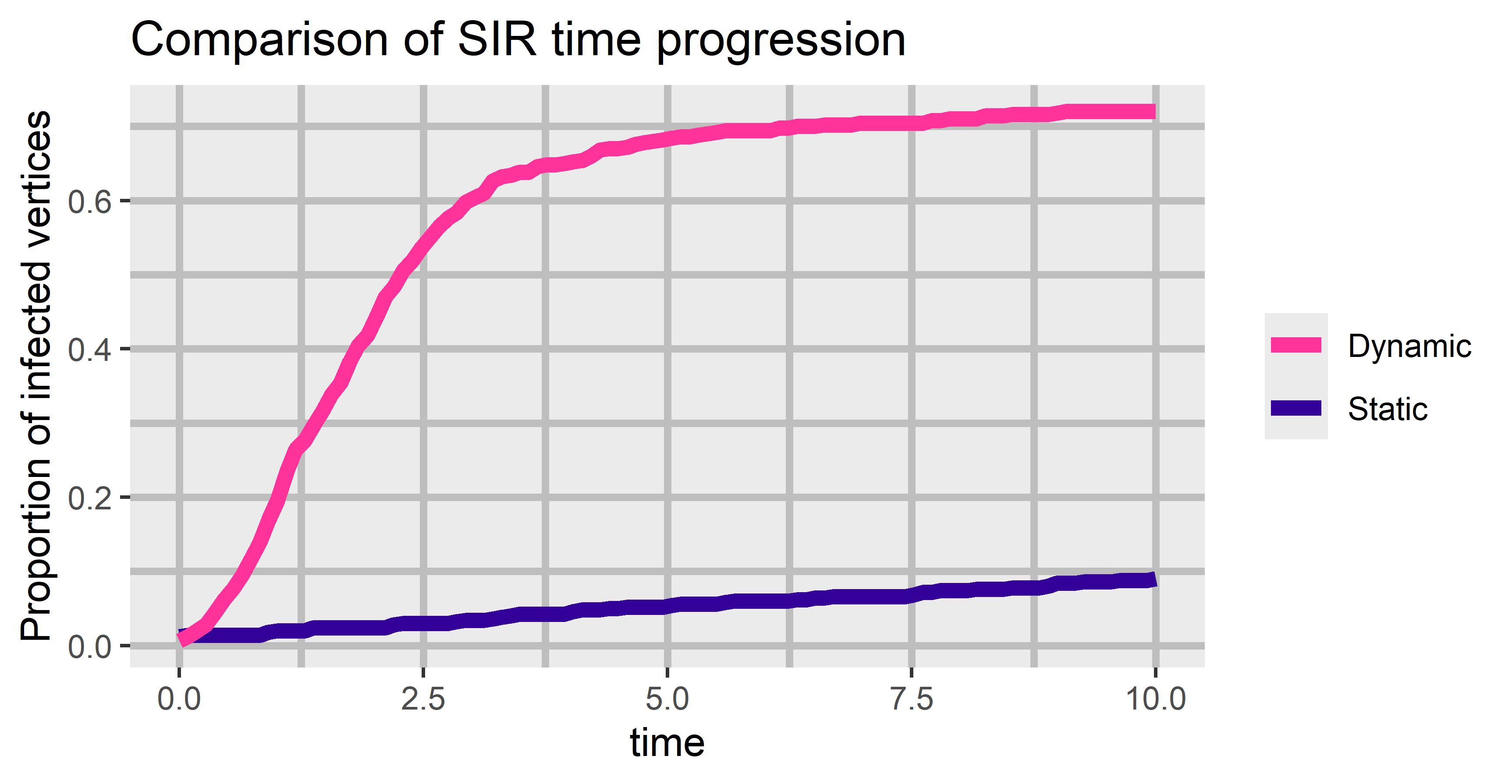}
             \caption[]%
             {{\small $I=4, R=5$}}    
     \end{subfigure}
     \caption{Comparison of the SIR epidemic time progression on static and dynamic Erd\H{o}s-R{\'e}nyi random graphs with $\gamma=3$ and $\rho=0.01$. The parameters $I$ and $R$ represent the rates of $D_I$ and $D_R$, respectively, which are assumed to follow exponential distributions. We have performed $500$ runs of simulations for both static and dynamic graphs for various values of $I$ and $R$, explained under the plots.}
     \label{fig:plots_1}
\end{figure}

\begin{figure}[h]
    \centering
    \begin{subfigure}[b]{0.3\textwidth}   
         \centering 
         \includegraphics[width=\textwidth]{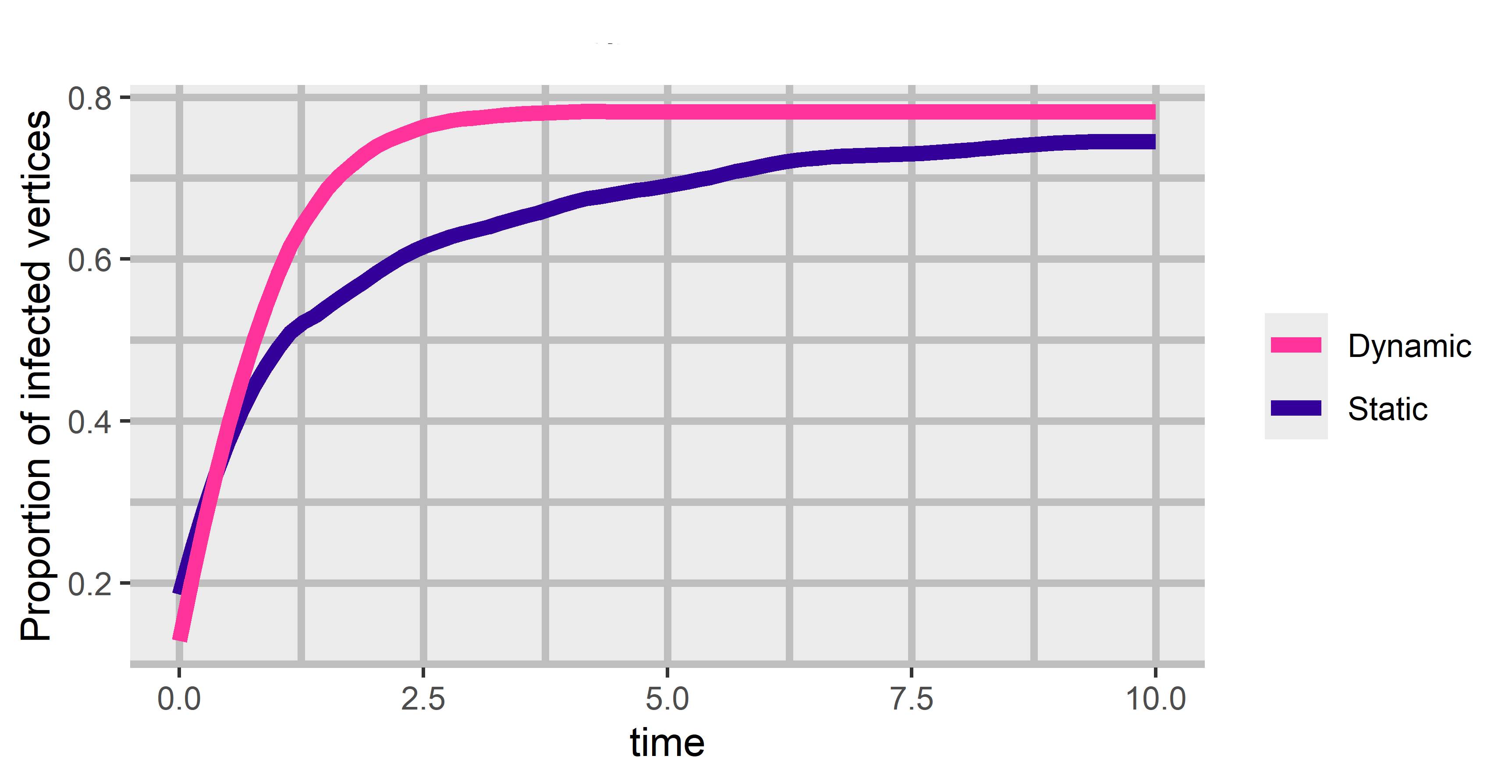}
             \caption[]%
             {{\small $I=0.5, R=3$}}    
     \end{subfigure}
     \hfill
     \begin{subfigure}[b]{0.3\textwidth}   
         \centering 
         \includegraphics[width=\textwidth]{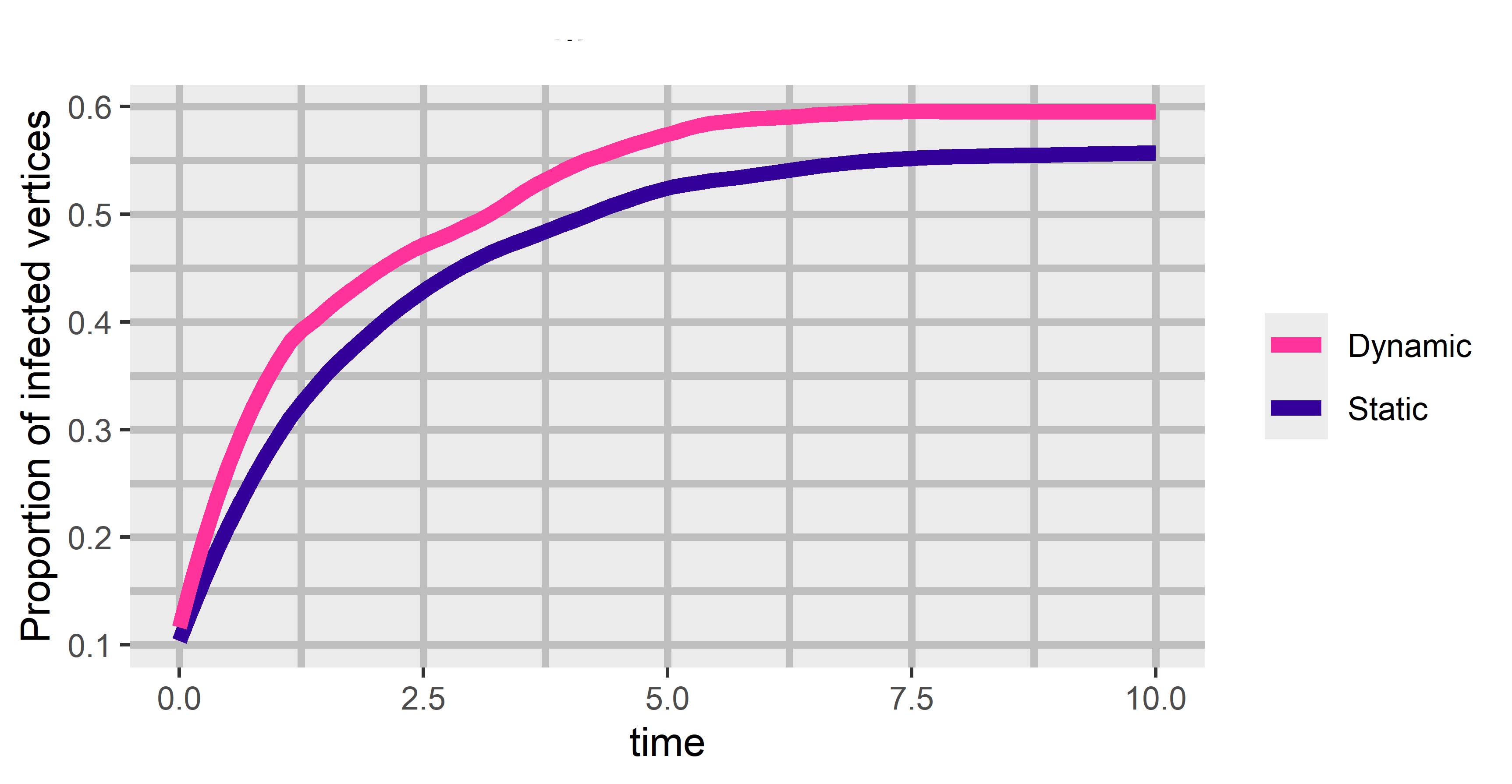}
             \caption[]%
             {{\small $I=2, R=3$}}
     \end{subfigure}
     \hfill
     \begin{subfigure}[b]{0.3\textwidth}   
         \centering 
         \includegraphics[width=\textwidth]{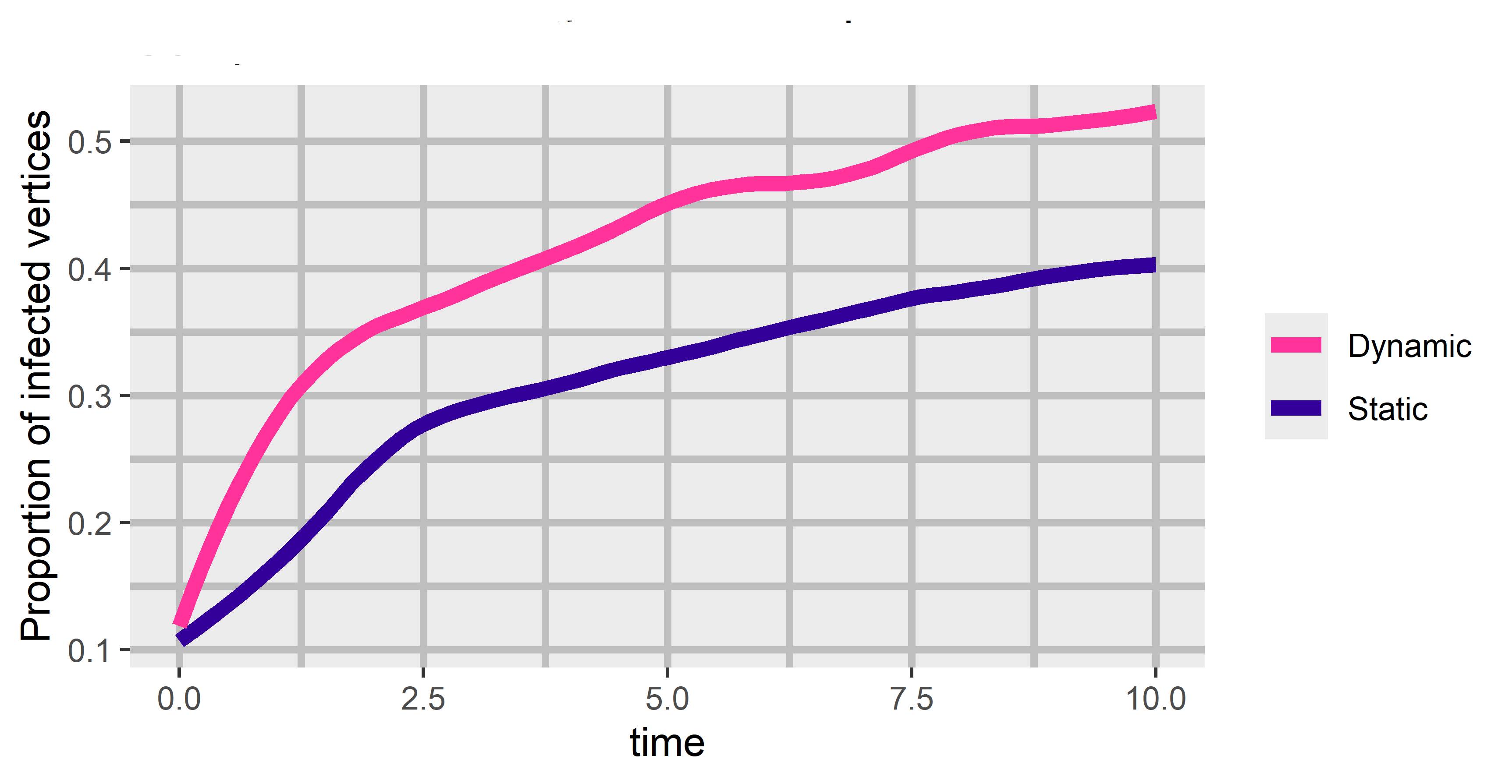}
             \caption[]%
             {{\small $I=5, R=4$}}    
     \end{subfigure}
     \caption{Comparison of the SIR epidemic curve on local limits of static and dynamic random intersection graphs with $\mathbf{E}[D_n]=2.052$ and $\rho=0.1$. The parameters $I$ and $R$ represent the rates of $D_I$ and $D_R$, respectively, which are assumed to follow exponential distributions. We have performed $500$ runs of simulations for both static and dynamic graphs for various values of $I$ and $R$, explained under the plots.}
     \label{fig:plots_3}
\end{figure}

\subsection{Discussion and conclusion}

In this work, we have extended the framework of epidemic modelling on static random graphs to dynamic random graphs, incorporating the concept of local time-marked union convergence. The main theorem establishes that, under the assumption of this stronger convergence, the SIR epidemic process on dynamic graphs converges to its counterpart on the time-marked local limit. This result generalizes existing findings and offers a flexible methodology for studying epidemics on dynamic networks without specifying a particular graph model.

Our theoretical exploration reveals the critical role of dynamic graph structures in modelling disease spread. While local convergence is sufficient to understand structural aspects in static graphs (see \cite{Alimohammadi2024}), dynamic settings require a more refined approach to capture the full complexity of epidemic processes. The introduction of the union graph, which aggregates all vertices connected to a root over time, provides a solution by offering control over indirect connections and temporal dynamics.

The simulation results further validate our theoretical contributions. By testing the dynamic Erd\H{o}s-R{\'e}nyi graph and its time-marked union local limit, we demonstrated that the latter provides a highly accurate approximation. This approximation enables us to perform efficient simulations by working with the local limit, reducing computational costs without sacrificing accuracy. These simulations open the door to addressing various intriguing questions about epidemics on dynamic graphs. One such question is the influence of graph dynamics on epidemic progression, which our results help to illuminate. Specifically, comparisons between static and dynamic graphs reveal how network dynamics shape the epidemic curve, providing valuable insights into the role of temporal changes in accelerating or moderating disease spread. On graphs with low clustering behavior, the dynamics tend to accelerate the spread of the disease. However, on graphs with higher clustering, the effect is more nuanced and depends heavily on parameters such as infection and recovery rates.

In summary, this work establishes a robust theoretical framework for analyzing epidemics on dynamic random graphs and demonstrates its practical applicability through simulations. The results not only highlight the benefits of incorporating dynamic graph models but also open up new avenues for exploring more complex and realistic epidemic processes on evolving networks.

The framework presented here can naturally be extended to account for vertex-level attributes that influence the transmission times of edges emanating from a vertex, while maintaining conditional independence of edge transmission times given these attributes. Additionally, the same techniques can accommodate models with more compartments, such as SEIR, as long as relevant independence assumptions hold and reinfection is not present. It would also be valuable to conduct simulations using infection and recovery time distributions beyond the exponential case, further exploring the flexibility of the framework.

For future work, it would be interesting to investigate dynamic graphs whose structure evolves in response to the epidemic process, such as models with preventive rewiring. Preventive rewiring, where susceptible individuals sever or redirect connections to infected individuals, has been shown to significantly influence epidemic outcomes. The authors of \cite{ball_britton_rewire} demonstrated that such strategies can either reduce the final epidemic size or, under certain conditions, unintentionally increase it. Exploring these dynamics further, particularly in more complex or realistic network settings, could provide valuable insights into the interplay between network adaptations and disease spread.

\section{Dynamic local convergence} \label{sec_locconv_theory}
In this section, we formally introduce the concepts that were discussed intuitively in Section \ref{sec_informal}. We also set the stage for the proof of the main result by introducing the key concepts. This involves a backward process, where we trace the epidemic's progression through past connections and incorporate the marks of the epidemic that encode crucial information like infection and recovery times.

\subsection{The metric space of rooted graphs}
Local convergence was introduced in \cite{Benjamini2001} and a few years later, independently, in \cite{Aldous2004}. It describes the resemblance of the neighbourhood of a vertex chosen uniformly at random to a certain limiting graph. To formalise this resemblance we introduce the notion of neighbourhoods and isomorphisms of graphs:
\begin{definition}[Rooted graph, isomorphism and $r$-neighbourhood] \label{rooted_graph_def}
\begin{enumerate} [(i)]
    \item We call a pair $(G, o)$ a \emph{rooted graph} if $G$ is a locally finite, connected graph and $o$ is a distinguished vertex of $G$. We denote the space of rooted graphs by $\mathcal{G}_{\star}$.
    \item Two graphs \(G_1 = (\mathcal{V}(G_1), \mathcal{E}(G_1))\) and \(G_2 = (\mathcal{V}(G_2), \mathcal{E}(G_2))\) are called \emph{isomorphic}, which we write as \(G_1 \simeq G_2\), when there exists a bijection \(\phi : \mathcal{V}(G_1) \to \mathcal{V}(G_2)\) such that \(\{u,v\} \in \mathcal{E}(G_1)\) precisely when \(\{\phi(u),\phi(v)\} \in \mathcal{E}(G_2)\).
    \item Similarly, two rooted graphs $(G_1, o_1)$, $(G_2, o_2)$ are \emph{rooted isomorphic} if there exists a graph-isomorphism between $G_1$ and $G_2$ that maps $o_1$ to $o_2$. We denote this isomorphism of rooted graphs by $(G_1,o_1)\simeq(G_2,o_2)$.
    \item For $r\in\mathbf{N}$, we define $B_r(G,o)$, the (closed) $r$-ball around $o$ in $G$ or $r$-neighbourhood of $o$ in $G$, as the subgraph of $G$ spanned by all vertices of graph distance at most $r$ from $o$. We think of $B_r(G,o)$ as a rooted graph with root $o$.
\end{enumerate}
\end{definition}
The notion of graph isomorphism enables us to define a metric on the space of rooted graphs:
\begin{definition}[Metric on rooted graphs] \label{def_metric_rooted}
Let $(G_1,o_1)$ and $(G_2,o_2)$ be two rooted connected graphs, and write $B_r(G_i,o_i)$ for the neighbourhood of vertex $o_i\in V(G_i)$. Let
\begin{align}
    R^{\star}=\sup\{r:B_r(G_1,o_1) \simeq B_r(G_2,o_2)\},
\end{align}
and define
\begin{align}
    \text{d}_{\mathcal{G}_{\star}}\big((G_1,o_1),(G_2,o_2) \big) = \frac{1}{R^{\star}+1}.
\end{align}
\end{definition}
It is important to note that when we refer to the space 
$\mathcal{G}_{\star}$ of rooted graphs under the metric $\text{d}_{\mathcal{G}_{\star}}$, this metric is defined on the set of equivalence classes in $\mathcal{G}_{\star}$, where the equivalence relation is determined by rooted isomorphisms. That is, two rooted graphs are considered equivalent if there exists an isomorphism between them that maps the root of one to the root of the other. By considering $\mathcal{G}_{\star}$ modulo this equivalence relation, the space $\left(\mathcal{G}_{\star}, \text{d}_{\mathcal{G}_{\star}}\right)$ becomes separable and complete and thus Polish (for a proof see \textup{\cite[Appendix A]{Hofstad2023}}).

\subsection{Local time-marked union convergence} \label{sec_loc_time_mark_un_conv}
To study epidemics on dynamic (random) graphs using local convergence, it is necessary to formalize a framework for discussing the local convergence of dynamic random graphs. Unlike the static case, dynamic graphs admit multiple notions of local convergence (refer to Section \ref{sec_techniqualities} for details). In this section, we introduce one such notion, termed local time-marked union convergence, which is essential for establishing the result in Theorem \ref{thm_dyn_conv_epid}. We start by formally introducing the union graph and its marks, which we previously described informally in Section \ref{sec_informal}.

\subsubsection{Marked union graph}

Our objective is to investigate the spread of an epidemic by means of a \emph{backward process} (see Section \ref{sec_backward_process}), which, for any vertex $v \in [n]$, traces its history backward in time to identify its infection time. Since tracking an epidemic on a dynamic graph $\left(G_n^s\right)_{s \in [0,T]}$ in this manner requires full knowledge of all dynamic trajectories up to the time of interest, we must gather all edges that became ON in $\left(G_n^s\right)_{s \in [0,T]}$ up to that point. To achieve this, we first introduce the concept of a \emph{union graph}, which consolidates all the edges that have been ON up to a given time threshold:

\begin{definition}[Union graph]
Let $\big(G_n^s\big)_{s\in[0,T]}=\big(([n], \mathcal{E}^s_n)\big)_{s\in[0,T]}$ be a dynamic graph process. We call the pair $([n], \mathcal{E}_n^{[0,T]})$, with $\mathcal{E}_n^{[0,T]} = \cup_{s\in[0,T]}\mathcal{E}^s_n$, the corresponding union graph up to time $T$, which we will denote $G_n^{[0,T]}$.
\end{definition}

Recall that in Section \ref{SIR_setup}, we set up the SIR epidemic so that the transmission time of an edge is initialized once the edge becomes ON for the first time, and the transmission must occur within one of the edge's ON periods to complete the disease spread. Therefore, collecting all edges is not enough to identify the infection times of vertices in a dynamic graph; we also need precise information about the activity times of these edges. For this reason, we introduce the notion of \emph{marks}. Marks in graphs allow us to record additional information about vertices and edges. We first provide a formal definition of marked graphs in general terms and then specify the time marks relevant to our setting:

\begin{definition}[Marked graphs] \label{def_marked_graphs}
Let $\mathcal{G}$ denote the set of all locally finite multi-graphs on a countable (finite or countably infinite) vertex set. A marked multi-graph is a multi-graph $G =(\mathcal{V}(G), \mathcal{E}(G))$, $G\in\mathcal{G}$, together with a set $\mathcal{M}(G)$ of marks taking values in a complete separable metric space $\Xi$, called the mark space, and containing the special symbol $\varnothing$ which is to be interpreted as “no mark”. $\mathcal{M}$ maps from $\mathcal{V}(G)$ and $\mathcal{E}(G)$ to $\Xi$. Images in $\Xi$ are called marks. Each edge is given two marks, one associated with (‘at’) each of its endpoints, in particular, $\mathcal{M}(v) \in \Xi$ for $v \in \mathcal{V}(G)$, and for $\{u,v\} \in \mathcal{E}(G)$, $\mathcal{M}(\{u,v\},u) \in \Xi$ and $\mathcal{M}(\{u,v\},v) \in \Xi$. We denote the set of graphs with marks from the mark space $\Xi$ by $\mathcal{G}(\Xi)$.
\end{definition}
We now move from the general notion of marked graphs to the specific context of dynamic graphs, where the marks represent the activity times of edges. In our setting, the temporal dynamics of the graph, including when edges become ON or OFF, is of central importance for understanding how the epidemic spreads. Thus, we define time marks on dynamic graphs to capture these time-varying behaviors. The following definition formalizes how these time marks are assigned to edges in a sequence of dynamic graphs:

\begin{definition}[Time marks of dynamic graphs] \label{def_time_marks_dyn_graphs}
Let $\big(G_n^s\big)_{s\in[0,T]}=\big(([n], \mathcal{E}^s_n)\big)_{s\in[0,T]}$ be a sequence of dynamic graphs. We specify the mark space $\Xi=\bigcup_{i=1}^{\infty} \left([0,T] \times [0,T]\right)^i$ and the mark set $\mathcal{M}\left(\big(G_n^s\big)_{s\in[0,T]}\right):\big(\mathcal{E}^s_n)\big)_{s\in[0,T]} \mapsto \bigcup_{i=1}^{\infty} \left(\{(s,t):0\leq s \leq t \leq T\}\right)^i$ determined by $\mathrm{ON}$ and $\mathrm{OFF}$ times of all edges that are $\mathrm{ON}$ at least once in $[0,T]$. Hence, for any edge $e\in \big(\mathcal{E}^s_n\big)_{s\in[0,T]}$,
\begin{align}
    \mathcal{M}(e) = \left(\left(\sigma^{e}_{i,\text{\rm{ON}}}, \sigma^{e}_{i,\text{\rm{OFF}}} \right)\right)_{i=1}^{N(e)},
\end{align}
with $N(e)$ denoting the total number of times edge $e$ switches $\mathrm{ON}$ in $[0,T]$ and $\sigma^{e}_{i,\text{\rm{ON}}}, \sigma^{e}_{i,\text{\rm{OFF}}}$ denoting the $i$-th time the edge $e$ switches $\mathrm{ON}$ and $\mathrm{OFF}$, respectively, in $[0,T]$. We fix $\sigma^{e}_{i,\text{\rm{ON}}}=0$ if $e\in \mathcal{E}^0_n$ (i.e., $e$ is $\mathrm{ON}$ at $0$) and, given $N(e)\geq1$, $\sigma^{e}_{i,\text{\rm{OFF}}}=T$ if $e$ switches $\mathrm{OFF}$ for the $i$-th time after $T$.
\end{definition}
We now formalise the concept of a union graph corresponding to some dynamic graph, equipped with time marks given by the edge activity in this dynamic graph:
\begin{definition}[Time-marked union graph] \label{def_time_marked_union_graph}
Let $\big(G_n^s\big)_{s\in[0,T]}=\big(([n], \mathcal{E}^s_n)\big)_{s\in[0,T]}$ be a sequence of dynamic random graphs. We define the time-marked union graph given by $\big(G_n^s\big)_{s\in[0,T]}$ as the pair
\begin{align*}
    \left(G_n^{[0,T]},\left(\left(\sigeon{e}{i},\sigeoff{e}{i}\right)_{i=1}^{N(e)}\right)_{e:\text{$e$ $\rm{ON}$ in $[0,T]$}}\right),
\end{align*}
where $G_n^{[0,T]}$ is the corresponding union graph and $\left(\left(\sigeon{e}{i},\sigeoff{e}{i}\right)_{i=1}^{N(e)}\right)_{e:\text{$e$ $\rm{ON}$ in $[0,T]$}}$ are the time marks of $\big(G_n^s\big)_{s\in[0,T]}$ as introduced in Definition \ref{def_time_marks_dyn_graphs}. To simplify the notation, in the remainder of this paper we denote this pair by $\markedunion$. Further, denote the rooted time-marked union graph determined by $\big(G_n^s\big)_{s\in[0,T]}$ and a uniformly chosen vertex $o_n$ by 
\begin{align*}
    \left(G_n^{[0,T]},o_n,\left(\left(\sigeon{e}{i},\sigeoff{e}{i}\right)_{i=1}^{N(e)}\right)_{e:\text{$e$ $\rm{ON}$ in $[0,T]$}}\right).
\end{align*}
Again in order to simplify the notation, in the remainder of this paper we denote the above by $\markedunionroot$.
\end{definition}
In the remainder of this paper, we will often condition on the dynamic graph and utilize the fact that this conditioning allows us to determine the time-marked union graph corresponding to the dynamic graph. The comprehensive information contained in the time-marked union graph enables us to effectively track the infection times of vertices in the original dynamic graph through a backward process. In Section \ref{sec_backward_process}, we explain how to perform this backward epidemic process on a time-marked union graph.

\subsubsection{Local time-marked union convergence}
We now proceed to define local time-marked union convergence. To establish this, we first generalize Definitions \ref{rooted_graph_def} and \ref{def_metric_rooted} to the setting of rooted marked graphs (recall Definition \ref{def_marked_graphs}):

\begin{definition}[Rooted marked graph and $r$-neighbourhood]
\begin{enumerate}[(i)]
    \item We choose a vertex $o$ in a marked graph $(G,\mathcal{M}(G))$ to be distinguished as the root. We denote the rooted marked graph by $(G,o,\mathcal{M}(G))$.\\
    We also denote the set of rooted marked graphs by $\mathcal{G}_{\star}(\Xi)$. A random rooted marked graph is a random variable taking values in $\mathcal{G}_{\star}(\Xi)$.
    \item The (closed) ball $B_r(G,o,\mathcal{M}(G))$ can be defined analogously to the unmarked graph ball (Definition \ref{rooted_graph_def} (iii)), by restricting the mark function to the subgraph as well.
\end{enumerate}
\end{definition}
Having formally defined the concept of a rooted marked graph and its $r$-neighbourhood, we now introduce a metric on the space of such graphs. This metric is crucial for formalizing the notion of convergence in this setting:

\begin{definition}[Metric on rooted marked graphs with continuous marks] \label{metric_marked_graphs}
Let $\text{d}_{\Xi}$ be a metric on the space of marks $\Xi$. Let
\begin{align}
    R^{\star}= \sup\{&r:B_r(G_1,o_1) \simeq B_r(G_2,o_2),\hspace{0.2cm} \text{and there exists $\phi$ such that}\\
    &\text{d}_{\Xi}((\mathcal{M}_1(i),\mathcal{M}_2(\phi(i))) \leq 1/r \hspace{0.15cm} \forall i\in V(B_r(G_1,o_1)),\nonumber\\
    &\text{d}_{\Xi}(\mathcal{M}_1((i,j)),\mathcal{M}_2(\phi(i,j))) \leq 1/r \hspace{0.15cm} \forall \{i,j\}\in E(B_r(G_1,o_1)) \nonumber
    \},
\end{align}
with $\phi:V(B_r(G_1,o_1))\longrightarrow V(B_r(G_2,o_2))$ running over all rooted isomorphisms between $B_r(G_1,o_1)$ and $B_r(G_2,o_2)$ that map $o_1$ to $o_2$. Then define
\begin{align}
    \text{d}_{\mathcal{G}_{\star}}\big((G_1,o_1,\mathcal{M}(G_1)),(G_2,o_2,\mathcal{M}(G_2)) \big) = \frac{1}{R^{\star}+1}.
\end{align}
This turns $\mathcal{G}_{\star}(\Xi)$ into a Polish space, i.e., a complete, separable metric space.
\end{definition}

In the following we also precisely specify a metric on time marks of dynamic graphs, consistently with Definition \ref{metric_marked_graphs}. This metric is crucial for formalizing the notion of local time-marked union convergence:

\begin{definition}[Metric on the time marks of dynamic graphs] \label{def_metric_marks_dyn_graphs}
Let $\big(G_n^s\big)_{s\in[0,T]}=\big(([n], \mathcal{E}^s_n)\big)_{s\in[0,T]}$ and $\big(\tilde{G}_{\tilde{n}}^s\big)_{s\in[0,\tilde{T}]} = \big(([\tilde{n}], \tilde{\mathcal{E}}^s_{\tilde{n}})\big)_{s\in[0,\tilde{T}]}$ be two distinct dynamic graphs. Let $\mathcal{M}$ and $\tilde{\mathcal{M}}$ denote the time mark sets of $\big(G_n^s\big)_{s\in[0,T]}$ and $\big(\tilde{G}_{\tilde{n}}^s\big)_{s\in[0,\tilde{T}]}$ respectively, as defined in Definition \ref{def_time_marks_dyn_graphs}. Then, for two edges $e\in\bigcup_{s\in[0,T]}\big(\mathcal{E}^s_n)\big), \Tilde{e}\in\bigcup_{s\in[0,T]}\tilde{\mathcal{E}}^s_{\tilde{n}}$, we define the distance between their time mark sets $\mathcal{M}(e)$ and $\tilde{\mathcal{M}}(\Tilde{e})$ as
\begin{align} \label{metric_on_marks}
    \text{d}_{\Xi}\left(\mathcal{M}(e),\tilde{\mathcal{M}}(\Tilde{e})\right) = |N(e)-N(\Tilde{e})|+ \sum_{i=1}^{\min\{N(e),N(\Tilde{e})\}} \left(|\sigma^{e}_{i,\text{\rm{ON}}}-\sigma^{\Tilde{e}}_{i,\text{\rm{ON}}}|+ |\sigma^{e}_{i,\text{\rm{OFF}}}-\sigma^{\Tilde{e}}_{i,\text{\rm{OFF}}}| \right).
\end{align}
\end{definition}
The above definition provides a metric structure for rooted time-marked union graphs, enabling us to rigorously define their convergence. This notion adapts the well-established concept of local convergence for marked graphs to the specific setting of union graphs with time marks as introduced above, given by dynamic graphs. Building on this framework, we now present the concept of local time-marked union convergence,  which plays a fundamental role in our results:

\begin{definition}[Local time-marked union convergence in probability] \label{def_mark_union_conv_probab} Let $\big(G_n^s\big)_{s\in[0,T]}=\big(([n], \mathcal{E}^s_n)\big)_{s\in[0,T]}$ be a sequence of dynamic random graphs with fixed size $n\to \infty$, and let $o_n$ be a uniformly chosen vertex. Let $\big(G^s,o\big)_{s\in[0,T]}$ denote a random element (with arbitrary distribution) of the set of rooted dynamic graphs. Denote the rooted time-marked union graph $\left(G^{[0,T]},o,\left(\left(t^e_{\sss i,\mathrm{ON}},t^e_{\sss i,\mathrm{OFF}}\right)_{i=1}^{N(e)}\right)_e\right)$ by $\marklim$. Then, the rooted time-marked union graph $\markedunionroot$ defined in Definition \ref{def_time_marked_union_graph} converges locally in probability to $\marklim$, if for any fixed rooted marked graph $\left(H,o,\left(\left(\bar{t}^e_{\sss i,\mathrm{ON}},\bar{t}^e_{\sss i,\mathrm{OFF}} \right)_{i=1}^{N(e)}\right)_e\right)$ denoted by $\hmarkedroot$ and $r\in\mathbf{N}$,

\begin{align} \label{cond_def_fdd_strong_conv_probab}
\mathbf{P}&\bigg(\text{d}_{\mathcal{G}_{\star}}\bigg(\markedunionroot, \hmarkedroot \bigg) \leq \frac{1}{r+1} \mid \left(G_n^s \right)_{s\in[0,T]}\bigg)\nonumber\\
    &:=\frac{1}{n} \sum_{v\in[n]} \mathds{1}_{\big\{\text{d}_{\mathcal{G}_{\star}}\big(\markedunionv,\hmarkedroot\big) \leq \frac{1}{r+1}\big\}}\\
    &\stackrel{\mathbf{P}}{\longrightarrow} \hspace{0.1cm} \mathbf{P}\bigg(\text{d}_{\mathcal{G}_{\star}}\bigg(\marklim, \hmarkedroot\bigg) \leq \frac{1}{r+1}\bigg). \nonumber
\end{align}
\end{definition}

\begin{remark} [Static marked graphs]
This section primarily focused on extending the framework of marked local convergence to the setting of dynamic graphs, where marks arise from dynamic edge activity. However, Definitions \ref{def_marked_graphs} and \ref{metric_marked_graphs} are formulated in broad generality and are equally applicable to static graphs, for which they were originally defined. In the static context, marks naturally do not relate to dynamic edge activity but can encode various types of information, such as membership in partitions (e.g., in bipartite graphs; see Section \ref{sec_drig_stat}). As noted prior to Definition \ref{def_mark_union_conv_probab}, the notion of marked local convergence is already well-established for static marked graphs (see \cite[Section 2.3.5]{Hofstad2023}) and is formulated in a manner analogous to Definition \ref{def_mark_union_conv_probab}. Thus, while marked local convergence serves as a standard tool for analyzing static graphs, this work demonstrates how to adapt the concept to the dynamic setting, enabling its application to the study of dynamic graphs.
\end{remark}

\subsection{Relationship to dynamic local convergence} \label{sec_techniqualities}
As previously noted, in the context of dynamic random graphs, it is possible to distinguish between various notions of local convergence. The approach discussed in Section \ref{sec_loc_time_mark_un_conv} involves aggregating all edges that were active within a specified time frame and examining the convergence of this marked union graph, where edges are marked by the times of their activity. This methodology effectively transforms a dynamic graph into a marked static one, thereby enabling the application of the well-established framework of marked local convergence to analyze dynamic structures.

Alternatively, one can consider local convergence of dynamic random graphs from a stochastic process perspective, wherein both the original graph and its limit remain inherently dynamic structures. This approach to dynamic local convergence was explored in \cite{Milewska2023}, where a dynamic random graph process was viewed as a dynamic random graph process as a sequence of functions in time: a collection of functions from $\mathbf{R}_+$ into the space of rooted graphs. Next, we took advantage of the fact that the space of rooted graphs is separable and complete under the metric $\text{d}_{\mathcal{G}_{\star}}$ (for a proof, see \textup{\cite[Appendix A.3]{Hofstad2023}}), and convergence of random processes that map onto a separable and complete space has been thoroughly investigated. We now briefly outline the background theory underlying this form of convergence.

Fix a metric space $(S,\text{d}_\text{S})$ that is separable and complete under the metric $\text{d}_\text{S}$. Consider random processes with sample paths in $D(\mathbf{R}_+,S)$ - the space of c{\`a}dl{\`a}g functions $f:\mathbf{R}_+ \longrightarrow S$ and equip $D(\mathbf{R}_+,S)$ with the Skorokhod $J_1$ topology. Classically, weak convergence of probability measures on $(D(\mathbf{R}_+,S),\mathcal{D})$, where $\mathcal{D}$ denotes the Borel $\sigma$-field generated by $D(\mathbf{R}_+,S)$, is defined as convergence of integrals of all bounded and continuous functions. Alternatively, weak convergence of probability measures on $(D(\mathbf{R}_+,S),\mathcal{D})$ can also be established through the convergence of finite-dimensional distributions together with tightness (see \textup{\cite[Theorem 13.1]{Billingsley2013}}). This, in turn, implies convergence of random processes with sample paths in $D(\mathbf{R}_+,S)$. Results obtained by Prohorov and Arzel{\`a}-Ascoli facilitate the verification of tightness: By Prohorov's theorem, in a separable and complete metric space, tightness and relative compactness in distribution are equivalent and the Arzel{\`a}-Ascoli theorem provides a characterization of relative compactness in the form of a bound on an appropriate modulus of continuity (see \textup{\cite[Theorem 13.2]{Billingsley2013}}). However, the mentioned modulus of continuity might not be easy to work with. For that reason, an alternative criterion for tightness is given in \textup{\cite[Theorem 13.3]{Billingsley2013}}. 
\begin{remark}
The cited \textup{\cite[Chapter 13]{Billingsley2013}} actually investigates convergence of processes with sample paths from $[0,1]$ to $\mathbf{R}$. However, \textup{\cite[Chapter 12]{Billingsley2013}} introduces preliminaries for \textup{\cite[Chapter 13]{Billingsley2013}}, where the author remarks that the provided theory naturally extends to processes taking values in other metric spaces, including a general separable and complete metric space. The domain of the sample paths can also be easily rescaled to $[0,T]$ for some $T\in\mathbf{R}_+$. Alternatively, \textup{\cite[Chapter 16]{kallenberg2002foundations}} treats processes with sample paths in $D(\mathbf{R}_+,S)$ directly and states analogous results (see \textup{\cite[Theorems 16.10, 16.12]{kallenberg2002foundations}}).
\end{remark}

When looking at $\big(G_n^s\big)_{s\in[0,T]}$ as a process in time, we choose a random root $o_n$ only \emph{once} and then we investigate the evolution of its neighbourhood in time. Applying the theory outlined above (\textup{\cite[Theorem 13.3]{Billingsley2013}}), we translate the convergence criteria into the context of dynamic random graphs. For the reader's convenience, we present this explicitly in the following definition:

\begin{definition}[Dynamic local weak convergence] \label{def_dyn_weak_conv}
Let $\big(G_n^s\big)_{s\in[0,T]}=\big(([n], \mathcal{E}^s_n)\big)_{s\in[0,T]}$ be a sequence of dynamic random graphs of a fixed size $n\to \infty$, and let $o_n$ be a vertex chosen uniformly at random from $[n]$. Let $\big(G^s,o\big)_{s\in[0,T]}$ denote a random element (with arbitrary distribution) of the set of rooted dynamic graphs, which we call a dynamic random rooted graph. Then, $\big(G_n^s\big)_{s\in[0,T]}$ \textbf{converges dynamically locally weakly} to $\big(G^s,o\big)_{s\in[0,T]}$ (i.e., converges weakly in $D(\mathbf{R}_+,S)$ in the Skorokhod $J_1$ topology, with $S = (\mathcal{G}_{\star},\text{d}_{\mathcal{G}_{\star}})$), if the following conditions hold:

\begin{enumerate}
\item The finite-dimensional distributions of $\big(G_n^s,o_n\big)_{s\in[0,T]}$ converge: 
    $\big(G_n^s,o_n\big)_{s\in[0,T]} \stackrel{fdd}{\longrightarrow} \big(G^s,o\big)_{s\in[0,T]}$, i.e., for all $s_1\leq s_2\leq\cdots\leq s_k\in[0,T]$ and for any fixed rooted graph sequence $\left((H_{1},o),\ldots,(H_{k},o)\right)$ and $r\in\mathbf{N}$,
    \begin{align} \label{cond_def_fdd_conv}
    \mathbf{P}\bigg(\forall j\in[k]: B_r\big(G_n^{s_j},o_n\big)\simeq (H_j,o)\bigg)&= \frac{1}{n}\mathbf{E}\Bigg[\sum_{i\in[n]}\mathds{1}_{\big\{\forall j\in[k]: B_r\big(G_n^{s_j},i\big)\simeq (H_j,o)\big\}} \Bigg]\nonumber \\
    &\longrightarrow \mathbf{P}\bigg(\forall j\in[k]: B_r\big(G^{s_j},o\big)\simeq (H_j,o)\bigg). \nonumber
    \end{align} 
    \item For every $T>0$, as $\delta \searrow 0$, \label{cond_tight_of_the_limit_dyn_weak_conv}
    \begin{align} \text{d}_{\mathcal{G}_{\star}}\bigg(\big(G(T),o\big),\big(G(T-\delta),o\big)\bigg) \stackrel{\mathbf{P}}{\longrightarrow} 0.
    \end{align}
    
    \item For all $\varepsilon, \eta>0$ there exists $n_0\geq1$ and $\delta>0$ such that for all $n\geq n_0$,
    \begin{align} \label{def_cond_tight_for_dyn_loc_lim_content}
    \mathbf{P}&\bigg(\sup_{(s,s_1,s_2)\in\mathscr{S}_{\delta}} \min\bigg[\text{d}_{\mathcal{G}_{\star}}\bigg(\big(G_n^{s_1},o_n\big),\big(G_n^s,o_n\big)\bigg), \text{d}_{\mathcal{G}_{\star}}\bigg(\big(G_n^s,o_n\big),\big(G_n^{s_2},o_n\big)\bigg)\bigg]>\varepsilon \bigg)\leq \eta,
    \end{align}
with $\mathscr{S}_{\delta} = \{(s,s_1,s_2):s\in[s_1,s_2], |s_2-s_1|\leq \delta\}$. Note that the above is equivalent to 
\begin{align} \label{def_dyn_loc_lim_tight_tobebounded}
    \mathbf{P}\bigg(\exists s\in\mathscr{T}_{\delta}: &B_{1/\varepsilon}\big(G_n^{s_1},o_n\big) \not\simeq  B_{1/\varepsilon}\big(G_n^s,o_n\big), B_{1/\varepsilon}\big(G_n^s,o_n\big) \not\simeq  B_{1/\varepsilon}\big(G_n^{s_2},o_n\big) \bigg)\leq \eta,
\end{align}
with $\mathscr{T}_{\delta} = \{(s_1,s_2): |s_2-s_1|\leq \delta\}$.
\end{enumerate}
\end{definition}

In Section \ref{sec_simulations_details}, we give examples of dynamic random graphs that converge locally in a time-marked union sense. However, the graphs presented there will also converge dynamically locally, in the sense of Definition \ref{def_dyn_weak_conv}. See our previous work \cite{Milewska2023}, where we studied the dynamic local convergence of the dynamic random intersection graph introduced in Definition \ref{def_dyn_drig}. This was achieved by first examining local marked union convergence. The dynamic local convergence of the dynamic Erd\H{o}s-R{\'e}nyi random graph from Definition \ref{dyn_er_def} can be derived as a consequence of the results presented in \cite{Milewska2023}. Another notable example is provided in \cite{dynweaklimit2023}, where the authors also introduce the concept of dynamic local convergence and apply it to a dynamic inhomogeneous random graph. However, their notion of dynamic local convergence differs from our classical stochastic process approach.

Dynamic local weak convergence can further be extended to dynamic local convergence in probability, analogous to the static setting. This extension is facilitated by the fact that, for the tightness conditions, convergence in distribution naturally implies convergence in probability. For a comprehensive discussion of this result, some consequences of dynamic local convergence, and a more detailed overview of the concept, we refer the reader to \cite{Milewska2025}.\\

We now turn to an intriguing observation made during our study of epidemics on dynamic random graphs. As we have mentioned, the concept of weak convergence is classically defined via the convergence of expectations of continuous and bounded functionals, a characterization formalized by Prokhorov's theorem. This notion extends to the local convergence of static graphs, where local weak convergence can be expressed as
\begin{align*}
\mathbf{E}[h(G_n,o_n)] \longrightarrow \mathbf{E}[h(G,o)],
\end{align*}
for continuous and bounded functionals $h:\mathcal{G}_{\star}\rightarrow\mathbf{R}$. For dynamic graphs, a similar concept holds, but special care is needed when defining the domain of such functionals.\\
In the case of static graphs, the relationship between the root $o_n$ and the rooted graph $(G_n,o_n)$ is straightforward: everything that affects $o_n$ lies within its connected component, which is fixed (even if randomly determined) once and for all.

However, for dynamic graphs, the situation is more subtle. Due to changes in edge statuses over time, a vertex may encounter many indirect connections. For instance, a neighbor of $o_n$ might have been connected to another vertex before connecting to $o_n$ and even if that edge switches $\rm{OFF}$ before $o_n$ becomes connected, this extra vertex could still influence $o_n$ (for example, in an epidemic process), despite never being part of its connected component at any specific time. This dynamic nature introduces complexity in specifying how to take functionals of dynamic graphs. As a result, the notion of dynamic local convergence, as introduced in \cite{Milewska2023}, translates to
\[
\mathbf{E}\left[h\left(\left(G^s_n, o_n\right)_{s\in[0,T]}\right)\right] \longrightarrow \mathbf{E}\left[h\left(\left(G^s,o\right)_{s\in[0,T]}\right)\right],
\]
for any continuous and bounded functional $h: D(\mathbf{R}_+, \mathcal{G}_\star) \to \mathbf{R}$, where $D(\mathbf{R}_+, \mathcal{G}_\star)$ is equipped with the Skorokhod $J_1$ topology. Here, the functional on the left-hand side depends only on the connected components of $o_n$ at each time $s\in[0,T]$. In contrast, the stronger notion of local time-marked union convergence stated in Definition \ref{def_mark_union_conv_probab}, requires convergence for all bounded and continuous functions that depend on the entire trajectory of changes in the union neighbourhood of $o_n$, as recorded by its time-marked union neighbourhood. This stronger form captures the full extent of the dynamic interactions, including all indirect influences, and is therefore essential to discuss epidemics on dynamic random graphs (see the proof of Proposition \ref{lem_conv_loc_approx}).

\begin{remark}
We remark that in common scenarios where the root $o_n$ is chosen uniformly at random from a large graph, such as the dynamic Erd\H{o}s-R{\'e}nyi random graph treated in Section \ref{sec_simulations_details} or the dynamic random intersection graph introduced in \cite{Milewska2023}, the two notions of convergence will hold. To explain it heuristically, note that the convergence of the connected components of a typical vertex at any point in time, combined with the tightness condition (which implies that the number of vertices encountering a significant amount of changes up to time $T$ is negligible), suggests that the probability of seeing any vertex experiencing a large number of changes or having a particularly large degree at any point in time is small. Consequently, the cumulative union neighbourhoods of most vertices remain bounded with high probability. However, this equivalence does not hold for all dynamic rooted graphs that converge locally in the sense of Proposition \ref{def_dyn_weak_conv}. As a counterexample, consider a dynamic graph where the root has two neighbours during the time interval $[1,2]$ and none at other times. One of these neighbors is connected to $n-3$ other vertices during $[0,0.5]$. Then the union neighbourhood of the root does not converge as $n$ becomes large, while the dynamic snapshots at any time point still do.
\end{remark}

\subsection{Backward process on a time-marked union graph} \label{sec_backward_process}
Determining infection times in a dynamic graph is significantly more challenging than in a static graph, as we must account for the constant changes in connections between vertices. However, by leveraging the time-marked union graph (recall Definition \ref{def_time_marked_union_graph}) and the fact that, with a positive proportion of initially infected vertices, paths to these vertices can be efficiently traced, we argue that a backward process, conditioned on the dynamic graph, can be performed similarly to a static graph in analogous situations. Nevertheless, the algorithm is more involved since we need to continuously check the dynamic status of the edges.

Fix a dynamic graph $\left(G^s\right)_{s\in[0,T]}$ with a known dynamic history. To simplify the notation, in the remainder of this paper, let $\markedg$ denote the corresponding time-marked union graph recording all the edges present in $[0,T]$ together with their ON and OFF times from Definition \ref{def_time_marked_union_graph}, i.e.,
\begin{align*}
    \markedg = \left(G^{[0,T]}, \left(\left( \teon{e}{i}, \teoff{e}{i}\right)_{\sss  i=1}^{\sss N(e)}\right)_{\sss e: \text{$e$ \rm{ON} in $[0,T]$}} \right).
\end{align*}
We now give a detailed description of the backward process on such a time-marked union graph. In Section \ref{sec_all_proofs_prelim} we will show how to apply this process to random dynamic graphs by conditioning on their history up to a certain time point.

Given $v\in[n]$ and $t\in[0,T]$, we perform the backward epidemic process started from a vertex $v$ at a time point $t$. First, we determine the recovery time $R_v$ of the vertex $v$ from $D_R$. Next, for each vertex $u\in[n]$ that has been a neighbour of $v$ at any time up to $t$, we draw its recovery time $R_u$ from $D_R$, as well as the transmission time $C(\{u,v\})$ between $u$ and $v$ from $D_I$. We also record the time activity of the edge $\{u,v\}$ (which we know once again because we have assumed full knowledge about the graph's dynamic trajectory). Should $\{u,v\}$ be active more than once up to time $t$, we record all its activity windows. For each edge $e$, we check whether the transmission time $C(\{u,v\})$ is less than the recovery time $ R_u $. If this condition is not met, transmission between $ u $ and $ v $ is impossible, and the edge is marked with an infinite infection time. We continue this breadth-first search backward from $ v $, marking edges with potential infection times based on their transmission times and setting these to infinity when the transmission time exceeds the recovery time, until we run out of edges. In this way, we identify the potential paths of the epidemic and first filter out impossible connections.

In the second part of the algorithm, we determine the actual infection times along the paths identified in the first part i.e., the transmission times $C(e)$ associated with the edges are tested against the time dynamic recorded by edge marks $\left(\left(\teon{e}{i}, \teoff{e}{i} \right)^{\sss N(e)}_{\sss i=1}\right)_{\sss e}$. The infection can only spread during the active periods of these edges. This time dynamic significantly complicates the situation because each possible infection path from an initially infected vertex to $v$ must be individually verified for dynamic feasibility. Each edge on these paths may receive different infection times for distinct paths that it is part of. This occurs because the infection timeline depends on the specific sequence of edges and their active periods for each path.

To give a precise description of this process, denote the set of all feasible paths from initially infected vertices to vertex $v$ identified in Part 1 of the backward process by $\mathscr{P}_v$. By feasible we mean all paths not containing any edges with infinite infection transmission marks: $\mathscr{P}_v:= \{\mathcal{P}: \forall e\in\mathcal{P}, C(e)\neq\infty\}$. Denote the number of edges in $\mathcal{P}\in\mathscr{P}_v$ by $|\mathcal{P}|$ and note that, apart from $v$, $\mathcal{P}$ has exactly $|\mathcal{P}|$ vertices. Note that vertices can coincide on two distinct paths $\mathcal{P}, \mathcal{P}'\in\mathscr{P}_v$. However, their infection times on each of these paths can be different due to the unique time dynamics of each path. For that reason we introduce notation which records information about the specific path $\mathcal{P}\in\mathscr{P}_v$ under consideration. We denote the $j$-th vertex on the path $\mathcal{P}\in\mathscr{P}_v$, counting from the origin (an initially infected vertex) towards $v$ by $u_j^{\mathcal{P}}$, so that $\mathcal{P} := \{\{u_1^{\mathcal{P}},u_2^{\mathcal{P}}\},\ldots,\{u_{|\mathcal{P}|}^{\mathcal{P}},v \}\}$, with $u_1^{\mathcal{P}}\in\mathcal{I}(o)$. Further, we denote the infection time of a vertex $u_j^{\mathcal{P}}$ on path $\mathcal{P}\in\mathscr{P}_v$ by $I_{u_j^{\mathcal{P}}}$. We now describe how to identify these infection times. First, since all paths in $\mathscr{P}_v$ start with an initially infected vertex, we set
\begin{align}
    I_{u_1^{\mathcal{P}}} = 0, \hspace{0.2cm} \forall \hspace{0.05cm} \mathcal{P}\in\mathscr{P}_v.
\end{align}
Subsequently, for all $j\in[|\mathcal{P}|]$, we define
\begin{align}
    s^0_{u_j^{\mathcal{P}}} = I_{u_j^{\mathcal{P}}}, \hspace{1cm} \text{if} \hspace{0.5cm} I_{u_j^{\mathcal{P}}} \in \bigcup_{\sss i=1}^{\sss N\left(e_j^{\mathcal{P}}\right)} \left[\teon{}{i},\teoff{e_j^{\mathcal{P}}}{i}\right],
\end{align}
with $e_j^{\mathcal{P}} = \{u_j^{\mathcal{P}},u_{j+1}^{\mathcal{P}}\}$, and
\begin{align}
    s^0_{u_j^{\mathcal{P}}} =
     \min_{i}\{\teon{e_j^{\mathcal{P}}}{i}:
     \teon{e_j^{\mathcal{P}}}{i}>  I_{u_j^{\mathcal{P}}}\}, \hspace{1cm} \text{if} \hspace{0.5cm} I_{u_j^{\mathcal{P}}} \notin \bigcup_{\sss i=1}^{\sss N\left(e_j^{\mathcal{P}}\right)} \left[\teon{e_j^{\mathcal{P}}}{i},\teoff{e_j^{\mathcal{P}}}{i}\right].
\end{align}
The above aligns with the transmission dynamics of the SIR epidemic model outlined in Section \ref{SIR_setup}. Specifically, the initialization time of the transmission process depends on whether the edge $e_{j-1}^{\mathcal{P}}$ is active at the time the vertex $u_{j-1}^{\mathcal{P}}$ gets infected. If the edge is active during this infection time, the initialization is immediate. Otherwise, the initialization occurs at the next activation time of the edge.
Then, for all $\mathcal{P}\in\mathscr{P}_v$ and $j\in\{2,\ldots,|\mathcal{P}|\}$,
\begin{align}
    I_{u_j^{\mathcal{P}}} =
     s^0_{u_{j-1}^{\mathcal{P}}} + C(e_{j-1}^{\mathcal{P}}),
\end{align}
if
\begin{align}
    s^0_{u_{j-1}^{\mathcal{P}}} + C(e_{j-1}^{\mathcal{P}}) \leq I_{u_{j-1}^{\mathcal{P}}} + R_{u_{j-1}^{\mathcal{P}}} \hspace{0.5cm} \text{and} \hspace{0.5cm} s^0_{u_{j-1}^{\mathcal{P}}} + C(e_{j-1}^{\mathcal{P}}) \in \bigcup_{\sss i=1}^{\sss N\left(e_{j-1}^{\mathcal{P}}\right)}\left[\teon{e_{j-1}^{\mathcal{P}}}{i},\teoff{e_{j-1}^{\mathcal{P}}}{i}\right],
\end{align}
where we recall that $R_{u_{j-1}^{\mathcal{P}}}$ denotes the recovery time of vertex $u_{j-1}^{\mathcal{P}}$. Otherwise, $I_{u_j^{\mathcal{P}}} = \infty$.\\

In summary, in the second part of the backward process, we follow each $\mathcal{P}\in\mathcal{P}_v$ from its starting point, which is an initially infected vertex, to vertex $v$, identifying all $\left(I_{u_j^{\mathcal{P}}}\right)_{j=1}^{|\mathcal{P}|}$ and treating them as edge weights with respect to $\mathcal{P}$.
Once we encounter the first infinite mark, we stop exploring the path any further and set $I_v^{\mathcal{P}}=\infty$, where by $I_v^{\mathcal{P}}=\infty$ we denote the infection time of vertex $v$ on path $\mathcal{P}\in\mathscr{P}_v$. If none of the intermediate infection times is infinite, then we compute
\begin{align}
    I_v^{\mathcal{P}} = I_{u_{|\mathcal{P}|}^{\mathcal{P}}} + C\left(\left\{u_{|\mathcal{P}|}^{\mathcal{P}},v\right\}\right).
\end{align}
Denote the actual infection time of $v$ by $T(v)$. Then,
\begin{align}
    T(v) = \min_{\mathcal{P}\in\mathscr{P}_v} I_v^{\mathcal{P}}.
\end{align}
In this way, we have successfully identified the infection time of $v$ in the time-marked union graph $\markedg$. Repeating these steps for all the other vertices, the backward process outputs the number of susceptible, infected and recovered vertices at time $t$. For details see Algorithm 1 below.
\begin{algorithm} 
\caption{Backward Investigation of Epidemic Process on Dynamic Graph}
\begin{algorithmic} \label{alg1}
\Require $v$: vertex, $t$: time point, $\mathcal{I}^{(\rho)}(0)$: set of initially infected vertices, time-marked union graph $\markedg$.\\
Let $O = \{\}$, $W = \{v\}$ be the list of observed and waiting-to-be-observed vertices, respectively.
\While {$W \neq \emptyset$}
    \State Let $u$ be the next vertex in $W$. Let $N^{[0,T]}(u)$ be the set of all neighbors of $u$ throughout $[0,T]$.
    \For {each $w \in N^{[0,T]}(u)$}
        \State Draw transmission time $C(\{u, w\})$ from $D_I$ and recovery time $R_w$ from $D_R$.
        \If {$R_w > C(\{u, w\})$ and $w \notin O$}
            \State Add $w$ to $W$.
        \EndIf
    \EndFor
    \State Remove $u$ from $W$ and add $u$ to $O$.
\EndWhile
\State Let $\mathscr{P}_v$ be the set of all feasible paths from $\mathcal{I}^{(\rho)}(0) \cap O$ to $v$ identified in the first part.
\For {each path $\mathcal{P} \in \mathscr{P}_v$}
    \State Let $\{u_1^{\mathcal{P}}, u_2^{\mathcal{P}}, \ldots, u_{|\mathcal{P}|+1}^{\mathcal{P}}\}$ be the vertices along $\mathcal{P}$, where $u_1^{\mathcal{P}} \in \mathcal{I}^{(\rho)}(0)$ and $u_{|\mathcal{P}|+1}^{\mathcal{P}} = v$. Set $I_{u_1^{\mathcal{P}}} = 0$.
    \For {$j = 1$ to $|\mathcal{P}|$}
        \State Let $e_j^{\mathcal{P}} = \{u_j^{\mathcal{P}}, u_{j+1}^{\mathcal{P}}\}$.
        \State \textbf{Compute} $s^0_{u_j^{\mathcal{P}}}$:
            \If{$I_{u_j^{\mathcal{P}}} \notin \bigcup_{\sss i=1}^{\sss N\left(e_j^{\mathcal{P}}\}\right)} \left[\teon{e_j^{\mathcal{P}}}{i}, \teoff{e_j^{\mathcal{P}}}{i}\right]$}
               \State $s^0_{u_j^{\mathcal{P}}} \gets \min_{i} \left\{\teon{e_j^{\mathcal{P}}}{i} : \teon{e_j^{\mathcal{P}}}{i} > I_{u_j^{\mathcal{P}}} \right\}$.
            \Else
                \State $s^0_{u_j^{\mathcal{P}}} \gets I_{u_j^{\mathcal{P}}}$.
            \EndIf
        \State \textbf{Compute} potential infection time:
            \[
            I_{u_{j+1}^{\mathcal{P}}} = s^0_{u_j^{\mathcal{P}}} + C\left(e_j^{\mathcal{P}}\right).
            \]
        \State \textbf{Check} transmission conditions:
            \If{$I_{u_{j+1}^{\mathcal{P}}} \leq I_{u_j^{\mathcal{P}}} + R_{u_j^{\mathcal{P}}}$ \textbf{and}\\
            \qquad\qquad$I_{u_{j+1}^{\mathcal{P}}} \in \bigcup_{\sss i=1}^{\sss N\left(e_j^{\mathcal{P}}\right)} \left[\teon{e_j^{\mathcal{P}}}{i},\teoff{e_j^{\mathcal{P}}}{i}\right]$}
                \State Continue the loop.
            \Else
                \State \textbf{Break} out of the loop for this path and set $I_v^{\mathcal{P}} = \infty$.
            \EndIf
    \EndFor
    \State Record $I_v^{\mathcal{P}} = I_{u_{|\mathcal{P}|}^{\mathcal{P}}} + C\left(\{u_{|\mathcal{P}|}^{\mathcal{P}}, v \}\right)$.
\EndFor
\State \textbf{Compute} the infection time of $v$:
    \[
    I_v = \min_{\mathcal{P} \in \mathscr{P}_v} I_v^{\mathcal{P}}.
    \]
\Return $I_v$
\end{algorithmic}
\end{algorithm}

\subsection{Epidemic marks given by the backward process} \label{sec_marks_epid}
Let $\left(G_n^s\right)_{s\in[0,T]}$ be a dynamic graph and $\markedunion$ the corresponding time-marked union graph as in Definition \ref{def_time_marked_union_graph}.

To formally represent initial statuses of the vertices, individual recovery times of each vertex and transmission times on each edge recorded during the backward process, we again use marks and set them up on time-marked union graphs, for which the backward process was defined. To distinguish between the space of time marks of a dynamic graph that we have defined in Definition \ref{def_time_marks_dyn_graphs}, we simply denote the space of epidemic marks identified during the backward process by $\Lambda$ throughout this paper. For the SIR epidemic that we are investigating in this paper, this mark space is expressed by $\Lambda = [0,\infty)\times\{S,I\}$. The first coordinate records the transmission times for edges and the recovery times for vertices. The second coordinate is relevant for vertices only and it records their initial states.

We represent the probability distribution associated with the set of epidemic marks $\mathcal{M}_{\Lambda}\left(\markedunion\right)$ as \makebox[\textwidth][s]{$\mathbf{P}_{\Lambda}\left(\cdot\mid \markedunion \right)$ or, alternatively, to accentuate the dynamic graph that $\markedunion$ is derived from, by}\\
$\mathbf{P}_{\Lambda}\left(\cdot\mid \left(G_n^s \right)_{s\in[0,T]}\right)$. Respectively, we represent the probability distribution associated with the set of epidemic marks $\mathcal{M}_{\Lambda}\left(\left(\markedunion,o_n\right)\right)$ on the rooted $\left(\markedunion,o_n\right)$ as $\mathbf{P}_{\Lambda}\left(\cdot\mid \left(\markedunion,o_n\right) \right)$. Subsequently, we denote the probability measure related to the epidemic dynamic on the limiting rooted time-marked union graph $\marklim$ as $\mu_{\Lambda}$. To derive it, we first select $\marklim\sim\mu$ and then assign marks applying $\mathbf{P}_{\Lambda}(\cdot\mid \marklim)$.

Applying $\mathbf{P}_{\Lambda}(\cdot\mid \left(\markedunion,o_n\right))$, the first coordinate draws transmission times for every available edge from the distribution $D_I$ and recovery times for each vertex from the distribution $D_R$. Subsequently, the second coordinate assigns status $I$ to each vertex independently with probability $\rho$ and status $S$ otherwise. It happens independently of the first coordinate. Based on all this information, the final epidemic statuses can be identified by the second part of the backward process per each possible infection path (see Algorithm 1).

We will also consistently represent the probability distribution and expectation associated with the randomness of the dynamic graph by $\mathbf{P}_n$ and $\mathbf{E}_n$, respectively. Further, the probability distribution and expectation used without subscripts, i.e., $\mathbf{P}$ and $\mathbf{E}$, will incorporate both the randomness of the epidemic process and of the dynamic graph.

\begin{definition}[Metric on the epidemic marks]
Here, we specify the metric on the marks assigned by the backwards process to time-marked union graphs. We first specify the metric on just epidemic marks, denoted by $\text{d}_{\Lambda}$, and then a metric on the joint marks space $\Lambda\times\Xi$, denoted by $\text{d}_{\Lambda\times\Xi}$.\\
Let $\left(G^s\right)_{s\in[0,T]}$ and $\left(\Tilde{G}^s\right)_{s\in[0,T]}$ be two distinct dynamic graphs. Let $\markedunion$ and $\markeduniontilde$ denote the time-marked union graphs corresponding to $\left(G_s\right)_{s\in[0,T]}$ and $\left(\Tilde{G}_s\right)_{s\in[0,T]}$, respectively. For $i\in\mathcal{V}(G), \tilde{i}\in\mathcal{V}(\tilde{G})$, we define
\begin{align*}
    \text{d}_{\Lambda}\left(\mathcal{M}_{\Lambda}(i),\tilde{\mathcal{M}}_{\Lambda}(\tilde{i}) \right) = \max\left\{|R_i-R_{\tilde{i}}|, \text{d}_0\left(\mathcal{M}_{\Lambda}(i),\tilde{\mathcal{M}}_{\Lambda}(\tilde{i}) \right) \right\},
\end{align*}
where $R_i, R_{\tilde{i}}$ denote the recovery times of $i$ and $\tilde{i}$, respectively, $|\cdot|$ denotes the Euclidean metric on $[0,\infty)$ and $\text{d}_0\left(\mathcal{M}_{\Lambda}(i),\tilde{\mathcal{M}}_{\Lambda}(\tilde{i}) \right)$ is given by
\begin{align*}
    &\text{d}_0\left(\mathcal{M}_{\Lambda}(i),\tilde{\mathcal{M}}_{\Lambda}(\tilde{i}) \right) = 0 \hspace{0.75cm}\text{if}\hspace{0.25cm} \mathcal{M}_{\Lambda}(i)=\tilde{\mathcal{M}}_{\Lambda}(\tilde{i}) = I \hspace{0.25cm}\text{or}\hspace{0.25cm} \mathcal{M}_{\Lambda}(i)=\tilde{\mathcal{M}}_{\Lambda}(\tilde{i}) = S,\\
    &\text{d}_0\left(\mathcal{M}_{\Lambda}(i),\tilde{\mathcal{M}}_{\Lambda}(\tilde{i}) \right) = 1 \hspace{0.75cm}\text{otherwise}.
\end{align*}
Further, for $e\in\mathcal{E}(G), \tilde{e}\in\mathcal{E}(\tilde{G})$, we define
\begin{align*}
    \text{d}_{\Lambda}\left(\mathcal{M}_{\Lambda}(e),\tilde{\mathcal{M}}_{\Lambda}(\tilde{e}) \right) = |C(e)-\tilde{C}(\tilde{e})|,
\end{align*}
where $C(e),\tilde{C}(\tilde{e})$ are transmission times on the edges $e$ and $\tilde{e}$, respectively, and $|\cdot|$ denotes the Euclidean metric on $[0,\infty)$.\\
Having defined $\text{d}_{\Lambda}$, we proceed to define $\text{d}_{\Lambda\times\Xi}$. Note that $\Xi$ only assigns marks to edges. Hence,
\begin{align*}
    \text{d}_{\Lambda\times\Xi}\left(\mathcal{M}_{\Lambda\times\Xi}(i),\tilde{\mathcal{M}}_{\Lambda\times\Xi}(\tilde{i}) \right) = \text{d}_{\Lambda}\left(\mathcal{M}_{\Lambda}(i),\tilde{\mathcal{M}}_{\Lambda}(\tilde{i}) \right).
\end{align*}
For $e\in\mathcal{E}(G), \tilde{e}\in\mathcal{E}(\tilde{G})$, we define
\begin{align*}
    \text{d}_{\Lambda\times\Xi}\left(\mathcal{M}_{\Lambda\times\Xi}(e),\tilde{\mathcal{M}}_{\Lambda\times\Xi}(\tilde{e}) \right) = \max\left\{ \text{d}_{\Lambda}\left(\mathcal{M}_{\Lambda}(e),\tilde{\mathcal{M}}_{\Lambda}(\tilde{e}) \right), \text{d}_{\Xi}\left(\mathcal{M}(e),\tilde{\mathcal{M}}(\tilde{e}) \right)\right\},
\end{align*}
with $\text{d}_{\Xi}\left(\mathcal{M}(e),\tilde{\mathcal{M}}(\tilde{e})\right)$ as in Definition \ref{def_metric_marks_dyn_graphs}.
\end{definition}

\section{Proof of Theorem \ref{thm_dyn_conv_epid}} \label{sec_proofs_overview}
The proof of Theorem \ref{thm_dyn_conv_epid} requires two steps. We first show that an epidemic ran within an $r$-ball of every vertex concentrates around the real-time evolution of an epidemic on the entire graph. Next, we show that an analogous approximation is also possible on the limiting graph. This section consists of a more detailed description of these steps together with their proofs.

\subsection{Auxiliary results}
We start by introducing the notation that will frequently appear throughout the rest of the paper. Let $\markedunion$ denote the time-marked union graph determined by the dynamic random graph
$\left(G_n^s\right)_{s\in[0,T]}$, with time-marks specified by the ON and OFF times of its edges. Denote the infection time of the root $o_n$ in the prelimit $(\markedunion,o_n,\mathcal{M}_{\Lambda}(\markedunion))$, assuming that the network is confined to $B_r\markedunionroot$, by $T_n^{(r)}(o_n)$. It refers to the length of the shortest path from the set of initially infected vertices in
$B_r\markedunionroot$ to the root $o_n$. Similarly, denote the infection time of the root $o$ in the limit $(\markedg,o,\mathcal{M}_{\Lambda}(\markedg))$, assuming that the network is confined to $B_r\marklim$, by $T^{(r)}(o)$, which corresponds to the length of the shortest path from the set of initially infected vertices in $B_r\marklim$ to the root $o$.\\

\noindent We aim at approximating the proportion $\susn$ of susceptible vertices at time $t$, for $t\in[0,T]$, with a collection of functions restricted to $r$-neighbourhoods, i.e.,
\begin{align} \label{susnr_formula}
    \susnr = \frac{1}{n} \sum_{v\in [n]} \mathds{1}_{\big\{t<T_n^{(r)}(v)\big\}}.
\end{align}
We can define $\infnr$ and $\recnr$ analogously to $\susnr$, and then also the entire vector $\epidnr$, with $\epidnr = (\susnr,\infnr,\recnr)$.  We start by bounding the first moment of the local approximation $\epidnr$ of $\epidn$, which we formalise in the following proposition:

\begin{proposition}[Local approximation - first moment] \label{lem_1stmom}
For a given graph $(G_n^s)_{s\in[0,T]}$ and for any $t\in[0,T]$, $r\geq1$, $\rho\in(0,1)$,
\begin{align}
    \mathbf{E}\left[\sup_{t\in[0,T]}|\epidn - \epidnr| \right] \leq (1-\rho)^r,
\end{align}
where the expectation is with respect to the epidemic (i.e., the transmission and recovery rates and the random set of initially infected vertices) and the randomness of the dynamic graph, and the norm used on $\mathbf{R}^3$ is the maximum norm.
\end{proposition}
The proof of Proposition \ref{lem_1stmom} is very similar to a related result for static graphs (see \textup{\cite[Lemma 4.1]{Alimohammadi2024}}). It relies on the fact that if we start an epidemic with a positive proportion of infected vertices, then these initially infected vertices will be locally reachable. As a result, an infection path of the root, i.e., the shortest path from one of the initially infected vertices to the root in terms of edge weight identified by the backward process, will contain only a limited number of vertices. The details are given in Section \ref{sec_proof_1mom}.

We now evaluate the second moment of the local approximation, on which we have the following bound:

\begin{proposition}[Local approximation dynamic random graphs - second moment]\label{lem_2ndmom_random}
Let $\big(G_n^s\big)_{s\in[0,T]}$ be a sequence of dynamic random graphs such that its corresponding time-marked union graph $\markedunion$ as in Definition \ref{def_time_marked_union_graph} converges locally in probability. Then, for any given $\delta>0$ and large enough $n$,
\begin{align}
    \sup_{t\in[0,T]}\Var(\susnr) \leq\delta,
\end{align}
where the variance is over the randomness of the epidemic process and the dynamic random graph $\big(G_n^s\big)_{s\in[0,T]}$.
\end{proposition}
The proof of Proposition \ref{lem_2ndmom_random} begins by applying the law of total variance to decompose the second moment into two components: the expectation of the conditional variance and the variance of the conditional expectation, where we condition on the dynamic random graph. To simplify the proof, the bound on the second moment, conditioned on the dynamic graph, is established in a separate lemma (see Lemma \ref{lem_2ndmom_union}). This bound leverages the fact that epidemic processes up to time $t$ for two vertices located far apart in the union graph are conditionally independent, given the graph process. Based on this, the first component is bounded by taking the expectation of the expression derived in Lemma \ref{lem_2ndmom_union} and using implications of local time-marked union convergence for neighbourhood sizes. Bounding the second component requires a more intricate argument. This begins by conditioning on all possible trajectories of the time-marked union graph associated with $\big(G_n^s\big)_{s\in[0,T]}$, followed by using consequences of local time-marked union convergence in probability. The full details of this argument are provided in Section \ref{sec_proof_2mom_random}.\\

We now proceed with the limiting time-marked union graph $\marklim$. Similarly as for the original graph, we would like to compare the actual epidemic to its local approximation. Hence, we define the proportion of susceptible vertices at time $t\in[0,T]$ in the limiting graph restrained to radius $l$ with respect to the infection time of the root as
\begin{align}
    s_l(t) = \mu_{\Lambda}\big(\mathds{1}_{\{t<T^{(l)}(o)\}}\big).
\end{align}
Analogously, we can define
\begin{align}
    r_l(t) = \mu_{\Lambda}\big(\mathds{1}_{\{t>T^{(l)}(o)+R_o\}}\big),
\end{align}
and then $i_l(t) = 1 - s_l(t) - r_l(t)$. Using once again the argument about local reachability of initially infected vertices, we can bound the difference between $s_l(t),s_{l'}(t)$ for any integers $l$ and $l'$, uniformly for all $t\in[0,T]$. Subsequently, using monotone convergence, we can also show that $s_l(t)$ (and analogously $i_l(t),r_l(t)$) is close to $s(t) = \lim_{l\to\infty}s_l(t)$:

\begin{proposition}[Local approximation of the limit]\label{lem_on_the_lim} For any measure $\mu$ on $\mathcal{G}_{\star}$, and any integers $l$ and $l'$, 
\begin{align} \label{on_the_lim_ineq}
\mathbf{E}\left[\mu_{\Lambda}\bigg(\sup_{t\in[0,T]}\big|\mathds{1}_{\left\{t<T^{(l)}(o)\right\}}-\mathds{1}_{\left\{t<T^{(l')}(o)\right\}}\big|\bigg)\right] \leq (1-\rho)^{\min\{l,l'\}}.
\end{align}
Thus, $s(t) = \lim_{l\to\infty}s_l(t)$, $i(t) = \lim_{l\to\infty} i_l(t)$ and $r(t) = \lim_{l\to\infty} r_l(t)$ are well-defined, and
\begin{align}
    \sup_{t\in[0,T]} |(s_l(t), i_l(t), r_l(t)) - (s(t), i(t), r(t))| \leq (1 - \rho)^l.
\end{align}
\end{proposition}
We prove Proposition \ref{lem_on_the_lim} in Section \ref{sec_proof_loc_approx_on_the_lim}. Next, note that $t \mapsto\susnr$ for $t\in[0,T]$ is not continuous in $t$. It has at most $n$ discontinuity points corresponding to infection times of all $v\in[n]$. Hence, to build a bridge between an epidemic on $\left(G_n^s\right)_{s\in[0,T]}$ and an epidemic on its limit, we need an extra argument showing that $\susnr$ converges to $s_r(t)$ for any $t\in[0,T]$, which is the subject of the following proposition:

\begin{proposition}[Convergence of the dynamic local approximation]\label{lem_conv_loc_approx} Let $(G_n^s)_{s\in[0,T]}$ be a sequence of dynamic graphs such that its corresponding time-marked union graph, $\markedunionroot$, converges locally in probability to $\marklim\sim\mu$. Then, for any $t\in[0,T]$,
\begin{align}
    \mathbf{E}_n\left[S_{n,r}^{(\rho)}(t) \right] \stackrel{\mathbf{P}}{\longrightarrow} s_r(t),
\end{align}
where convergence in probability is with respect to the randomness of the dynamic random graph $(G_n^s)_{s\in[0,T]}$.
\end{proposition}
The proof of Proposition \ref{lem_conv_loc_approx} follows in two steps: We first define a suitable functional related to the probability of infecting the root of a given input time-marked union graph and show that it is continuous and bounded (see Lemma \ref{lem_fct_cont_bounded}). Then, we rely on the fact that conditioning on a dynamic graph process is equivalent to knowing the time-marked union graph and we apply the alternative characterization of weak convergence in terms of convergence of bounded and continuous functionals to deduce the claim. Details are given in Section \ref{sec_proof_conv_loc_approx}.

With the propositions established above, we can now proceed to the proof of our main result, which will be presented in the following section.

\subsection{Proof of the convergence of the epidemic process}\label{sec_proof_main_thm}

\begin{proof}[Proof of Theorem \ref{thm_dyn_conv_epid} subject to Propositions \ref{lem_1stmom}-\ref{lem_conv_loc_approx}]
We follow an approach similar to the proof for static graphs, see \textup{\cite[Proof of Theorem 2.5]{Alimohammadi2024}}. We derive the desired result for the proportion of susceptible vertices $\susn$. The proof is similar for infectious and recovered vertices. We first use Propositions \ref{lem_1stmom} and \ref{lem_2ndmom_random} to apply the first- and second-moment method to the proportion of susceptible vertices. By the Markov inequality and Proposition \ref{lem_1stmom}, for any $\delta'>0$, any $t\in[0,T]$ and any $r$,
\begin{align}
    \mathbf{P}\big(|\susn-\susnr|>\delta'\big) \leq\frac{1}{\delta'}\mathbf{E}\big[|\susn-\susnr|\big]\leq \frac{1}{\delta'}\mathbf{E}\big[\sup_{t\in[0,T]}|\susn-\susnr|\big]\leq \frac{(1-\rho)^r}{\delta'}.
\end{align}
Moreover, by the Chebyshev inequality, for any $\delta'>0$, any $t\in[0,T]$ and any $r\geq1$,
\begin{align}
    \mathbf{P}\big(|\susnr-\mathbf{E}[\susnr]|>\delta' \big)\leq\frac{1}{(\delta')^2}\Var(\susnr).
\end{align}
Combining the two inequalities and applying the triangle inequality we obtain, again for any $\delta'>0$, any $t\in[0,T]$ and any $r$,
\begin{align}
    \mathbf{P}\bigg(|\susn-\mathbf{E}[\susnr]|>2\delta'\bigg)\leq \frac{(1-\rho)^r+\Var(\susnr)}{\delta'}.
\end{align}
Recall that by Proposition \ref{lem_2ndmom_random}, for $n$ large enough, we can bound $\Var(\susnr)$ by any given $\delta>0$. Thus, for any $t\in[0,T]$, any $\Tilde{\delta}>0$ and $r$ and $n$ large enough,
\begin{align}
    \mathbf{P}\bigg(\bigg|\susn-\mathbf{E}[\susnr]\bigg|>\Tilde{\delta}\bigg) \leq \Tilde{\delta}.
\end{align}
Next, by Proposition \ref{lem_conv_loc_approx}, we know that $\mathbf{E}_n[\susnr]$ converges in probability to $s_r(t)$, for any $t\in[0,T]$, and by Proposition \ref{lem_on_the_lim} we have that $s_r(t)$ gets arbitrarily close to $s(t)$, for large $r$ and any $t\in[0,T]$. Thus, we can conclude that, for any $t\in[0,T]$ and $n$ large enough,
\begin{align}
    \susn \stackrel{\mathbf{P}}{\longrightarrow} s(t).
\end{align}
Same proof applies to $\recn$, and, since $\infn = 1 - \susn - \recn$, also to $\infn$.
\end{proof}

\section{Proofs of Propositions \ref{lem_1stmom} and \ref{lem_2ndmom_random}} \label{sec_all_proofs_prelim}
Here we provide proofs of Propositions \ref{lem_1stmom} and \ref{lem_2ndmom_random}.

\subsection{Proof of the bound on the first moment of the local approximation} \label{sec_proof_1mom}

\begin{proof}[Proof of Proposition \ref{lem_1stmom}]
We adapt the proof of the equivalent result for static graphs, see \textup{\cite[Lemmas 4.1 and 4.4]{Alimohammadi2024}}. We first prove the result conditionally on the trajectory of the dynamic graph. Note that for $\epidn$ and $\epidnr$ to differ, the infection time assigned to a significant number of vertices when inspecting their $r$-neighbourhoods only would have to be different than their infection times found when considering the entire graph. Hence,
\begin{align} \label{lem_1mom_eq_upper_bound}
    \mathbf{E}_{\Lambda}\left[\sup_{t\in[0,T]}|\epidn - \epidnr|\mid \big(G_n^s\big)_{s\in[0,T]} \right]&\leq \mathbf{E}_{\Lambda}\left[ \frac{1}{n}\big|\big\{v:T_n^{(r)}(v)\neq T_n^{\sss(\infty)}(v)\big\}\big|\mid \big(G_n^s\big)_{s\in[0,T]}\right]\\
    &= \mathbf{P}_{\Lambda}\left(T_n^{(r)}(o_n)\neq T_n^{\sss(\infty)}(o_n)\mid \big(G_n^s\big)_{s\in[0,T]}\right),\nonumber
\end{align}
where $\mathbf{E}_{\Lambda}$ is an expectation with respect to the marks of the epidemic (recall Section \ref{sec_marks_epid}). Note that given $\big(G_n^s\big)_{s\in[0,T]}$, we can determine the time-marked union graph $\markedunion$ (Definition \ref{def_time_marked_union_graph}), and the backward process can successfully investigate epidemic progression on such a time-marked union graph. Hence, we are investigating situations in which the shortest weighted path with respect to the edge weight identified by the backward process originates from an initially infected vertex situated \emph{beyond} the $r$-neighbourhood. Note that for this to happen, this path cannot contain any initially infected vertices in the $r$-neighbourhood. Otherwise, the infection time would have been correctly identified in the $r$-neighbourhood. Thus, note that if we condition on all transmission and recovery times and only keep the initial infections random, we obtain
\begin{align}
    \mathbf{P}\bigg(T_n^{(r)}(o_n)\neq T_n^{\sss(\infty)}(o_n)\mid \left(\markedunion, o_n,
    \mathcal{M}_{\Lambda'}\left(\markedunion\right)\right)\bigg)\leq (1-\rho)^r,
\end{align}
where $\mathcal{M}_{\Lambda'}(\cdot)$ contains only the marks corresponding to transmission and recovery times, not the information on the initial infections, and the probability is then taken with respect to these initial infections. Taking expectations with respect to the randomness of transmission and recovery times and using the tower property yields
\begin{align}
    \mathbf{P}_{\Lambda}&\big(T_n^{(r)}(o_n)\neq T_n^{\sss(\infty)}(o_n)\mid \markedunion\big)\\
    &= \mathbf{E}_{\Lambda'}\left[\mathbf{P}\bigg(T_n^{(r)}(o_n)\neq T_n^{\sss(\infty)}(o_n)\mid \left(\markedunion, o_n,
    \mathcal{M}_{\Lambda'}\left(\markedunion\right)\right)\bigg)\right] \nonumber\\
    &\leq \mathbf{E}_{\Lambda'}\left[(1-\rho)^r\right] = (1-\rho)^r,\nonumber
\end{align}
which provides the desired upper bound (see \eqref{lem_1mom_eq_upper_bound}) on the first moment, conditionally on the time-marked union graph and hence, conditionally on the trajectory of the dynamic graph, as the latter determines the first. Applying the tower property again, this time taking expectation with respect to the randomness of the dynamic graph, finishes the proof, since
\begin{align}
    \mathbf{E}_{\Lambda}\big[\sup_{t\in[0,T]}\big|\epidn - \epidnr\big| \big]&=\mathbf{E}_n\bigg[\mathbf{E}_{\Lambda}\big[\sup_{t\in[0,T]}|\epidn - \epidnr|\mid \big(G_n^s\big)_{s\in[0,T]}\big]\bigg]\\
    &\leq \mathbf{E}_n\left[\mathbf{P}_{\Lambda}\big(T_n^{(r)}(o_n)\neq T_n^{\sss(\infty)}(o_n)\mid \markedunion\big) \right]\nonumber\\
    &\leq \mathbf{E}_n\big[(1-\rho)^r\big] = (1-\rho)^r,\nonumber
\end{align}
where $\mathbf{E}_n$ refers to the randomness of the dynamic graph.
\end{proof}

\subsection{Proof of the bound on the second moment of the local approximation} \label{sec_proof_2mom_random}
As mentioned in Section \ref{sec_proofs_overview}, before proving Proposition \ref{lem_2ndmom_random}, we first derive a bound on the second moment, conditioned on the dynamic graph. This result is presented in the following lemma:
\begin{lemma}[Local approximation given the dynamic graph - second moment]\label{lem_2ndmom_union}
Let $\big(G_n^s\big)_{s\in[0,T]}$ be a sequence of dynamic random graphs and let $G^{[0,t]}_n$  be the corresponding union graph up to time $ t\in[0,T]$. Then,
\begin{align}
    \sup_{t\in[0,T]} \Var_{\Lambda}\big(\susnr\mid \big(G_n^s\big)_{s\in[0,T]}\big) \leq \frac{1}{n} + \varepsilon_{2r}(G^{[0,T]}_n),
\end{align}
where $\varepsilon_{r}(G^{[0,T]}_n)$ denotes the proportion of node pairs separated by distance at most $r$ in $G_n^{[0,T]}$, i.e.,
\begin{align} \label{vareps_dist}
    \varepsilon_{r}&\left(G^{[0,T]}_n\right)=\frac{1}{n^2} \left| \left\{ (u, v) \in [n] \times [n] : \mathrm{dist}_{G^{[0,T]}_n}(u,v) \leq r \right\} \right|,
\end{align}
and the variance is over the randomness of the epidemic process only. 
\end{lemma}

\begin{proof}
We adapt the proof of the equivalent result for static graphs, see \textup{\cite[Lemma 4.2]{Alimohammadi2024}}. The proof follows from the fact that the events $\{t<T_n^{(r)}(u)\}$ and $\{t<T_n^{(r)}(v)\}$ for any $t\in[0,T]$ are independent given the realisation of the graph, if $\text{dist}_{G^{[0,T]}_n}(u,v) > 2r$. Let
\begin{align*}
    Z_v(t) = \mathds{1}\{t<T_n^{(r)}((G_n^s)_{s\in[0,T]},v,\mathcal{M}_{\Lambda}((G_n^s)_{s\in[0,T]}))\}.
\end{align*}
Note that given $((G_n^s)_{s\in[0,T]})$, we can determine the corresponding marked union graph and successfully perform the backward process. Hence, for any $t\in[0,T]$, we can write
\begin{align}
    n^2 \Var_{\Lambda}\big(\susnr\mid \big(G_n^s\big)_{s\in[0,T]}\big) &= \mathbf{E}_{\Lambda}\big[\big(\sum_{v\in[n]}Z_v(t) - \sum_{v\in[n]}\mathbf{E}_{\Lambda}[Z_v(t) ] \big)^2  \big]\\
    =&\sum_{v\in[n]} \mathbf{E}_{\Lambda}\big[\big(Z_v(t) - \mathbf{E}_{\Lambda}[Z_v(t)]\big)^2\big]\nonumber\\
    &+ \sum_{u,v\in[n]:\text{dist}_{G^{[0,T]}_n}(u,v)\leq 2r} \mathbf{E}_{\Lambda}\big[\big(Z_v(t) - \mathbf{E}_{\Lambda}[Z_v(t)]\big)\big(Z_u(s) - \mathbf{E}_{\Lambda}[Z_u(s)]\big)\big]\nonumber\\
    &+\sum_{u,v\in[n]:\text{dist}_{G^{[0,T]}_n}(u,v)>2r} \mathbf{E}_{\Lambda}\big[\big(Z_v(t) - \mathbf{E}_{\Lambda}[Z_v(t)]\big)\big(Z_u(s) - \mathbf{E}_{\Lambda}[Z_u(s)]\big)\big]\nonumber\\
    =&\sum_{v\in[n]} \mathbf{E}_{\Lambda}\big[\big(Z_v(t) - \mathbf{E}_{\Lambda}[Z_v(t)]\big)^2\big]\nonumber\\
    &+ \sum_{u,v\in[n]:\text{dist}_{G^{[0,T]}_n}(u,v)\leq 2r} \mathbf{E}_{\Lambda}\big[\big(Z_v(t) - \mathbf{E}_{\Lambda}[Z_v(t)]\big)\big(Z_u(s) - \mathbf{E}_{\Lambda}[Z_u(s)]\big)\big]\nonumber,
\end{align}
where we have made use of the mentioned independence. Note that $0\leq Z_v(t)\leq 1$, so an obvious upper bound is $|Z_v(t) - \mathbf{E}_{\Lambda}[Z_v(t)]|\leq 1$. Thus,
\begin{align}
    n^2 \Var_{\Lambda}\big(\susnr\mid \big(G_n^s\big)_{s\in[0,T]}\big) \leq n + n^2 \varepsilon_{2r}(G^{[0,T]}_n).
\end{align}
\end{proof}
Having derived the above bound, we proceed to the proof of Proposition \ref{lem_2ndmom_random}:
\begin{proof}[Proof of Proposition \ref{lem_2ndmom_random}]
We adapt the proof of the equivalent result for static graphs, see \textup{\cite[Lemma 4.4]{Alimohammadi2024}}. By the law of total variance, for any $t\in[0,T]$,
\begin{align} \label{law_total_var}
    \Var\big(\susnr\big) = \mathbf{E}_n\big[\Var_{\Lambda}\big(\susnr\mid \big(G_n^s\big)_{s\in[0,T]}\big)\big]+ \Var_n\big(\mathbf{E}_{\Lambda}\big[\susnr\mid \big(G_n^s\big)_{s\in[0,T]}\big]\big).
\end{align}
\textbf{Step 1: Bounding the expectation of the conditional variance.} We bound the first term applying Lemma \ref{lem_2ndmom_union} as
\begin{align} \label{1st_term_bound1}
    \mathbf{E}_n\big[\Var_{\Lambda}\big(\susnr\mid \big(G_n^s\big)_{s\in[0,T]}\big)\big] \leq \frac{1}{n} + \mathbf{E}_n\left[\varepsilon_{2r}(G^{[0,T]}_n)\right].
\end{align}
It remains to bound $\mathbf{E}_n\big[\varepsilon_{2r}(G^{[0,T]}_n)\big]$ which, we recall, is the proportion of node pairs separated by distance at most $2r$ in $G_n^{[0,T]}$ in \eqref{vareps_dist}. We will derive a bound on $\varepsilon_{2r}\left(G^{[0,T]}_n\right)$ by introducing $\varepsilon_r\big(G_n^{[0,T]},l\big)$ - the empirical probability that the collective union neighbourhood $B_r(G_n^{[0,T]},o_n)$ contains at least $l$ vertices, with $o_n$ chosen uniformly at random, i.e.,
\begin{align}
    \varepsilon_r\big(G_n^{[0,T]},l\big) = \frac{1}{n}\sum_{v\in [n]} \mathds{1}_{\big\{|B_r(G_n^{[0,T]},v)|\geq l\big\}}.
\end{align}
For any $l\in\mathbf{N}$,
\begin{align}
    n^2\varepsilon_{2r}(G^{[0,T]}_n) &= \sum_{v\in[n]} |B_{2r}(G^{[0,T]}_n,v)|\\
    &= \sum_{v:|B_{2r}(G^{[0,T]}_n,v)|\geq l} |B_{2r}(G^{[0,T]}_n,v)| + \sum_{v:|B_{2r}(G^{[0,T]}_n,v)|<l} |B_{2r}(G^{[0,T]}_n,v)|.\nonumber
\end{align}
Note that there are $n\varepsilon_{2r}(G^{[0,T]}_n,l)$ vertices with $|B_{2r}(G^{[0,T]}_n,v)|\geq l$. Using an obvious bound of $|B_{2r}(G^{[0,T]}_n,v)|<n$ for such vertices and the bound of $|B_{2r}(G^{[0,T]}_n,v)|<l$ for the rest, we arrive at
\begin{align}
    \varepsilon_{2r}(G^{[0,T]}_n) \leq \varepsilon_{2r}\big(G^{[0,T]}_n,l\big) + \frac{l}{n},
\end{align}
which substituted into \eqref{1st_term_bound1} yields
\begin{align}
    \mathbf{E}_n\big[\Var_{\Lambda}\big(\susnr\mid \big(G_n^s\big)_{s\in[0,T]}\big)\big] \leq \frac{1}{n} + \frac{l}{n} + \mathbf{E}_n\left[\varepsilon_{2r}(G^{[0,T]}_n,l)\right].
\end{align}
Note that local convergence of the time-marked union graph in particular implies that also the unmarked union graph $G_n^{[0,T]}$ converges locally and hence, the neighbourhoods in $G_n^{[0,T]}$ are tight (see \textup{\cite[Appendix C.3]{Alimohammadi2024}}), i.e., for all $r<\infty$ and $\delta>0$, there exists $l<\infty$ such that for all $n$ large enough,
\begin{align}
    \mathbf{P}_n\big(\varepsilon_r\big(G_n^{[0,T]},l\big) \leq \delta\big) \geq 1-\delta.
\end{align}
Hence, for any $\Tilde{\delta}>0$, $\mathbf{E}_n[\varepsilon_{2r}\big(G^{[0,T]}_n,l\big)] \leq \mathbf{E}_n[\varepsilon_{2r}\big(G^{[0,T]}_n,l\big)] \leq \Tilde{\delta}$. Thus, if $n$ is large enough, for any given $\bar{\delta}>0$,
\begin{align} \label{1st_term_bound_final}
    \mathbf{E}_n\big[\Var_{\Lambda}\big(\susnr\mid \big(G_n^s\big)_{s\in[0,T]}\big)\big] \leq \bar{\delta}.
\end{align}
\textbf{Step 2: Bounding the variance of the conditional expectation.} For the second term in \eqref{law_total_var}, note that given $\big(G_n^s\big)_{s\in[0,T]}$, we can
determine the time-marked union graph $\markedunion$. Consequently, we can then determine the full dynamic trajectory of any vertex $v\in[n]$, including its indirect connections. Note further that the backward process can correctly establish the infection time of any vertex $v\in[n]$ given this full dynamic trajectory, as the latter delivers all the necessary information about the dynamic changes that can influence the infection status of $v$. Our objective is now to condition on all possible shapes that a dynamic trajectory of a uniformly chosen vertex may take up to time $T$ and use consequences of dynamic local convergence to bound the second term in (\ref{law_total_var}). However,  as we are operating in a continuous-time framework and the time marks in $\markedunion$ are continuous, we first need to suitably discretize the timeline.\\

\noindent\textbf{Step 2a: Discretization of time.} For that purpose, let us define the change points $\tau^{(n)}_1(v), \tau^{(n)}_2(v), \dots$ as the moments when the ON/OFF states of edges recorded by the time-marked neighbourhood in the union graph of a vertex $v\in[n]$ change, i.e., either one of the currently ON edges switches OFF, or vice versa. To distinguish the marks in $\markedunion$ from the marks in its local limit $\markedg$, we denote the former as $\left(\left(\sigeon{e}{i}, \sigeoff{e}{i}\right)_{\sss i=1}^{\sss N(e)}\right)$ and the latter as $\left(\left(\teon{e}{i}, \teoff{e}{i}\right)_{\sss i=1}^{\sss N(e)}\right)$. The first change point $\tau^{(n)}_1(v)$ is the earliest time when either any edge in the initial neighbourhood of vertex $v$ (i.e., $e$ such that $\sigeon{e}{1}=0$) switches OFF, or when a new one (i.e., $e$ such that $\sigeon{e}{1}>0$) switches ON. Formally, $\tau^{(n)}_1(v)$ is defined as
\begin{align*}
    \tau^{(n)}_1(v)= \min_{e \in B_r(\markedunion, v)} \left( \left\{ \sigeoff{e}{1} : \sigeon{e}{1} = 0 \right\} \cup \left\{ \sigeon{e}{1} : \sigeon{e}{1} > 0 \right\} \right).
\end{align*}
The subsequent change points $\tau^{(n)}_2(v), \tau^{(n)}_3(v), \dots$ are defined similarly. Indeed, after the first change point $\tau^{(n)}_1(v)$, the next change point $\tau^{(n)}_2(v)$ is the earliest time when another edge switches ON or OFF after $\tau^{(n)}_1(v)$: For $i\geq2$, the $i$-th change point $\tau^{(n)}_i(v)$ is defined as
\begin{align*}
    \tau^{(n)}_i(v)= \min_{e \in B_r(\markedunion, v), j\in\{1,\ldots,N(e)\}} \big( \left\{ \sigeoff{e}{j} : \sigeoff{e}{j} > \tau^{(n)}_{i-1}(v) \right\} \cup \left\{ \sigeon{e}{j} : \sigeon{e}{j} > \tau^{(n)}_{i-1}(v) \right\} \big),
\end{align*}
and we remind the reader that $N(e)$ denotes the total number of status changes of an edge $e$. Thus, the change points $\tau^{(n)}_1(v), \tau^{(n)}_2(v), \dots$ describe the sequence of time points when the ON/OFF statuses of the edges in the neighbourhood of $v$ changes. Recall Definition \ref{def_metric_marks_dyn_graphs}. To further simplify the notation, we denote $B_r\left(\markedunion, o_n\right)$ by $B_r^{\sss G_n^{\sss T}}(o_n)$ and the conditional probability with respect to the randomness of $\markedunion$, $\mathbf{P}(\cdot\mid\markedunion)$, by $\mathbf{P}_{\sss G_n^{\sss T}}$. Note that as a direct consequence of local convergence of the time-marked union graph (see Definition \ref{def_mark_union_conv_probab}), we know that for $n$ large enough, for any fixed marked rooted graph $\left(H,o,\left(\bar{t}^{\sss\Tilde{e}}_{\sss\mathrm{ON}},\bar{t}^{\sss\Tilde{e}}_{\sss\mathrm{OFF}}\right)_{\sss\Tilde{e}}\right)$, denoted from now on by $\hmarkedroot$, and for any $r\in\mathbf{N}$,
\begin{align*}
&\mathbf{P}_{\sss G_n^{\sss T}}\bigg(\forall e \in B_r^{\sss G_n^{\sss T}}(o_n), \phi(e)\in \hmarkedroot: \text{d}_{\Xi}\left(\mathcal{M}(e),\tilde{\mathcal{M}}(\phi(e))\right)\leq \frac{1}{r+1}\bigg)\\
&\stackrel{\mathbf{P}}{\longrightarrow} \hspace{0.1cm} \mathbf{P}\bigg(\forall e \in B_r^{\sss G^{\sss T}}(o), \phi(e)\in \hmarkedroot:\text{d}_{\Xi}\left(\mathcal{M}(e),\tilde{\mathcal{M}}(\phi(e))\right)\leq \frac{1}{r+1}\bigg),\nonumber
\end{align*}
where $B_r^{\sss G^{\sss T}}(o)$ denotes the $r$-neighborhood of the root $o$ in the limiting time-marked union graph $\marklim$. This implies that for $n$ large enough, any $r\in\mathbf{N}$ and for every edge $e$ in $B_r\left(\markedunion,o_n\right)$, there exist some limiting $\left(\left(t^{\sss\Tilde{e}}_{\sss i,\text{\rm{ON}}}, t^{\sss\Tilde{e}}_{\sss i,\text{\rm{OFF}}} \right)\right)_{\sss i=1}^{\sss N(\Tilde{e})}$, given by the time-marks of the edges in $\marklim$, such that 
\begin{align} \label{marks_vs_limit_marks}
\text{d}_{\Xi}\left(\left(\left(\sigeon{e}{i}, \sigeoff{e}{i} \right)\right)_{\sss  i=1}^{\sss N(e)}, \left(\left(t^{\sss\Tilde{e}}_{\sss i,\text{\rm{ON}}}, t^{\sss\Tilde{e}}_{\sss i,\text{\rm{OFF}}} \right)\right)_{\sss i=1}^{\sss N(\Tilde{e})}\right) \leq \frac{1}{r+1}.
\end{align}
As a consequence, since $\tau^{(n)}_1(o_n), \tau^{(n)}_2(o_n),\ldots,\tau^{(n)}_{\sss N_r}(o_n)$ are defined in terms of the edge marks in $\markedunion$, there exists $\left(\tau_1,\ldots,\tau_{\sss K}\right)$ such that, as $n\to\infty$,
\begin{align}
    \left(\tau^{(n)}_1(o_n), \tau^{(n)}_2(o_n),\ldots,\tau^{(n)}_{\sss N_r}(o_n) \right) \stackrel{\text{d}}{\longrightarrow} \left(\tau_1,\ldots,\tau_{\sss K}\right),
\end{align}
where $N_r$ denotes the number of changes in the dynamic neighbourhood of $o_n$ and $\left(\tau_1,\ldots,\tau_{\sss K}\right)$ are given by
\begin{align*}
    \tau_1 = \min_{e \in B_r(G^{[0,T]}, o)} \left( \left\{ t_{\sss 1,\text{OFF}}^e : t_{\sss 1,\text{ON}}^e = 0 \right\} \cup \left\{ t_{\sss 1,\text{ON}}^e : t_{\sss 1,\text{ON}}^e > 0 \right\} \right),
\end{align*}
and for $i\geq2$,
\begin{align*}
    \tau_{i} = \min_{e \in B_r(G^{[0,T]}, o), j\in\{1,\cdots,N(e)\}} \left\{ \big(t_{\sss j,\text{OFF}}^e: t_{\sss j,\text{OFF}}^e > \tau_{i-1} \right\}\cup \left\{ t_{\sss j,\text{ON}}^e : t_{\sss j,\text{ON}}^e > \tau_{i-1} \right\} \big).
\end{align*}
As a consequence of the local time-marked union convergence, it also holds that $K$ is finite with high probability.\\

\noindent\textbf{Step 2b: Partitioning the probability space.} Having argued that we are allowed to discretize the timeline, we can express $\mathbf{E}_{\Lambda}\big[\susnr\mid \big(G_n^s\big)_{s\in[0,T]}\big]$ in terms of all possible subgraph trajectories of the time-marked union neighbourhood. Denote, for $K$ fixed and `$\leq$' acting componentwise,
\begin{align}
    \mathbf{P}&\left((\tau^{(n)}_1(o_n), \ldots, \tau^{(n)}_{\sss N_r}(o_n))\leq (s_1,\ldots,s_K)\mid N_r=K, \left(G_n^s\right)_{s\in[0,T]}\right)\nonumber\\
    &=\frac{1}{n}\sum_{v\in[n]}\mathds{1}_{\{(\tau^{(n)}_1(v),\cdots,\tau^{(n)}_{\sss K}(v))\leq (s_1,\cdots,s_K) \}} = F^{(n)}_K(\bar{s}_{\sss K}),
\end{align}
where $\bar{s}_{\sss K}$ denotes a vector of $K$ time points $(s_1,\ldots,s_K) \in[0,T]^K$. Also denote, for $\bar{s}_{\sss K}, K$ fixed and any sequence of $K$ rooted graphs $H^{\star}_1,\ldots,H^{\star}_K$,
\begin{align*}
    P_{\sss r,\bar{s}_{\sss K}}^{\sss (G_n^{\sss T})}(H^{\star}_1,\ldots,H^{\star}_K)= \frac{1}{n} \sum_{v\in [n]} \mathds{1}_{\big\{\forall i \in \{1,\ldots,K\} \hspace{0.2cm} B^{s_i}_r(\markedunion,v)\simeq H^{\star}_i\big\}},
\end{align*}
where $B^{s}_r(\markedunion,v)$ denotes the $r$-neighbourhood of $v$ in the time-marked union graph $(\markedunion,v)$ restricted to edges $e$ such that $s\in\bigcup_{\sss  i=1}^{\sss N(e)}\left[\sigeon{e}{i},\sigeoff{e}{i}\right]$. Note that the rooted graphs $H^{\star}_1,\ldots,H^{\star}_K$ are not necessarily connected. While this may deviate from the customary approach of considering only connected rooted graphs, it is necessary in this context to account for indirect connections, which are integral to the analysis. We write
\begin{align} \label{susnr_write_out}
    \mathbf{E}_{\Lambda}\big[\susnr\mid \big(G_n^s\big)_{s\in[0,T]}\big]= &\sum_{k=0}^{\infty} \mathbf{P}(N_r=k) \int_{\bar{s}_{\sss k}} \sum_{\Tilde{H}^{\star}_{[k+1]}\in\Tilde{\mathcal{H}}} P_{\sss r,\bar{s}_{\sss k}}^{\sss (G_n^{\sss T})}\left(\Tilde{H}^{\star}_{[k+1]}\right)\nonumber\\ &\times\mathbf{P}_{\Lambda}\left(t<T^{(r)}\left(\Tilde{H}^{\star}_{[k+1]} \right) \mid \Tilde{H}^{\star}_{[k+1]} \right)\text{d}F^{(n)}_k(\bar{s}_{\sss k}),
\end{align}
with $\Tilde{H}^{\star}_{[k]}$ denoting a $k$-element sequence of rooted graphs, $\Tilde{\mathcal{H}}$ a set of all sequences of rooted graphs and $T^{(r)}\left(\Tilde{H}^{\star}_{[k]} \right)$ the infection time of the root in a dynamic graph with dynamic trajectory $\Tilde{H}^{\star}_{[k]}$.\\

\noindent\textbf{Step 2c: Truncating the number of changes $N_r$.} Note further that \eqref{susnr_write_out} can be split as
\begin{align}
    \mathbf{E}_{\Lambda}\big[\susnr\mid \big(G_n^s\big)_{s\in[0,T]}\big]= &\sum_{k=0}^{K} \mathbf{P}(N_r=k) \int_{\bar{s}_{\sss k}} \sum_{\Tilde{H}^{\star}_{[k+1]}\in\Tilde{\mathcal{H}}} P_{\sss r,\bar{s}_{\sss k}}^{\sss (G_n^{\sss T})}\left(\Tilde{H}^{\star}_{[k+1]}\right)\nonumber\\
    &\qquad\times\mathbf{P}_{\Lambda}\left(t<T^{(r)}\left(\Tilde{H}^{\star}_{[k+1]} \right) \mid \Tilde{H}^{\star}_{[k+1]} \right)\text{d}F^{(n)}_k(\bar{s}_{\sss k}) \nonumber\\
    &+\sum_{k=K}^{\infty} \mathbf{P}(N_r=k) \int_{\bar{s}_{\sss k}} \sum_{\Tilde{H}^{\star}_{[k+1]}\in\Tilde{\mathcal{H}}} P_{\sss r,\bar{s}_{\sss k}}^{\sss (G_n^{\sss T})}\left(\Tilde{H}^{\star}_{[k+1]}\right)\nonumber\\ &\qquad\times\mathbf{P}_{\Lambda}\left(t<T^{(r)}\left(\Tilde{H}^{\star}_{[k+1]} \right) \mid \Tilde{H}^{\star}_{[k+1]} \right)\text{d}F^{(n)}_k(\bar{s}_{\sss k})\nonumber\\
    &\equiv X_K + Y_K. \nonumber
\end{align}
Hence,
\begin{align} \label{var_split1}
    \Var_n\left(\mathbf{E}_{\Lambda}\big[\susnr\mid \big(G_n^s\big)_{s\in[0,T]}\big]\right)= \Var_n\left(X_K\right) + \Var_n\left(Y_K\right) + 2\mathrm{Cov}_n\left(X_K,Y_K\right).
\end{align}
Note that $0\leq X_K \leq1$ and $0\leq Y_K \leq1$, since $0\leq\susnr\leq1$. Hence, $\Var_n(Y_K)\leq\mathbf{E}_n\left[\left(Y_K\right)^2\right]\leq\mathbf{E}_n[Y_K]$. Furthermore, note that as a consequence of convergence of the marked union graph, for $n$ large enough, $N_r$ is finite with high probability and hence, $\mathbf{P}(N_r\geq K)$ vanishes for $K$ sufficiently large. Finally, note that the inner expression in $Y_K$, i.e.,
\begin{align}
    \int_{\bar{s}_{\sss k}} \sum_{\Tilde{H}^{\star}_{[k+1]}\in\Tilde{\mathcal{H}}} P_{\sss r,\bar{s}_{\sss k}}^{\sss (G_n^{\sss T})}\left(\Tilde{H}^{\star}_{[k+1]}\right)\times\mathbf{P}_{\Lambda}\left(t<T^{(r)}\left(\Tilde{H}^{\star}_{[k+1]} \right) \mid \Tilde{H}^{\star}_{[k+1]} \right)\text{d}F^{(n)}_k(\bar{s}_{\sss k}),
\end{align}
is also bounded by $1$, as it equals the probability that a randomly chosen vertex becomes infected, conditionally on $N_r$. Combining all of the above arguments we obtain that, for $n$ and $K$ large enough depending on $\delta>0$,
\begin{align}
\Var_n(Y_K)\leq\mathbf{E}_n[Y_K] \leq \sum_{k=K}^{\infty} \mathbf{P}_n(N_r=k) = \mathbf{P}_n(N_r\geq K) \leq \delta,
\end{align}
for any $\delta>0$. We can bound $\mathrm{Cov}_n\left(X_K,Y_K\right)$ in a similar manner, again using the fact that $0\leq X_K \leq1$ and $0\leq Y_K \leq1$ as
\begin{align}
    \mathrm{Cov}_n\left(X_K,Y_K\right) \leq \mathbf{E}_n\left[X_K \cdot Y_K\right] \leq \mathbf{E}_n\left[Y_K\right] \leq \delta.
\end{align}
Hence, we obtain that for $n$ and $K$ large enough, for any $\delta>0$,
\begin{align} \label{var_bound1}
    \Var_n\left(\mathbf{E}_{\Lambda}\big[\susnr\mid \big(G_n^s\big)_{s\in[0,T]}\big]\right) \leq \Var_n\left(X_K\right) + 6\delta.
\end{align}
We proceed to bound $\Var_n\left(X_K\right)$.\\

\noindent\textbf{Step 2d: Truncating the graph sizes.} Note that we can decompose $X_K$ as
\begin{align}
    X_K=&\sum_{k=0}^{K} \mathbf{P}(N_r=k) \int_{\bar{s}_{\sss k}} \sum_{\Tilde{H}^{\star}_{[k+1]}\in\Tilde{\mathcal{H}}} P_{\sss r,\bar{s}_{\sss k}}^{\sss (G_n^{\sss T})}\left(\Tilde{H}^{\star}_{[k+1]}\right)\times\mathbf{P}_{\Lambda}\left(t<T^{(r)}\left(\Tilde{H}^{\star}_{[k+1]} \right) \mid \Tilde{H}^{\star}_{[k+1]} \right)\text{d}F^{(n)}_k(\bar{s}_{\sss k})\\
    = &\sum_{k=0}^{K} \mathbf{P}(N_r=k) \int_{\bar{s}_{\sss k}} \sum_{\Tilde{H}^{\star}_{[k+1]}\in\Tilde{\mathcal{H}}_l} P_{\sss r,\bar{s}_{\sss k}}^{\sss (G_n^{\sss T})}\left(\Tilde{H}^{\star}_{[k+1]}\right)\times\mathbf{P}_{\Lambda}\left(t<T^{(r)}\left(\Tilde{H}^{\star}_{[k+1]} \right) \mid \Tilde{H}^{\star}_{[k+1]} \right)\text{d}F^{(n)}_k(\bar{s}_{\sss k})\nonumber\\
    &+\sum_{k=0}^{K} \mathbf{P}(N_r=k) \int_{\bar{s}_{\sss k}} \sum_{\Tilde{H}^{\star}_{[k+1]}\in\Tilde{\mathcal{H}}_l^c} P_{\sss r,\bar{s}_{\sss k}}^{\sss (G_n^{\sss T})}\left(\Tilde{H}^{\star}_{[k+1]}\right)\times\mathbf{P}_{\Lambda}\left(t<T^{(r)}\left(\Tilde{H}^{\star}_{[k+1]} \right) \mid \Tilde{H}^{\star}_{[k+1]} \right)\text{d}F^{(n)}_k(\bar{s}_{\sss k})\nonumber\\
    \equiv&X_K^{'} + X_K^{''},\nonumber
\end{align}
where we have denoted the space of sequences of graphs of size at most $l$ by $\Tilde{\mathcal{H}}_l$ and the space of sequences of graphs with at least one graph of size bigger than $l$ by $\Tilde{\mathcal{H}}_l^c$. We have 
\begin{align} \label{var_split2}
    \Var_n\left(X_K\right) = \Var_n\left(X_K^{'}\right) + \Var_n\left(X_K^{''}\right) + 2\mathrm{Cov}_n\left(X_K^{'},X_K^{''}\right).
\end{align}
We can apply the same reasoning as in the previous step. Note that, since $0\leq X_K\leq1$, $0\leq X_K^{'} \leq1$ and $0\leq X_K^{''} \leq1$. Hence, $\Var_n(X_K^{''})\leq\mathbf{E}_n\left[\left(X_K^{''}\right)^2\right]\leq\mathbf{E}_n[X_K^{''}]$. Furthermore, $\mathbf{E}_n[X_K^{''}]\leq\mathbf{E}_n[\Tilde{X}_K^{''}]$, with
\begin{align}
    \Tilde{X}_K^{''} = \sum_{k=0}^{K} \mathbf{P}(N_r=k) \int_{\bar{s}_{\sss k}} \varepsilon_r(\markedunion,l) \text{d}F^{(n)}_k(\bar{s}_{\sss k}),
\end{align}
with
\begin{align}
    \varepsilon_r(\markedunion,l) = \frac{1}{n}\sum_{v\in [n]} \mathds{1}_{\big\{|B_r(\markedunion,v)|>l\big\}}.
\end{align}
where we recall that $B_r(\markedunion,v)$ denotes the $r$-neighbourhood of $v$ in the marked union graph $(\markedunion,v)$. We compute
\begin{align} \label{tilde_Y_bound}
    \mathbf{E}_n[\Tilde{X}_K^{''}] \leq \mathbf{E}_n\left[\varepsilon_r(\markedunion,l) 
    \cdot \sum_{k=0}^{\infty} \mathbf{P}_n(N_r=k) \int_{\bar{s}_{\sss k}} \text{d}F^{(n)}_k(\bar{s}_{\sss k}) \right]= \mathbf{E}_n\left[\varepsilon_r(\markedunion,l)\right].
\end{align}
Recall our argumentation that local convergence of the marked union graph implies tight neighbourhood sizes in the union graph. Thus, given any $\delta'>0$ for large enough $l$ and $n$, $\mathbf{P}_n(\varepsilon_r(\markedunion, l)\geq\delta')\leq1-\delta'$ and hence,
\begin{align*}
    \mathbf{E}_n[\Tilde{X}_K^{''}]\leq \mathbf{E}_n\left[\varepsilon_r(\markedunion,l)\right] \leq 2\delta'.
\end{align*}
Combining the bounds on $\Var_n(X_K^{''})$ that we have derived earlier, we obtain
\begin{align*}
\Var_n(X_K^{''})\leq\mathbf{E}_n\left[\left(X_K^{''}\right)^2\right]\leq\mathbf{E}_n[X_K^{''}]\leq \mathbf{E}_n[\Tilde{X}_K^{''}].
\end{align*}
Substituting the bound on $\mathbf{E}_n[\Tilde{X}_K^{''}]$ yields $\Var_n(X_K^{''})\leq 2\delta'$. We can bound $\mathrm{Cov}_n(X_K^{'},X_K^{''})$ in a similar manner, once again thanks to the fact that $0\leq X_K^{'}\leq1$, $0\leq X_K^{''}\leq1$ as
\begin{align} \label{cov_bound}
    \mathrm{Cov}_n(X_K^{'},X_K^{''}) \leq \mathbf{E}_n[X_K^{'} X_K^{''}] \leq \mathbf{E}_n[X_K^{''}] \leq 2\delta.
\end{align}
Combining all of the above and plugging it into \eqref{var_split2} yields
\begin{align}
    \Var_n\left(X_K\right) \leq \Var_n\left(X_K^{'}\right) + 6\delta',
\end{align}
which after substituting in \eqref{var_bound1} gives us
\begin{align} \label{var_bound2}
    \Var_n\left(\mathbf{E}_{\Lambda}\big[\susnr\mid \big(G_n^s\big)_{s\in[0,T]}\big]\right) \leq \Var_n\left(X_K^{'}\right) + 6\delta + 6\delta'.
\end{align}
It remains to bound $\Var_n\left(X_K^{'}\right)$.\\

\noindent\textbf{Step 2e: Bounding the remaining variance term $\bm{\Var_n\left(X_K^{'}\right)}$.} In the last step, we apply the definition of the variance and subsequently one of consequences of local convergence of the marked union graph to obtain the final bound on $\Var_n\left(\mathbf{E}_{\Lambda}\big[\susnr\mid \big(G_n^s\big)_{s\in[0,T]}\big]\right)$.
By the definition of variance,
\begin{align} \label{2nd_mom_ran_eq1}
    &\Var_n\left(X_K^{'}\right)= \mathbf{E}_n\Bigg[\Bigg(\sum_{k=0}^{K} \mathbf{P}_n(N_r=k)\int_{\bar{s}_{\sss k}} \sum_{\Tilde{H}^{\star}_{[k+1]}\in\Tilde{\mathcal{H}}_l} \left(P_{\sss r,\bar{s}_{\sss k}}^{\sss (G_n^{\sss T})}\left(\Tilde{H}^{\star}_{[k+1]}\right) - p_r^{\left(n\right)}\left(\Tilde{H}^{\star}_{[k+1]}\right)\right)\\
&\hspace{3cm}\times\mathbf{P}_{\Lambda}\left(t<T^{(r)}\left(\Tilde{H}^{\star}_{[k+1]} \right) \mid \Tilde{H}^{\star}_{[k+1]} \right)\text{d}F^{(n)}_k(\bar{s}_{\sss k})\Bigg)^2\Bigg],\nonumber
\end{align}
where $p_r^{(n)}(H^{\star}_1,\cdots,H^{\star}_K) = \mathbf{E}_n\big[P_{\sss r,\bar{s}_{\sss K}}^{\sss (G_n^{\sss T})}(H^{\star}_1,\ldots,H^{\star}_K) \big]$ and we recall that $\mathbf{E}$ denotes the expectation with respect to the randomness of $(G_n^s)_{s\in[0,T]}$. Denote
\begin{align}
    \mathlarger{\Delta}_{\sss r,\bar{s}_{\sss k}}^{\sss (G_n^{\sss T})}\left(\Tilde{H}^{\star}_{[k]}\right) = P_{\sss r,\bar{s}_{\sss k}}^{\sss (G_n^{\sss T})}\left(\Tilde{H}^{\star}_{[k]}\right) - p_r^{\left(n\right)}\left(\Tilde{H}^{\star}_{[k]}\right),
\end{align}
and
\begin{align}
    a_k = \int_{\bar{s}_{\sss k}} \sum_{\Tilde{H}^{\star}_{[k+1]}\in\Tilde{\mathcal{H}}_l} 
    \mathlarger{\Delta}_{\sss r,\bar{s}_{\sss k}}^{\sss (G_n^{\sss T})}\left(\Tilde{H}^{\star}_{[k]}\right)\mathbf{P}_{\Lambda}\left(t<T^{(r)}\left(\Tilde{H}^{\star}_{[k+1]} \right) \mid \Tilde{H}^{\star}_{[k+1]} \right)\text{d}F^{(n)}_k(\bar{s}_{\sss k}).
\end{align}
By expanding the square in \eqref{2nd_mom_ran_eq1} we obtain
\begin{align} \label{2nd_mom_ran_eq2}
    \Var_n\left(X_K^{'}\right) = \mathbf{E}_n\left[\sum_{k=0}^K \mathbf{P}^2_n(N_r=k) a_k^2 \right] + 2\mathbf{E}_n\left[\sum_{k<k'}^K \mathbf{P}_n(N_r=k) \mathbf{P}_n(N_r=k') a_k a_k'\right].
\end{align}
Since $\mathbf{P}^2_n(N_r=k) \leq 1$ and $a_k\geq0$, the first term on the right-hand-side of \eqref{2nd_mom_ran_eq2} can be bounded by $\mathbf{E}_n\left[\sum_{k=0}^K a_k^2 \right]$. Note further that
\begin{align} \label{2nd_mom_ran_eq3}
    \mathbf{E}_n\left[\sum_{k=0}^K a_k^2 \right]= &\mathbf{E}_n\Bigg[\sum_{k=0}^{K}\Bigg(\int_{\bar{s}_{\sss k}}\sum_{\Tilde{H}^{\star}_{[k+1]}\in\Tilde{\mathcal{H}}_l} \mathlarger{\Delta}_{\sss r,\bar{s}_{\sss k}}^{\sss (G_n^{\sss T})}\left(\Tilde{H}^{\star}_{[k]}\right)\times\mathbf{P}_{\Lambda}\left(t<T^{(r)}\left(\Tilde{H}^{\star}_{[k+1]} \right) \mid \Tilde{H}^{\star}_{[k+1]} \right)\text{d}F^{(n)}_k(\bar{s}_{\sss k})\Bigg)^2\Bigg]\\
    =& \mathbf{E}_n\Bigg[\sum_{k=0}^{K}\Bigg(\sum_{\Tilde{H}^{\star}_{[k+1]}\in\Tilde{\mathcal{H}}_l} \int_{\bar{s}_{\sss k}} \mathlarger{\Delta}_{\sss r,\bar{s}_{\sss k}}^{\sss (G_n^{\sss T})}\left(\Tilde{H}^{\star}_{[k]}\right)\times\mathbf{P}_{\Lambda}\left(t<T^{(r)}\left(\Tilde{H}^{\star}_{[k+1]} \right) \mid \Tilde{H}^{\star}_{[k+1]} \right)\text{d}F^{(n)}_k(\bar{s}_{\sss k})\Bigg)^2\Bigg].\nonumber
\end{align}
By expanding the square in \eqref{2nd_mom_ran_eq3} we obtain
\begin{align} \label{expand_square1}
    \Var_n\left(X_K^{'}\right)&\leq\mathbf{E}_n\Bigg[\sum_{k=0}^{K}\sum_{\Tilde{H}^{\star}_{[k+1]}\in\Tilde{\mathcal{H}}_l} \Bigg(\int_{\bar{s}_{\sss k}} \mathlarger{\Delta}_{\sss r,\bar{s}_{\sss k}}^{\sss (G_n^{\sss T})}\left(\Tilde{H}^{\star}_{[k]}\right)\times\mathbf{P}_{\Lambda}\left(t<T^{(r)}\left(\Tilde{H}^{\star}_{[k+1]} \right) \mid \Tilde{H}^{\star}_{[k+1]} \right)\text{d}F^{(n)}_k(\bar{s}_{\sss k})\Bigg)^2\Bigg]\\
    &+\sum_{k=0}^{K}\sum_{\Tilde{H}^{\star}_{[k+1]}\in\Tilde{\mathcal{H}}_l} \sum_{\Tilde{H}'^{\star}_{[k+1]}\in\Tilde{\mathcal{H}}_l} 2 \mathbf{E}_n\Bigg[\Bigg(\int_{\bar{s}_{\sss k}} \mathlarger{\Delta}_{\sss r,\bar{s}_{\sss k}}^{\sss (G_n^{\sss T})}\left(\Tilde{H}^{\star}_{[k]}\right)\times\mathbf{P}_{\Lambda}\left(t<T^{(r)}\left(\Tilde{H}^{\star}_{[k+1]} \right) \mid \Tilde{H}^{\star}_{[k+1]} \right)\text{d}F^{(n)}_k(\bar{s}_{\sss k})\nonumber\\
    &\times \int_{\bar{s}_{\sss k}} \mathlarger{\Delta}_{\sss r,\bar{s}_{\sss k}}^{\sss (G_n^{\sss T})}\left(\Tilde{H'}^{\star}_{[k]}\right)\times\mathbf{P}_{\Lambda}\left(t<T^{(r)}\left(\Tilde{H'}^{\star}_{[k+1]} \right) \mid \Tilde{H'}^{\star}_{[k+1]} \right)\text{d}F^{(n)}_k(\bar{s}_{\sss k})\Bigg)\Bigg].\nonumber
\end{align}
Note that, again by the Cauchy-Schwarz inequality,
\begin{align} \label{eq_ah_bound1}
    \Bigg(\int_{\bar{s}_{\sss k}} &\mathlarger{\Delta}_{\sss r,\bar{s}_{\sss k}}^{\sss (G_n^{\sss T})}\left(\Tilde{H}^{\star}_{[k]}\right)\times\mathbf{P}_{\Lambda}\left(t<T^{(r)}\left(\Tilde{H}^{\star}_{[k+1]} \right) \mid \Tilde{H}^{\star}_{[k+1]} \right)\text{d}F^{(n)}_k(\bar{s}_{\sss k})\Bigg)^2\\
    &\leq \int_{\bar{s}_{\sss k}} \left(\mathlarger{\Delta}_{\sss r,\bar{s}_{\sss k}}^{\sss (G_n^{\sss T})}\left(\Tilde{H}^{\star}_{[k]}\right)\right)^2\text{d}F^{(n)}_k(\bar{s}_{\sss k}) \times \int_{\bar{s}_{\sss k}} \mathbf{P}^2_{\Lambda}\left(t<T^{(r)}\left(\Tilde{H}^{\star}_{[k+1]} \right) \mid \Tilde{H}^{\star}_{[k+1]} \right)\text{d}F^{(n)}_k(\bar{s}_{\sss k}),\nonumber
\end{align}
and hence, the first term in \eqref{expand_square1} can be again bounded by
\begin{align} \label{eq_ah_bound2}
    \mathbf{E}_n&\Bigg[\sum_{k=0}^{K}\sum_{\Tilde{H}^{\star}_{[k+1]}\in\Tilde{\mathcal{H}}_l} \Bigg(\int_{\bar{s}_{\sss k}} \mathlarger{\Delta}_{\sss r,\bar{s}_{\sss k}}^{\sss (G_n^{\sss T})}\left(\Tilde{H}^{\star}_{[k]}\right)\times\mathbf{P}_{\Lambda}\left(t<T^{(r)}\left(\Tilde{H}^{\star}_{[k+1]} \right) \mid \Tilde{H}^{\star}_{[k+1]} \right)\text{d}F^{(n)}_k(\bar{s}_{\sss k})\Bigg)^2\Bigg]\\
    &\leq \sum_{k=0}^{K}\sum_{\Tilde{H}^{\star}_{[k+1]}\in\Tilde{\mathcal{H}}_l} \int_{\bar{s}_{\sss k}} \mathbf{E}_n\left[ \left(\mathlarger{\Delta}_{\sss r,\bar{s}_{\sss k}}^{\sss (G_n^{\sss T})}\left(\Tilde{H}^{\star}_{[k]}\right)\right)^2\right]\text{d}F^{(n)}_k(\bar{s}_{\sss k}).\nonumber
\end{align}
The second term in \eqref{expand_square1} can be bounded in a similar way. To simplify the notation, denote
\begin{align}
    b_{\Tilde{H}^{\star}_{[k]}} = \int_{\bar{s}_{\sss k}} \mathlarger{\Delta}_{\sss r,\bar{s}_{\sss k}}^{\sss (G_n^{\sss T})}\left(\Tilde{H}^{\star}_{[k]}\right)\times\mathbf{P}_{\Lambda}\left(t<T^{(r)}\left(\Tilde{H}^{\star}_{[k+1]} \right) \mid \Tilde{H}^{\star}_{[k+1]} \right)\text{d}F^{(n)}_k(\bar{s}_{\sss k}),
\end{align}
so that the second term in \eqref{expand_square1} becomes $\sum_{k=0}^{K}\sum_{\Tilde{H}^{\star}_{[k+1]}\in\Tilde{\mathcal{H}}_l} \sum_{\Tilde{H}'^{\star}_{[k+1]}\in\Tilde{\mathcal{H}}_l} 2 \mathbf{E}_n\left[b_{\Tilde{H}^{\star}_{[k]}} b_{\Tilde{H'}^{\star}_{[k]}} \right].$ Using the fact that $b_{\Tilde{H}^{\star}_{[k]}} b_{\Tilde{H'}^{\star}_{[k]}}\leq \big|b_{\Tilde{H}^{\star}_{[k]}} b_{\Tilde{H'}^{\star}_{[k]}}\big|$ and subsequently applying the Cauchy-Schwarz inequality, we obtain
\begin{align} \label{eq_ah_bound3}
\sum_{k=0}^{K}\sum_{\Tilde{H}^{\star}_{[k+1]}\in\Tilde{\mathcal{H}}_l} \sum_{\Tilde{H}'^{\star}_{[k+1]}\in\Tilde{\mathcal{H}}_l} 2 \mathbf{E}\left[b_{\Tilde{H}^{\star}_{[k]}} b_{\Tilde{H'}^{\star}_{[k]}} \right] \leq \sum_{k=0}^{K}\sum_{\Tilde{H}^{\star}_{[k+1]}\in\Tilde{\mathcal{H}}_l} \sum_{\Tilde{H}'^{\star}_{[k+1]}\in\Tilde{\mathcal{H}}_l} 2 \sqrt{\mathbf{E}_n\left[\left(b_{\Tilde{H}^{\star}_{[k]}}\right)^2\right]} \sqrt{\mathbf{E}_n\left[\left(b_{\Tilde{H'}^{\star}_{[k]}}\right)^2\right]}.
\end{align}
This term, in turn, can again be bounded in the same way as the first term in \eqref{expand_square1}, by applying the same steps as in \eqref{eq_ah_bound1}-\eqref{eq_ah_bound2}. Hence, combining \eqref{eq_ah_bound1}-\eqref{eq_ah_bound3}, we arrive at
\begin{align} \label{eq_last_bound}
    \mathbf{E}_n\left[\sum_{k=0}^K a_k^2 \right] \leq &\sum_{k=0}^{K}\sum_{\Tilde{H}^{\star}_{[k+1]}\in\Tilde{\mathcal{H}}_l} \int_{\bar{s}_{\sss k}} \mathbf{E}_n\left[ \left(\mathlarger{\Delta}_{\sss r,\bar{s}_{\sss k}}^{\sss (G_n^{\sss T})}\left(\Tilde{H}^{\star}_{[k]}\right)\right)^2\right]\text{d}F^{(n)}_k(\bar{s}_{\sss k})\\ &+\sum_{k=0}^{K}\sum_{\Tilde{H}^{\star}_{[k+1]}\in\Tilde{\mathcal{H}}_l} \sum_{\Tilde{H}'^{\star}_{[k+1]}\in\Tilde{\mathcal{H}}_l} 2 \int_{\bar{s}_{\sss k}} \sqrt{\mathbf{E}_n\left[ \left(\mathlarger{\Delta}_{\sss r,\bar{s}_{\sss k}}^{\sss (G_n^{\sss T})}\left(\Tilde{H}^{\star}_{[k]}\right)\right)^2\right]}\sqrt{\mathbf{E}_n\left[ \left(\mathlarger{\Delta}_{\sss r,\bar{s}_{\sss k}}^{\sss (G_n^{\sss T})}\left(\Tilde{H'}^{\star}_{[k]}\right)\right)^2\right]}\text{d}F^{(n)}_k(\bar{s}_{\sss k}). \nonumber
\end{align}
It remains to bound the second term in \eqref{2nd_mom_ran_eq2}, i.e., $2\mathbf{E}_n\left[\sum_{k<k'}^K \mathbf{P}_n(N_r=k) \mathbf{P}_n(N_r=k') a_k a_k'\right]$. Note that
\begin{align}
    2\mathbf{E}_n\left[\sum_{k<k'}^K \mathbf{P}(N_r=k) \mathbf{P}(N_r=k') a_k a_k'\right] &= 2\sum_{k<k'}^K \mathbf{P}(N_r=k) \mathbf{P}_n(N_r=k') \mathbf{E}_n\left[a_k a_k'\right] \\
    &\leq 2\sum_{k<k'}^K \mathbf{P}_n(N_r=k) \mathbf{P}_n(N_r=k') \mathbf{E}_n\left[|a_k a_k'|\right] \leq 2\sum_{k<k'}^K \sqrt{ \mathbf{E}_n\left[a^2_k\right]\mathbf{E}\left[a^2_{k'}\right]},\nonumber
\end{align}
where the last step follows again by the Cachy-Schwarz inequality. The terms $\mathbf{E}_n\left[a^2_k\right]$ and $\mathbf{E}_n\left[a^{'2}_{k}\right]$ can be bounded by
\begin{align}
    \mathbf{E}_n\left[a^2_k\right] \leq &\sum_{\Tilde{H}^{\star}_{[k+1]}\in\Tilde{\mathcal{H}}_l} \int_{\bar{s}_{\sss k}} \mathbf{E}_n\left[ \left(\mathlarger{\Delta}_{\sss r,\bar{s}_{\sss k}}^{\sss (G_n^{\sss T})}\left(\Tilde{H}^{\star}_{[k]}\right)\right)^2\right]\text{d}F^{(n)}_k(\bar{s}_{\sss k})\\ &+\sum_{\Tilde{H}^{\star}_{[k+1]}\in\Tilde{\mathcal{H}}_l} \sum_{\Tilde{H}'^{\star}_{[k+1]}\in\Tilde{\mathcal{H}}_l} 2 \int_{\bar{s}_{\sss k}} \sqrt{\mathbf{E}_n\left[ \left(\mathlarger{\Delta}_{\sss r,\bar{s}_{\sss k}}^{\sss (G_n^{\sss T})}\left(\Tilde{H}^{\star}_{[k]}\right)\right)^2\right]}\sqrt{\mathbf{E}_n\left[ \left(\mathlarger{\Delta}_{\sss r,\bar{s}_{\sss k}}^{\sss (G_n^{\sss T})}\left(\Tilde{H'}^{\star}_{[k]}\right)\right)^2\right]}\text{d}F^{(n)}_k(\bar{s}_{\sss k}),\nonumber
\end{align}
and analogously for $\mathbf{E}_n\left[a^{'2}_{k}\right]$, by taking the same steps as in \eqref{2nd_mom_ran_eq3}-\eqref{eq_last_bound}, so that
\begin{align} \label{eq_last_bound2}
    2&\mathbf{E}_n\left[\sum_{k<k'}^K \mathbf{P}_n(N_r=k) \mathbf{P}(N_r=k') a_k a_k'\right]\\
    &\leq 2\sum_{k<k'}^K \sqrt{\sum_{\Tilde{H}^{\star}_{[k+1]}\in\Tilde{\mathcal{H}}_l} \int_{\bar{s}_{\sss k}} \mathbf{E}_{H_k}\left[\Delta^2\right]\text{d}F^{(n)}_k(\bar{s}_{\sss k})+\sum_{\Tilde{H}^{\star}_{[k+1]}\in\Tilde{\mathcal{H}}_l} \sum_{\Tilde{H}'^{\star}_{[k+1]}\in\Tilde{\mathcal{H}}_l} 2 \int_{\bar{s}_{\sss k}} \sqrt{\mathbf{E}_{H_k}\left[\Delta^2\right]}\sqrt{\mathbf{E}_{H'_k}\left[\Delta^2\right]}\text{d}F^{(n)}_k(\bar{s}_{\sss k})}\nonumber\\
    &\quad\times \sqrt{\sum_{\Tilde{H}^{\star}_{[k'+1]}\in\Tilde{\mathcal{H}}_l} \int_{\bar{s}_{\sss k'}} \mathbf{E}_{H_{k'}}\left[\Delta^2\right]\text{d}F^{(n)}_{k'}(\bar{s}_{\sss k'})+\sum_{\Tilde{H}^{\star}_{[k'+1]}\in\Tilde{\mathcal{H}}_l} \sum_{\Tilde{H}'^{\star}_{[k'+1]}\in\Tilde{\mathcal{H}}_l} 2 \int_{\bar{s}_{\sss k'}} \sqrt{\mathbf{E}_{H_{k'}}\left[\Delta^2\right]}\sqrt{\mathbf{E}_{H'_{k'}}\left[\Delta^2\right]}\text{d}F^{(n)}_{k'}(\bar{s}_{\sss k'})},\nonumber
\end{align}
with $\mathbf{E}_{H_k}\left[\Delta^2\right]=\mathbf{E}_n\left[ \left(\mathlarger{\Delta}_{\sss r,\bar{s}_{\sss k}}^{\sss (G_n^{\sss T})}\left(\Tilde{H}^{\star}_{[k]}\right)\right)^2\right]$
Thus, plugging in \eqref{eq_last_bound} and \eqref{eq_last_bound2} into in \eqref{2nd_mom_ran_eq2}, we can obtain the final bound. The expectations of the form $\mathbf{E}_n\left[ \left(\mathlarger{\Delta}_{\sss r,\bar{s}_{\sss k}}^{\sss (G_n^{\sss T})}\left(\Tilde{H}^{\star}_{[k]}\right)\right)^2\right]$, in turn, can be bounded by applying a consequence of local convergence of the marked union graph. Let $N_{r,l}$ be the number of rooted graphs of radius $r$ and size $l$. In \cite{Alimohammadi2024} the authors proved that local convergence of static graphs implies that the difference between the proportion of vertices whose neighbourhoods are isomorphic to some rooted graph $H^{\star}$ and the expected value of this proportion is arbitrarily small for $n$ large enough (see \textup{\cite[Appendix C.3]{Alimohammadi2024}} for the proof that local convergence implies a stable neighbourhood structure). The marked local convergence in probability has an analogous consequence, i.e., for a given $\tilde{\delta}\leq \frac{\delta}{N_{r,l}}$, there exists $N$ such that for all $n > N$, all $K<\infty$ and for all sequences $\Tilde{H}^{\star}_{[K]} = \left(H^{\star}_1,\ldots,H^{\star}_K\right)$, $\mathbf{E}_n\big[\big(P_{\sss r,\bar{s}_{\sss K}}^{\sss (G_n^{\sss T})}(\Tilde{H}^{\star}_{[K]}) - p_r^{(n)}(\Tilde{H}^{\star}_{[K]})\big)^2\big]\leq \delta^{''}$. Note again that, as $K$ and $\Tilde{\mathcal{H}}_l$ are finite (since the number of graphs of size at most $l$ is finite), the sums in \eqref{eq_last_bound} are over finite sets. Hence, we obtain $\Var_n\left(X_K^{'}\right)\leq 9\delta^{''}$. Substituting into \eqref{var_bound2} yields
\begin{align} \label{2nd_term_bound_final}
    \Var_n\left(\mathbf{E}_{\Lambda}\big[\susnr\mid \big(G_n^s\big)_{s\in[0,T]}\big]\right) \leq 9\delta^{''} + 6\delta + 6\delta'.
\end{align}\\
\textbf{Step 3: Deriving the final bound.} Substituting \eqref{1st_term_bound_final} and \eqref{2nd_term_bound_final} into \eqref{law_total_var}, we obtain that, for $n$ large enough and any $r\in\mathbf{N}$,
\begin{align}
    \Var\big(\susnr\big) \leq \bar{\delta} + 9\delta^{''} + 6\delta + 6\delta',
\end{align}
which can be made arbitrarily small by choosing $\delta, \bar{\delta}, \delta^{'}$ and $\delta^{''}$ small.
\end{proof}

\section{Proofs of Propositions \ref{lem_on_the_lim} and \ref{lem_conv_loc_approx}} \label{sec_all_proofs_lim}
\subsection{Proof of the bound on the local approximation on the limiting graph} \label{sec_proof_loc_approx_on_the_lim}

\begin{proof}[Proof of Proposition \ref{lem_on_the_lim}]
We adapt the proof of the equivalent result for static graphs, see \textup{\cite[Lemma 4.5]{Alimohammadi2024}}. The bound (\ref{on_the_lim_ineq}) follows by the same argument as the bound in Proposition \ref{lem_1stmom}. The second part of the claim is a consequence of the fact that $s_l(t)$ is monotone decreasing in $l$ and $r_l(t)$ is monotone increasing in $l$. This implies that the limits $s(t) = \lim_{l\to\infty}s_l(t)$ and $r(t) = \lim_{l\to\infty}r_l(t)$ are well-defined and the bound on the difference with their limits follows from (\ref{on_the_lim_ineq}). Using the fact that $i_l(t) = 1 - s_l(t) - r_l(t)$ completes the proof by extending these arguments to the developments in the number of infected vertices.
\end{proof}

\subsection{Proof of convergence of the local approximation} \label{sec_proof_conv_loc_approx}

In order to prove Proposition \ref{lem_conv_loc_approx}, we now define the following functional and subsequently show that it is continuous and bounded:

\begin{lemma}\label{lem_fct_cont_bounded}
Define the functional $h_{t,r}:\mathcal{G}_{\star}\mapsto[0,1]$ by
\begin{align} \label{functional_h} 
h_{t,r}\left(\markedg\right)= \mathbf{P}_{\Lambda}\left(t<T^{(r)}(o)\mid \markedg \right).
\end{align}
For any $t\in[0,T]$ and any $r\in\mathbf{N}$, $h_{t,r}$ is bounded and continuous in $t$.
\end{lemma}

\begin{proof}
The boundedness of $h_{t,r}$ follows immediately since probabilities are bounded between $0$ and $1$. We will derive the continuity by showing that, for two graphs with similar dynamic trajectories, the difference in the probabilities of their roots becoming infected before a given time $t\in[0,T]$ is small. Fix two distinct rooted marked union graphs $\marklim$ and $\marklimbar$. Let $\varepsilon>0$ be given. We want to find $\delta=\delta(\varepsilon)>0$ such that if
\begin{align}\label{dist_marked}
    \text{d}_{\mathcal{G}_{\star}}\left(\marklim, \marklimbar \right) < \delta,
\end{align}
where $\text{d}_{\mathcal{G}_{\star}}$ denotes the distance with respect to the metric on marked rooted graphs as in Definition \ref{metric_marked_graphs}, then
\begin{align} \label{dist_to_show}
\bigg|\mathbf{P}_{\Lambda}\bigg(t<T^{(r)}(o)\mid \marklim \bigg)- \mathbf{P}_{\Lambda}\bigg(t<T^{(r)}(o)\mid \marklimbar \bigg)\bigg| < \varepsilon.
\end{align}
By the definition of $\text{d}_{\mathcal{G}_{\star}}$, (\ref{dist_marked}) implies that the corresponding rooted union graphs $\left(G^{[0,T]},o\right)$ and $\left(\Bar{G}^{[0,T]},\Bar{o}\right)$ are isomorphic up to $r=1/\delta$. Moreover, from Definition \ref{def_metric_marks_dyn_graphs}, this also implies that for every edge $e\in\marklim$ there exists a corresponding edge $\Bar{e}\in\marklimbar$ such that $\Bar{e}=\phi(e)$, where $\phi()$ is the isomorphism between the two rooted graphs. Further, for each such pair of edges $(e,\phi(e))$,
\begin{align}
    \text{d}_{\Xi}\left(\mathcal{M}(e),\mathcal{M}(\phi(e))\right)= \Bigg(|N(e)-N(\phi(e))|+ \sum_{i=1}^{\min\{N(e),N(\phi(e))\}} \left(|t^{e}_{i,\text{\rm{ON}}}-\Bar{t}^{\phi(e)}_{i,\text{\rm{ON}}}| + |t^e_{i,\text{\rm{OFF}}}-\Bar{t}^{\phi(e)}_{i,\text{\rm{OFF}}}|\right)\Bigg) \leq \delta.
\end{align}
Thus, for $\delta$ small, for each $e$, $N(e)=N(\phi(e))$ and and for each $i\in\{1,\ldots,N(e)\}$, $\Bar{t}^{\phi(e)}_{\sss i,\text{\rm{ON}}} = t^{e}_{\sss i,\text{\rm{ON}}}\pm\delta, \Bar{t}^{\phi(e)}_{\sss i,\text{\rm{OFF}}} = t^{e}_{\sss i,\text{\rm{OFF}}}\pm\delta$. Hence, there exists an isomorphism for all paths in $B_r\marklim$ and $B_r\marklimbar$ up to $r=1/\delta$ and all edge marks on isomorphic edges in the two graphs are at most $\delta$ away. Thus, for $\delta$ sufficiently small, the dynamic $r$-neighbourhood trajectories are \emph{uniformly close} throughout $[0,T]$ up to large $r$, which will allow us to couple the infection processes in the subsequent steps of the proof.\\

To conclude the claim, first note that the right-hand side of (\ref{functional_h}) depends only on $B_r\marklim$. Hence, (\ref{dist_to_show}) can be simplified to
\begin{align} \label{dist_to_show_rewrite}
\bigg|&\mathbf{P}_{\Lambda}\bigg(t<T^{(r)}\left(B_r\marklim\bigg)\mid \marklim \right)\\
&-\mathbf{P}_{\Lambda}\bigg(t<T^{(r)}\left(B_r\marklimbar\bigg)\mid \marklimbar \right)\bigg| < \varepsilon.\nonumber
\end{align}
As we have shown that the two graph marked graphs are isomporhic up to $r=1/\delta$, it remains to show that infection times in the $1/\delta$- neighbourhoods of the roots in such graphs are close. To do so, we derive a coupling of an epidemic process on $1/\delta$-neighbourhoods of our graphs.\\

We start by coupling the sets of initially infected vertices. Recall that at time $t=0$ each vertex is supposed to be infected with some probability $\rho$ independently of all the other vertices. Define a sequence of i.i.d. random variables $(U_v)_{v\in\mathcal{V}\left(B_{1/\delta}\marklim\right)}\sim \rm{Unif}([0,1])$, where $\rm{Unif}([0,1])$ denotes uniform distribution on $[0,1]$. We set up the following infection rule: infect $v$ in $B_{1/\delta}\marklim$ if $U_v\leq\rho$; similarly, infect $\phi(v)$ in $B_{1/\delta}\marklimbar$ if $U_v\leq\rho$, where $\phi$ is again the isomorphism between the two graphs. This couples the sets of initially infected vertices in both graphs such that they are identical (with respect to the graph isomorphism). We continue to the further infection step. We generate pairs $(D^e_I, D^e_R)_e$ for each $e$ in $B_{1/\delta}\marklim$. Then, we let the infection pass through the edge $e$ if and only if $D^e_I\leq D^e_R$ and additionally $D^e_I\in \bigcup_i[t^e_{\sss i,\mathrm{ON}},t^e_{\sss i,\mathrm{OFF}}]$. Similarly, we let the infection pass through $\phi(e)$ if and only if $D^e_I\leq D^e_R$ and additionally $D^e_I\in \bigcup_i[\Bar{t}^{\phi(e)}_{\sss i,\mathrm{ON}},\Bar{t}^{\phi(e)}_{\sss i,\mathrm{OFF}}]$. Note that in such a case the probability that an infection passes through edge $e$ but not through $\phi(e)$ is equal to $\mathbf{P}(D_I^e\in\cup_i(t^e_{\sss i,\mathrm{ON}},t^e_{\sss i,\mathrm{ON}}\pm\delta]\cup(t^e_{\sss i,\mathrm{OFF}},t^e_{\sss i,\mathrm{OFF}}\pm\delta])$. Hence, for any pair $\{e,\phi(e)\}$, these ON periods differ by at most $2 \delta$. As we have assumed that $D_I$ is continuous, for each $\varepsilon'>0$ we can find $\delta>0$ such that $\sup_x\mathbf{P}(D_I^e\in[x-\delta,x+\delta])\leq \varepsilon'$. Thus, for each $\varepsilon'>0$ we can find $\delta>0$ such that $\mathbf{P}(D_I^e\in\cup_i(t^e_{\sss i,\mathrm{ON}},t^e_{\sss i,\mathrm{ON}}\pm\delta]\cup(t^e_{\sss i,\mathrm{OFF}},t^e_{\sss i,\mathrm{OFF}}\pm\delta])\leq N(e)\varepsilon'$.\\

Applying the coupling inequality, the expression in (\ref{dist_to_show_rewrite}) can be bounded by the probability that the root $o$ becomes infected by time $t$ in $B_{1/\delta}\marklim$ but the root $\Bar{o}=\phi(o)$ does not get infected in $B_{1/\delta}\marklimbar$ during the coupled epidemic process. For that to happen, at least one path from $I(0)$ - the set of initially infected vertices in $B_{1/\delta}\marklim$ - to $o$ needs to be successful and its total weight needs to be smaller than $t$, while no such path exists between $\phi(I(0))$ and $\Bar{o}$. Denote a collection of paths from $I(0)$ to $o$ by $\mathscr{P}_o$ and all paths from $\phi(I(0))$ to $\Bar{o}$ by $\mathscr{\Bar{P}}_{\Bar{o}}$. Note that due to the structural similarity between the two graphs, for each $\mathcal{P}\in\mathscr{P}_o$ there exists $\phi(\mathcal{P})\in\mathscr{\Bar{P}}_{\Bar{o}}$. For each path $ \mathcal{P} $, let $ |\mathcal{P}| $ denote its length (number of edges). Denote the probability that $ \mathcal{P} $ succeeds in $ G $ while $ \phi(\mathcal{P}) $ fails in $\Bar{G}$ by $\mathbf{P}_{\mathcal{P}}$. Then, $\mathbf{P}_{\mathcal{P}}$ is bounded by
\begin{align}
\mathbf{P}_{\mathcal{P}} \leq |\mathcal{P}| \times \varepsilon' \max_{e\in\mathcal{P}}N(e),
\end{align}
Thus,
\begin{align}
(\ref{dist_to_show}) &\leq \sum_{\mathcal{P} \in \mathscr{P}_o} \mathbf{P}_{\mathcal{P}} \leq \varepsilon'\max_{e\in\mathscr{P}_o}N(e) \sum_{\mathcal{P} \in \mathscr{P}_o} |\mathcal{P}| = \varepsilon' \kappa\max_{e\in\mathscr{P}_o}N(e),
\end{align}
where $\kappa = \sum_{\mathcal{P} \in \mathscr{P}_o} |\mathcal{P}|$. In a finite $r$-neighbourhood, the number and length of all paths is bounded (note that the paths are self-avoiding) and thus, such an expression can be made arbitrarily small by taking $\delta$ small, which also couples the epidemic up to some $r$ large as we have established the coupling for all $r\leq1/\delta$. Hence, we can conclude that for any $r\in\mathbf{N}$ and any $\varepsilon>0$, (\ref{dist_to_show}) holds with $\delta=\delta(\varepsilon)$ satisfying
\begin{align*}
 \sup_x\mathbf{P}(D_I^e\in[x-\delta,x+\delta])\leq \frac{\varepsilon}{\max_{e\in\mathscr{P}_o}N(e)\kappa}.
\end{align*}
\end{proof}

By applying this functional along with local time-marked union convergence, we can prove Proposition \ref{lem_conv_loc_approx}, as demonstrated below:

\begin{proof}[Proof of Proposition \ref{lem_conv_loc_approx}]
By the assumption that the time-marked union graph $\markedunionroot$ given by $(G_n^s)_{s\in[0,T]}$ converges locally in probability to $\marklim$, we know that, for any continuous and bounded functional $h$,
\begin{align} \label{eq_lem34_1}
\mathbf{E}_n\left[h\left(\markedunionroot\right)\mid (G_n^s)_{s\in[0,T]}\right]\stackrel{\mathbf{P}}{\longrightarrow} \mathbf{E}_{\mu}\left[h\left(\marklim\right)\right].
\end{align}
Taking $h_{r,t}$ as in (\ref{functional_h}), the right-hand side of \eqref{eq_lem34_1},
\begin{align} \label{eq_lem34_2}
\mathbf{E}_n&\left[h_{t,r}\left(\markedunionroot\right)\mid (G_n^s)_{s\in[0,T]}\right]=\mathbf{E}_n\bigg[\mathbf{P}_{\Lambda}\bigg(t<T_n^{(r)}(o_n)\mid \left(\markedunion,o_n\right) \bigg)\mid (G_n^s)_{s\in[0,T]}\bigg]\\ &=\mathbf{P}_{\Lambda}\bigg(t<T_n^{(r)}(o_n)\mid \left(\markedunion,o_n\right) \bigg)= \mathbf{E}_{\Lambda}\left[\frac{1}{n}\sum_{i\in[n]}\mathds{1}_{\{\text{$i$ infected by $t$}\}} \mid \left(\markedunion,o_n\right)\right]\nonumber\\
&= \mathbf{E}_{\Lambda}\left[\susnr \mid \left(\markedunion,o_n\right)\right] = \mathbf{E}_{\Lambda}\left[\susnr \mid \left(G_n^s\right)_{s\in[0,T]}\right].\nonumber
\end{align}
Further, the left-hand side of \eqref{eq_lem34_1} becomes
\begin{align} \label{eq_lem34_3}
    \mathbf{E}_{\mu}\left[h_{t,r}\left(\marklim\right)\right]&= \mathbf{E}_{\mu}\left[\mathbf{P}_{\Lambda}\left(t<T^{(r)}(o) \mid \marklim \right)\right]= \mathbf{P}_{\Lambda}\left(t<T^{(r)}(o)\right) = s_r(t),
\end{align}
and it follows that 
\begin{align} \label{eq_lem34_4}
    \mathbf{E}_{\Lambda}\left[\susnr \mid \left(G_n^s\right)_{s\in[0,T]}\right] \stackrel{\mathbf{P}}{\longrightarrow}  s_r(t).
\end{align}
Note that
\begin{align}
    \mathbf{E}\left[\susnr \right] = \mathbf{E}_n\left[\mathbf{E}_{\Lambda}\left[\susnr \mid \left(G_n^s\right)_{s\in[0,T]}\right]\right],
\end{align}
where $\mathbf{E}_n$ denotes the expectation with respect to the randomness of the dynamic graph. Further, note that both sides of \eqref{eq_lem34_4} are bounded by $1$, since they can be expressed in terms of probabilities (see \eqref{eq_lem34_2} and \eqref{eq_lem34_3}). Hence, applying the dominates convergence theorem to \eqref{eq_lem34_4} with $\mathbf{E}_n$ yields the claim.
\end{proof}

\section{Examples of converging dynamic random graphs} \label{sec_simulations_details}
In this section, we provide more details of the dynamic graph models presented in Section \ref{sec_examples_dyn_rg} and their local limits.

\subsection{The Erd\H{o}s-R{\'e}nyi random graphs}

In Section \ref{sec_simulations_overview}, we presented the results of our numerical simulations aimed at validating Theorem \ref{thm_dyn_conv_epid} (see Figure \ref{fig:plots_2}). By comparing the epidemic's progression on the full dynamic graph with that on its time-marked union local limit, the simulations demonstrated the accuracy of the local limit approximation. These results show that using the local limit not only simplifies computations but also maintains a high level of reliability in modelling epidemic spread on dynamic graphs. We now formally introduce the static and dynamic Erd\H{o}s-R{\'e}nyi random graph models used for this simulation, along with their local limits. 

In 1960 Erd\H{o}s and R{\'e}nyi published a paper \cite{ER_1960} which contained a very profound analysis of a graph consisting of $n$ vertices and a fixed number of edges, added uniformly at random without replacement. This paper was not the first one concerned with the topic of random graph models. However, due to multiple important results, it is commonly considered to have started the field. Note that the above formulation of the model differs from the one that has been widely popularised - a graph on $n$ vertices with an edge between each pair added independently with probability $p$. Such a definition of the model was actually introduced by Gilbert in \cite{Gilbert_1959} but because of the rich findings of \cite{ER_1960}, it remained known as the Erd\H{o}s-R{\'e}nyi random graph. Naturally, the two formulations are closely related (for details as well as an overview of results on the Erd\H{o}s-R{\'e}nyi random graph, see \cite{Bollobas84, Hofstad2016, JLR2000}).

We will use the latter definition with $p=\gamma/n$ and denote the model by $\mathrm{ER}_n(\gamma/n)$.

\subsubsection{Static Erd\H{o}s-R{\'e}nyi random graph}

To apply Theorem \ref{thm_dyn_conv_epid} we first specify the local limit of $\mathrm{ER}_n(\gamma/n)$ in the following theorem:

\begin{theorem}[Local limit of the Erd\H{o}s-R{\'e}nyi random graph \textup{\cite[Theorem 2.18]{Hofstad2023}}]
Fix $\gamma>0$. $\mathrm{ER}_n(\gamma/n)$ converges locally in probability to a Poisson branching process with mean offspring $\gamma$.
\end{theorem}

Hence, following Theorem \ref{thm_dyn_conv_epid}, to determine the time trajectory of an SIR epidemic on the Erd\H{o}s-R{\'e}nyi random graph, it suffices to establish the probability that the root of a Poisson branching process tree with mean offspring $\gamma$ becomes infected (and subsequently recovered) by some time point $t$. This reduces the problem to finding the shortest-weighted path from the set of initially infected vertex to $o$, where the weights are given by transmission times drawn from the distribution $D_I$:
\begin{align}
    \susn \stackrel{\mathbf{P}}{\longrightarrow} s(t) = \mathbf{P}\left(t<\min_{\mathcal{P}\in\mathscr{P}_o} I_o^{\mathcal{P}}\right),
\end{align}
where $\mathscr{P}_o$ denotes the set of all paths from the set of initially infected vertices to $o$ and $I_o^{\mathcal{P}}$ is the infection time of $o$ on the path $\mathcal{P}$. The weights can be determined by performing the first part of the backward process from Section \ref{sec_backward_process}. Time dynamic verification from the second part is not needed in the static case.

\subsubsection{Dynamic Erd\H{o}s-R{\'e}nyi random graph}
In this section, we expand on the dynamic Erd\H{o}s-R{\'e}nyi random graphs introduced in Sections \ref{sec_examples_dyn_rg}. We provide a detailed description of the dynamic local limit of the graph from Definition \ref{dyn_er_def} and give a justification for why the alternative formulation from Definition  \ref{def_alt_er}. Note that the stationary distribution of $\mathrm{DER}_n(\gamma/n)$ from Definition \ref{dyn_er_def} is a static $\mathrm{ER}_n(\gamma/n)$ with the same edge probability. We begin with introducing the limiting time-marked union object:
\begin{definition}[Limiting time-marked union tree] \label{def_er_lim_union}
The \emph{time-marked union tree}, denoted by $\mathrm{BP}^{\sss \mathrm{ON, OFF}}\left(\gamma(1+T)\right)$, is a Poisson branching process tree with mean offspring distribution $\gamma(1+T)$. Furthermore, each edge $e$ in $\mathrm{BP}^{\sss \mathrm{ON, OFF}}\left(\gamma(1+T)\right)$ is assigned a pair of marks $(t^{\sss e}_{\sss\rm{ON}}, t^{\sss e}_{\sss\rm{OFF}})$, where the marks are independent and identically distributed (i.i.d.) copies of random variables with the joint cumulative distribution function $F^{\sss\rm{ON,OFF}}_{\sss T}$, given by  
\begin{align} \label{law_of_the_marks_formula_er}
F^{\sss\rm{ON,OFF}}_{\sss T}(s_1, s_2) = \frac{1 - e^{-s_2 + s_1} + s_1}{1 + T}, \quad 0 \leq s_1 \leq s_2 \leq \infty.
\end{align}
\end{definition}
With the limiting object formally defined, we state the result on local time-marked union convergence of the  Dynamic Erd\H{o}s-R{\'e}nyi random graph $\mathrm{DER}_n(\gamma/n)$:
\begin{theorem}[Local time-marked union limit of the dynamic Erd\H{o}s-R{\'e}nyi random graph] \label{dyn_er_union_lim}
Fix $\gamma>0$ and a positive integer $n$. The Dynamic
Erd\H{o}s-R{\'e}nyi random graph process $\big(ER^s_n(\gamma/n) \big)_{s\in[0,T]}$ converges in probability in the local time-marked union graph sense to the time-marked union tree $\mathrm{BP}^{\sss \mathrm{ON, OFF}}\left(\gamma(1+T)\right)$ introduced in Definition \ref{def_er_lim_union}.
\end{theorem}
\begin{proof}
Local time-marked union convergence in probability of $\big(ER^s_n(\gamma/n) \big)_{s\in[0,T]}$ to $\mathrm{BP}\left(\gamma(1+T)\right)$ with the marks specified in Theorem \ref{dyn_er_union_lim} follows directly from \textup{\cite[Theorem 2.25]{Milewska2023}} (stated in the following section as Theorem \ref{thm_drig_quote}), since $\mathrm{DER}_n(\gamma/n)$ is a special case of dynamic random intersection graph (see \textup{\cite[Remark 1.7]{Milewska2023}}.
\end{proof}
\begin{remark}
In the Dynamic Erd\H{o}s-R{\'e}nyi random graph process, the number of edges that go $\mathrm{ON}$ more than once in $[0,T]$ is negligible, hence the $\mathrm{ON}$ and $\mathrm{OFF}$ marks on the edges in the limiting graph are expressed with a single interval rather than a union of intervals as in Definition \ref{def_time_marks_dyn_graphs}.
\end{remark}

\subsubsection{Alternative dynamic Erd\H{o}s-R{\'e}nyi random graph}
Note that in case of the model from Definition \ref{def_alt_er}, the union graph is equal to $\mathrm{ER}^0_n(\gamma/n)$, i.e., the graph at time $s=0$. Marks of the edges are independent and given by interarrival times of a Poisson point process, which do not depend on the graph size. This implies that local time-marked union convergence holds.

\subsection{Random intersection graphs}
In a random intersection graph, denoted by $\mathrm{RIG}$ and introduced in \cite{Singer_thesis}, vertices are associated with communities (also referred to as groups), and edges are placed between vertices when they share groups. The community memberships arise in a random way, determined by some underlying random bipartite graph which places vertices on one side and communities on the other and links these two disjoint sets. We will think of the two sets as corresponding to left- and right-vertices. Once the links between vertices and groups are assigned, the random intersection graph is a deterministic function: vertices are connected by an edge if they meet in the same community, which gives rise to a clique structure between members of groups. The procedure forming the resulting intersection graph from the underlying bipartite graph is often called a community projection or a one-mode projection.

The underlying bipartite graph of group memberships can be generated in various ways, for instance by performing a percolation on the complete bipartite graph (binomial $\mathrm{RIG}$ \cite{Fill2000, Karonski1999, Singer_thesis} or inhomogeneous $\mathrm{RIG}$ \cite{Bloznelis_Damarackas2013, Deijfen2009}), by prescribing the number of community links to each vertex and connecting them to communities chosen uniformly at random (uniform $\mathrm{RIG}$ \cite{Blackburn2009, Rybarczyk2011} or generalised $\mathrm{RIG}$ \cite{Bloznelis2010, Bloznelis2013, Bloznelis2017, Bloznelis_Jaworski_Kurauskas2013, Godehardt_Jaworski}), or by prescribing both the number of community links to every vertex and the number of members to each community and matching them uniformly at random (i.e., group memberships given by a bipartite configuration model) \cite{Lelarge2015, Hofstad2022, Hofstad2018, Newman2003}.

\subsubsection{Static random intersection graph} \label{sec_drig_stat}
In our previous work \cite{Milewska2023}, we introduced and studied both static and dynamic versions of a specific intersection graph model. In this paper, we extend that work by examining the spread of an SIR epidemic on these graphs, with the results presented in Section \ref{sec_simulations_overview}. The plots in that section were generated by applying Theorem \ref{thm_dyn_conv_epid}, simulating the progression of the SIR model on the static and dynamic local limits of the graph models. For the reader’s convenience, we now provide a brief summary of the model definition, expanding on Definition \ref{def_dyn_drig}, and the corresponding local convergence results, to clarify the objects under study.\\

\noindent\textbf{Vertices and groups.} As in Definition \ref{def_dyn_g}, let $[n] = \{1, \ldots, n\}$ be the fixed vertex set. Each vertex $i \in [n]$ is assigned a deterministic weight $w_i$. Further, let $[n]_k$ denote the set of subsets of size $k$ of $[n]$ and let $\cup_{k\geq 2} [n]_{k}$ - a union of all $k$-element subsets of $[n]$ with $k\geq2, k\in\mathbf{N}$ - denote the set of groups.\\

We start by defining the first layer corresponding to the underlying bipartite graph:\\

\noindent\textbf{The static bipartite graph.} The stationary (static) bipartite generalised random graph, denoted $\bgrg$, is formed by drawing edges between the left-vertices in $[n]$ and right-vertices (groups) in $\cup_{k\geq2}[n]_k$, with $k\geq 2$. Specifically, we draw edges between the group $a\in\cup_{k\geq2}[n]_k$ and each of its vertices from $[n]$ with probability $\piaon$, independently of all other groups, where $\piaon$ will be defined below. To be consistent with the dynamic case terminology, we say that these groups are ON. Analogously, with probability $\piaoff$, there are no edges between the group $a\in\cup_{k\geq2}[n]_k$ and each of its vertices, i.e., such groups are OFF. These probabilities are given by
\begin{align} \label{pi_ON_1}
    \pi^a_{\text{\rm{ON}}} = \frac{f(|a|)\prod_{i \in a}w_i}{\ell_n^{|a|-1}+f(|a|)\prod_{i \in a}w_i} \hspace{0.5cm} \text{and} \hspace{0.5cm} \pi^a_{\text{\rm{OFF}}} = \frac{\ell_n^{|a|-1}}{\ell_n^{|a|-1}+f(|a|)\prod_{i \in a}w_i}.
\end{align}
As one can see, we allow every collection of $k$ vertices, with $k\in[2,n]$, to form a group. We now proceed to define the resulting static intersection graph.\\

\noindent\textbf{The static intersection graph.} The second layer - the resulting static random intersection graph denoted $\drig$ - has vertex set $[n]$. It is formed from $\bgrg$ via the community projection, i.e., by drawing an edge between two vertices $i,j\in[n]$ for every ON group that they are in together in $\bgrg$. Hence, $\drig$ is a projection of $\bgrg$, the random multi-graph given by the edge multiplicities $\big(X\big(i,j)\big)_{i,j\in[n]}$, such that
\begin{align} \label{com_proj}
    X(i,j) = \sum_{k=2}^{\infty}\sum_{a\in[n]_k} \mathds{1}_{\{\text{$i$ in $a$}\}\cap\{\text{$j$ in $a$}\}}\mathds{1}_{\{\text{$a$ is \rm{ON}}\}}.
\end{align}

Next, we formulate our assumptions on the weights of left-vertices. These assumptions are necessary to guarantee (dynamic) local convergence of $\bgrg$ and $\drig$.\\

\noindent \textbf{Assumptions on weights.} We first define what the empirical distribution of the weights is:
\begin{definition}[Empirical vertex weights distribution] We define the empirical distribution function of the vertex weights as
\begin{align}
    F_n(x) = \frac{1}{n} \sum_{i \in [n]} \mathds{1}_{\{w_i \leq x\}}, \hspace{0.2cm} \text{for} \hspace{0.2cm} x\geq 0.
\end{align}
\end{definition}
\noindent $F_n$ can be interpreted as the distribution function of the weight of a vertex $o_n$, chosen from $[n]$ uniformly at random. We denote the weight of $o_n$ by $W_n=w_{o_n}$.
We impose the following conditions on the vertex weights:
\begin{condition} [Regularity condition for vertex weight] \label{cond_weights}
There exists a distribution function $F$ such that, as $n \to \infty$, the following conditions hold:
\begin{enumerate}[(a)]
    \item Weak convergence of vertex weights:
    \begin{align}
        W_n \stackrel{d}{\longrightarrow} W,
    \end{align}
    where $W_n$ and $W$ have distribution functions $F_n$ and $F$, respectively. Equivalently, for any $x$ for which $x \mapsto F(x)$ is continuous,
    \begin{align}
        \lim_{n\to\infty} F_n(x)=F(x).
    \end{align}
    \item Convergence of average vertex weight:
    \begin{align}
        \lim_{n\to\infty} \mathbf{E}[W_n] = \mathbf{E}[W],
    \end{align}
    where $W_n$ and $W$ have distribution functions $F_n$ and $F$, respectively. Further, we assume that $\mathbf{E}[W]>0$.
\end{enumerate}
\end{condition}

\noindent \textbf{Assumptions on the dependence on the group sizes.}
We take 
    \begin{equation}
    \label{f-choice}
    f(|a|) = |a|!p_{|a|},
    \end{equation}
where $(p_{k})_{k\geq2}$ is the probability mass function of the group sizes. A particularly important case is a power-law group-size distribution where $p_k$ is approximately proportional to $k^{-(\alpha+1)}$ for $k$ large. We denote
\begin{align} \label{assump_1_mom}
    \zeta = \sum_{k=2}^{\infty} kp_k,
\end{align}
and assume $\zeta<\infty$. We also assume that the second moment of the group-size distribution is finite, so that $\alpha>2$, i.e.,
\begin{align} \label{assump_2_mom}
    \zeta_{(2)} = \sum_{k=2}^{\infty} k^2p_k < \infty.
\end{align}

Having explained the static random intersection model from \cite{Milewska2023}, we proceed to explain its local limit.\\

\noindent \textbf{The static limiting object $(\mathrm{BP}_{\upsilon},0)$.} We start by introducing $(\mathrm{BP}_{\upsilon},0)$, the local limit in probability of $\bgrg$. Naturally, as we are dealing with two types of vertices - the left and the right ones - a typical neighbourhood in this graph will be different depending on the type of the root. However, it is not possible to determine whether a uniformly chosen root was a left- or a right-vertex just on the basis of its neighbourhood. Hence, we introduce additional marks to keep track of different types of vertices. Let $\Tilde{\Xi}^b=\{l,r,\varnothing\}$ be the set of marks. We mark left-vertices as $l$ and right-vertices as $r$. Formally,
\begin{align}
    \mathcal{M}^b_n(x)=
    \begin{cases}
             l \hspace{0.2cm} &\text{if} \hspace{0.2cm} x\in [n],\\
             r \hspace{0.2cm} &\text{if} \hspace{0.2cm} x\in [n]_{k\geq2},\\
             \varnothing \hspace{0.2cm} &\text{if} \hspace{0.2cm} x \hspace{0.1cm} \text{is an edge in $\bgrg$}.
    \end{cases}
\end{align}
Now we introduce the limiting object $(\mathrm{BP}_{\upsilon},\mathcal{M}^{\upsilon},0)$, the local limit of the $\bgrg$ equipped with the mark function $\mathcal{M}^b_n$, while $(\mathrm{BP}_{\upsilon},0)$ is then obtained by ignoring the mark function. Define a mixing variable $\upsilon$ as 
\begin{align}
    \mathbf{P}(\upsilon=l)=\frac{1}{1+\Bar{M}} \hspace{0.6cm} \text{and} \hspace{0.6cm} \mathbf{P}(\upsilon=r)=\frac{\Bar{M}}{1+\Bar{M}},
\end{align}
where $\Bar{M}=\mathbf{E}[W]$ is the limit in probability of $M_n/n$, with $M_n = \#\{a\in[n]_k: \text{$a$ is \rm{ON}}\}$ (see \textup{\cite[Theorem B.7]{Milewska2023}} for the derivation of $\Bar{M}$). Then, $(\mathrm{BP}_{\upsilon},\mathcal{M}^{\upsilon},0)$ is a mixture of two marked ordered BP-trees, $(\mathrm{BP}_l,\mathcal{M}^l,0)$ and $(\mathrm{BP}_r,\mathcal{M}^r,0)$:
\begin{align}
    (\mathrm{BP}_{\upsilon},\mathcal{M}^{\upsilon},0) \stackrel{d}{=} \mathds{1}_{\{\gamma=l\}}(\mathrm{BP}_l,\mathcal{M}^l,0)+\mathds{1}_{\{\gamma=r\}}(\mathrm{BP}_r,\mathcal{M}^r,0),
\end{align}
where $(\mathrm{BP}_l,\mathcal{M}^l,0)$ describes the neighbourhood of a left-vertex and $(\mathrm{BP}_r,\mathcal{M}^r,0)$ of a right-vertex. Hence,
$(\bgrg,\mathcal{M}^b_n,V_n^{(l)})$ converges locally in probability to $(\mathrm{BP}_l,\mathcal{M}^l,0)$, and $(\bgrg,\mathcal{M}^b_n,V_n^{(r)})$ converges locally in probability to $(\mathrm{BP}_r,\mathcal{M}^r,0)$, where $V_n^{(l)}$ and $V_n^{(r)}$ denote vertices chosen uniformly from the set of all left- and right-vertices respectively. The mixing variable $\gamma$ can thus be re-interpreted as the random mark of the root.\\

Before we proceed, we need to introduce the size-biased and shift version of a random variable:
\begin{definition} \label{def_shift_variable}
For an $\mathbb{N}$-valued random variable $X$ with $\mathbf{E}[X] < \infty$, we define its size-biased distribution $X^{\star}$ and
the shift variable $\Tilde{X}$ by their probability mass functions, for all $k\in\mathbb{N}$,
\begin{align}
    \mathbf{P}(X^{\star}=k)=\frac{k\mathbf{P}(X=k)}{\mathbf{E}[X]} \hspace{1cm} \text{and} \hspace{1cm} \mathbf{P}(\Tilde{X}=k)=\mathbf{P}(X^{\star}-1=k).
\end{align}
\end{definition}
Now we can continue with the description of the random ordered marked tree $(\mathrm{BP}_l,\mathcal{M}^l,0)$ itself. We consider a discrete-time branching process where the offspring of any two individuals are independent. We then give the individuals in even and odd generations marks $l$ and $r$, respectively. Generation 0 contains the root alone and the root's offspring distribution is $D^{(l)}$, where $D^{(l)}$ is a mixed-Poisson variable with mixing parameter $W\zeta$ (see \textup{\cite[Theorem B.3]{Milewska2023}}). In consecutive generations, the offspring distribution of individuals marked with $l$ will be $\Tilde{D}^{(l)}$ and of individuals marked with $r$ will be $\Tilde{D}^{(r)}$, where $D^{(r)}$ is such that, as $n\to\infty$,
\begin{align}\label{conv_unif_a}
    \mathbf{P}(D^{(r)} = k) \stackrel{\mathbf{P}}{\longrightarrow} p_k.
\end{align} 
$(\mathrm{BP}_r,\mathcal{M}^r,0)$ is defined analogously with reversed roles of $l$ and $r$.\\

\noindent \textbf{Static local limit of $\drig$.} Having specified the local limit of the underlying $\bgrg$, we proceed to the limit of the resulting graph $\drig$.\\

\noindent \textbf{The static limiting object $(\mathrm{CP},o)$.} The limit that we denote by $(\mathrm{CP},o)$ is a random rooted graph and the `community projection' (see (\ref{com_proj})) of $(\mathrm{BP}_{l}, \mathcal{M}^{l},0)$ in the same way that $\drig$ is the “community projection” of the underlying $\bgrg$: it extracts only vertices marked as $l$ and builds links between pairs of vertices that are connected to the same vertex with mark $r$. Let us accentuate that even though this limit is not a tree, it relies on the tree-like structure of the underlying $\bgrg$. This constructs the local limit $(\mathrm{CP},o)$ of $\drig$.\\

Having explained the limiting objects, we state the limiting result of $\bgrg$ and $\drig$:
\begin{restatable}[Local convergence of $\bgrg$ and $\drig$ \textup{\cite[Theorem 1.4]{Milewska2023}}]{theorem}{locconvstatic}\label{loc_conv_static} Consider $\bgrg$ under Condition \ref{cond_weights}. As $n\to\infty$, $(\bgrg,V^b_n)$ converges locally in probability to $(\mathrm{BP}_{\upsilon},0)$. Consequently for $\drig$ under Condition \ref{cond_weights}, as $n\to\infty$, $(\drig,o_n) $ converges locally in probability to $(\mathrm{CP},o)$.
\end{restatable}

For the proof of Theorem \ref{loc_conv_static} we refer the reader to \textup{\cite[Appendix B.4]{Milewska2023}}.

\subsubsection{Dynamic random intersection graph}
We now define the dynamic version of the model introduced in the previous section. The construction of the dynamic intersection graph is analogous, starting with the bipartite structure of the set of left-vertices $[n]$ and the set of right-vertices $a\in \cup_{k\geq 2}[n]_k$, followed by a community projection. However, in the dynamic setting, the groups alternate between the ON and OFF states rather than staying in one of them permanently, which gives rise to a set of edges that evolves over time as in Definition \ref{def_dyn_g}. We now characterise the rules of this process.\\

\noindent \textbf{Group dynamics.} Every group $a\in \cup_{k\geq 2}[n]_k$ will alternate between an ON and OFF state independently of all other groups, following a continuous-time Markov process. The holding times, i.e., the time that a group spends in each of the states, are exponentially distributed with rates
\begin{align} \label{holding_times}
    \lambda^a_{\text{\rm{ON}}}=1 \hspace{1cm} \text{and} \hspace{1cm} \lambda^a_{\text{\rm{OFF}}}=\frac{f(|a|)\prod_{i\in a}w_i}{\ell_n^{|a|-1}},
\end{align}
respectively, where $\bm{w} = (w_i)_{i \in [n]}$ are certain vertex weights, $\ell_n = \sum_{i\in[n]} w_i$ is the total weight and $|a|$ denotes the size of a group $a\in \cup_{k\geq 2} [n]_k$. Note that he stationary distribution $\bm{\pi} = [\pi_{\text{\rm{ON}}},\pi_{\text{\rm{OFF}}}]$ of these Markov chains is given by \eqref{pi_ON_1}, as
\begin{align} \label{pi_ON}
    \pi^a_{\text{\rm{ON}}} = \frac{\laoff}{\laon+\laoff} = \frac{f(|a|)\prod_{i \in a}w_i}{\ell_n^{|a|-1}+f(|a|)\prod_{i \in a}w_i} \hspace{0.5cm} \text{and} \hspace{0.5cm} \pi^a_{\text{\rm{OFF}}} = \frac{\laon}{\laon+\laoff} = \frac{\ell_n^{|a|-1}}{\ell_n^{|a|-1}+f(|a|)\prod_{i \in a}w_i}.
\end{align}
We initialise the group statuses with probabilities corresponding to the stationary distribution, i.e., at time $s=0$, each group is ON with probability $\piaon$ and OFF with probability $\piaoff$, independently of all other groups. To obtain our dynamic random intersection graph, we draw an edge between all the vertices in groups that are ON, so that, in the dynamic random graph, the groups represent dynamic cliques. However, to define it more precisely, we first explain the construction of the underlying dynamic bipartite graph.\\

\noindent \textbf{The dynamic bipartite graph process.} We denote the dynamic bipartite generalised random graph by $\big(\bgrgs\big)_{s\in[0,T]}$. It is a dynamic graph process in which at time $s=0$ every group is ON with probability $\piaon$ and OFF with probability $\piaoff$, independently of all other groups. For $s>0$, groups keep switching ON and OFF, according to the evolution of the continuous-time Markov Chains explained before. Analogously to the static $\bgrg$, the edges are drawn between a group $a$ and all of its vertices whenever the group is ON, and are removed when the group switches OFF (hence, for all $s\in[0,T]$, $\bgrgs$ is equal in distribution to the static $\bgrg$).\\

\noindent \textbf{The dynamic intersection graph process.} We denote the dynamic random intersection graph process by $\big(\drigs\big)_{s\in[0,T]}$. $\big(\drigs\big)_{s\in[0,T]}$ is obtained from $\big(\bgrgs\big)_{s\in[0,T]}$ similarly to the way in which $\drig$ is obtained from $\bgrg$: for every $s\in[0,T]$, an edge is drawn between $i,j\in[n]$ for every ON group that they are in together at time $s$ in $\big(\bgrgs\big)_{s\in[0,T]}$. Hence, $\big(\drigs\big)_{s\in[0,T]}$ is again a projection of $\big(\bgrgs\big)_{s\in[0,T]}$, the random dynamic multi-graph given by the edge multiplicities $ \big(X^s(i,j)\big)_{i,j\in[n],s\in[0,T]}$ such that for each $s\in[0,T]$,
\begin{align} \label{com_proj_dyn}
    X^s(i,j) = \sum_{k=2}^{\infty}\sum_{a\in[n]_k} \mathds{1}_{\{\text{$i$ in $a$}\}\cap\{\text{$j$ in $a$}\}}\mathds{1}_{\{\text{$a$ is \rm{ON} at time $s$}\}}.
\end{align}
We now introduce the following limiting time-marked union objects:

\begin{definition}[Limiting mixed time-marked union tree and its community projection]
The mixed time-marked union tree, denoted by $\left(\mathrm{BP}_{\upsilon}^{\sss[0,T],\mathrm{ON,OFF}},o\right)$ is analogous to $(\mathrm{BP}_{\upsilon}, o)$, with the key difference being that the left offspring distribution, $D^{(l)}$, is replaced by $D^{(l),[0,T]}$, a Poisson-distributed random variable with parameter $W\zeta(1+T)$. The right offspring distribution remains the same as $D^{(r)}$ and the appearance of size-biased versions of $D^{(l),[0,T]}$ and $D^{(r)}$ occurs in the same way as in $(\mathrm{BP}_{\upsilon}, o)$. Additionally, each vertex labeled with $r$ in $\left(\mathrm{BP}_{\upsilon}^{\sss[0,T],\mathrm{ON,OFF}},o\right)$ is assigned marks $(t^{\sss a}_{\sss\rm{ON}}, t^{\sss a}_{\sss\rm{OFF}})$, which are independent and identically distributed (i.i.d.) copies of random variables with the joint cumulative distribution function $F^{\sss\rm{ON,OFF}}_{\sss T}$, given by  
\begin{align} \label{law_of_the_marks_formula}
F^{\sss\rm{ON,OFF}}_{\sss T}(s_1, s_2) = \frac{1 - e^{-s_2 + s_1} + s_1}{1 + T}.
\end{align}
Furthermore, the marked community projection of $\left(\mathrm{BP}_{\upsilon}^{\sss[0,T],\mathrm{ON,OFF}},o\right)$, denoted by $\left(\mathrm{CP}^{\sss[0,T],\mathrm{ON,OFF}},o\right)$, is constructed from $\left(\mathrm{BP}_{\upsilon}^{\sss[0,T],\mathrm{ON,OFF}},o\right)$ via the community projection procedure described before (see \eqref{com_proj_dyn}). The marks on the edges between vertices in $\left(\mathrm{CP}^{\sss[0,T],\mathrm{ON,OFF}},o\right)$ inherit the marks of the $r$-labeled vertex they were linked to in $\left(\mathrm{BP}_{\upsilon}^{\sss[0,T],\mathrm{ON,OFF}},o\right)$.The marks on the edges in $\left(\mathrm{CP}^{\sss[0,T],\mathrm{ON,OFF}},o\right)$ are inherited as follows: for $l$-labeled vertices linked to the same $r$-labeled vertex in the bipartite structure, the edge between any two such $l$-labeled vertices after the community projection inherits the mark of the edges that originally linked them to the shared $r$-labeled vertex.
\end{definition}
Having formally defined the above objects, we proceed to state the result on the local time-marked union limits of $\big(\bgrgs\big)_{s\in[0,T]}$ and $\big(\drigs\big)_{s\in[0,T]}$:
\begin{theorem}[Local time-marked union limit of $\bgrgs$ and $\drigs$ \textup{\cite[Theorem 2.25]{Milewska2023}}]\label{thm_drig_quote}
Under Condition \ref{cond_weights}, as $n\to\infty$, the bipartite generalised random graph process $\big(\big(\bgrgs,o_n\big)\big)_{s\in[0,T]}$ converges locally in probability in a local time-marked union graph sense to $\left(\mathrm{BP}_{\upsilon}^{\sss[0,T],\mathrm{ON,OFF}},o\right)$. Consequently, the dynamic intersection rooted graph process 
$\big(\big(\drigs,o_n\big)\big)_{s\in[0,T]}$ converges locally in probability in the local time-marked union graph sense to $\left(\mathrm{CP}^{\sss[0,T],\mathrm{ON,OFF}},o\right)$.
\end{theorem}

For the proof of Theorem \ref{thm_drig_quote} we refer the reader to \textup{\cite[Remark 2.30 and Section 3.5]{Milewska2023}}. 

\subsection{Dynamic Configuration Model}
We now provide a heuristic argument for the local time-marked union convergence of the dynamic configuration model introduced in Definition \ref{def_dyn_CM}. At time $0$, this model reduces to the standard (static) configuration model. As $n \to \infty$, the static configuration model converges locally in distribution to a Galton-Watson tree: the root vertex has degree $k$ with probability $q_k$, with  
\begin{align}
    q_k = \lim_{n \to \infty} \mathbf{P}(D_n = k),
\end{align}
where $D_n$ is the degree of a typical vertex in the configuration model (see \cite[Theorem 4.1]{Hofstad2023} for details). Then, the offspring distribution in subsequent generations is given by $\tilde{q}_k$, with
\begin{align}
    \tilde{q}_k = \frac{(k+1)q_{k+1}}{\mathbf{E}[D_n]}.
\end{align}
Next, we consider the union graph constructed up to time $T$ from the dynamic model. This union graph includes all the edges present at time $0$, as well as any new edges formed due to rewiring events that occurred up to time $T$.\\

In the dynamic model, edges are rewired at a global rate $\alpha \ell_n$. Since two edges are chosen uniformly at random at each rewiring event, each individual edge is selected at rate $2\alpha$. Therefore, a vertex that initially has degree $k$ sees its $k$ edges chosen at a total rate $2\alpha k$. When an edge is chosen, it is re-paired in a way that creates a `new' connection with probability $2/3$. Consequently, the rate at which new edges are formed at a vertex of initial degree $k$ is $4\alpha k/3$. Over the time interval $[0,T]$, the total number of new connections formed at such a vertex is Poisson-distributed with mean $4\alpha k T/3$.\\

Thus, the local limit of the union graph up to time $T$ is again a Galton-Watson tree, but with an augmented offspring distribution. Initially, a vertex with degree $k$ contributes $k$ offspring, corresponding to its original half-edges, where the root vertex has degree distribution $q_k$, and the vertices in subsequent generations have the size-biased distribution $\tilde{q}_k$. Additionally, given its initial degree $k$, each vertex accumulates an independent number of new connections, denoted by $X_k \sim \mathrm{Poisson}(4\alpha k T/3)$, over the time interval $(0, T]$.\\

Furthermore, the times at which new edges are created are points of independent Poisson processes. Since the edge-creation times are driven by memoryless Poissonian dynamics, the time-marked version of the union graph also converges locally. Each edge in the limit can be `marked' with the arrival times of these newly formed connections, which are distributed as Poisson processes with rate $4\alpha k/3$. Therefore, the local time-marked union convergence holds as well.

\subsection{Simulation setup and interpretation}
In this section we elaborate on the simulation study presented in Section \ref{sec_sim_on_dyn_rg}.

\subsubsection{Estimation performance on the dynamic Erd\H{o}s-R{\'e}nyi random graph}

The key finding of this paper is the claim that the spread of an epidemic on the dynamic graph converges to the spread of this epidemic on a limit of the marked union graph given by this dynamic graph (see Theorem \ref{thm_dyn_conv_epid}). Consequently, the proportion of individuals infected by time $t$ converges to the probability that the root of the limiting marked union graph becomes infected by time $t$. This statement has been proven theoretically and validated empirically through simulations (see Section \ref{sec_sim_on_dyn_rg}). Below, we briefly outline the simulation setup and key observations.\\

\noindent\textbf{Theoretical intuition.} The theoretical framework predicts that the error of the limiting approximation behaves as $(1-\rho)^r$ (Proposition \ref{lem_1stmom}), where $\rho$, the proportion of initially infected vertices, strongly influences accuracy. Lower values of $\rho$ require increasing the depth 
$r$ to achieve a good approximation, adding computational complexity.\\

\noindent\textbf{Simulation setup.} To verify Theorem \ref{thm_dyn_conv_epid}, we conducted simulations on a dynamic graph with $25,000$ vertices and a corresponding limiting time-marked union tree. Assuming exponentially distributed infection and recovery times, $D_I$ and $D_R$, we tested three parameter scenarios, performing $500$ epidemic runs for each. Errors were measured between the spread on the graph and the limiting tree.\\

\noindent\textbf{Key results.} The simulations confirm the theoretical predictions. For higher values of $\rho$, such as $0.5$, the approximation is highly accurate even for small depths $r=5$. However, as $\rho$ decreases, accuracy deteriorates unless deeper neighborhoods are explored (see Figure \ref{fig:plots_2}). This highlights the trade-off between computational feasibility and accuracy for small $\rho$.\\

\noindent\textbf{Conclusion.} The simulations validate the theoretical results and provide insights into the interplay between $\rho$ and $r$. Future work may focus on more efficient methods for handling low $\rho$ cases or exploring real-world implications of the approximation.

\medskip

\noindent
\parbox{0.84\textwidth}{\paragraph{\bf Acknowledgement.} The work of MM is supported by the European Union’s Horizon 2020 research and innovation programme under the Marie Skłodowska-Curie grant agreement no. 945045, and that of MM and RvdH by the NWO Gravitation project NETWORKS under grant no. 024.002.003.}
\hspace{0.3cm}\parbox{0.14\textwidth}{
\begin{center}
\includegraphics[width=0.14\textwidth]{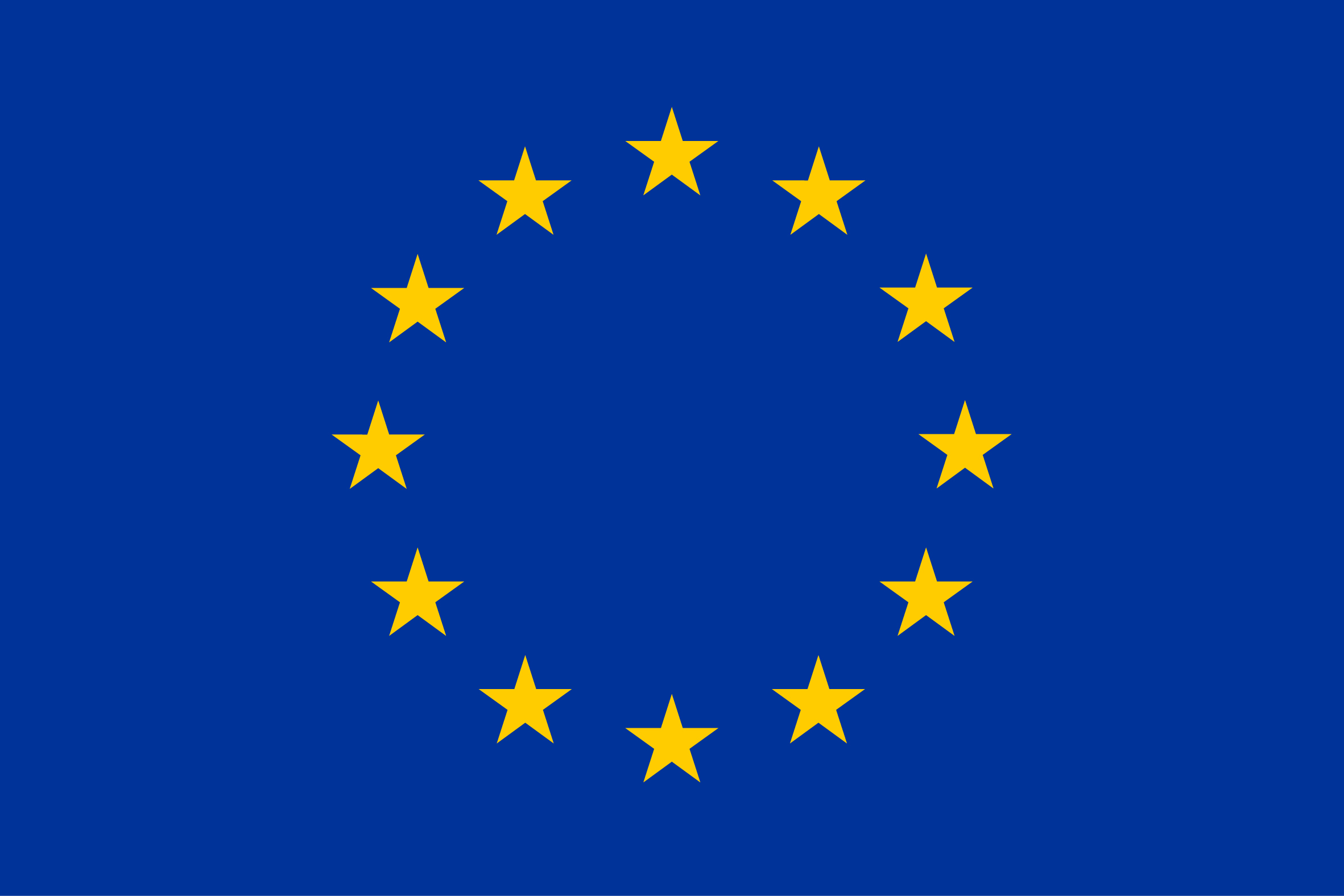}
\end{center}}



\begin{thebibliography}{10}

\bibitem{Acemoglu_2023}
D.~Acemoglu, A.~Makhdoumi, A.~Malekian, and A.~Ozdaglar.
\newblock Testing, voluntary social distancing and the spread of an infection.
\newblock {\em Operations Research}, 2023.

\bibitem{Aldous2004}
D.~Aldous and J.~M. Steele.
\newblock The objective method: probabilistic combinatorial optimization and
  local weak convergence.
\newblock In {\em Probability on discrete structures}, volume 110 of {\em
  Encyclopaedia Math. Sci.}, pages 1--72. Springer, Berlin., 2004.

\bibitem{Alimohammadi2024}
Y.~Alimohammadi, C.~Borgs, R.~van~der {Hofstad}, and A.~Saberi.
\newblock Epidemic forecasting on networks: bridging local samples with global
  outcomes.

\bibitem{Altmann_1995}
M.~Altmann.
\newblock Susceptible-infected-removed epidemic models with dynamic
  partnerships.
\newblock {\em Journal of Mathematical Biology}, 6(33):661--675, 1995.

\bibitem{AvenadHvdH2018}
L.~Avena, H.~G\"ulda\c{s}, R.~van~der {Hofstad}, and F.~den Hollander.
\newblock Random walks on dynamic configuration models: a trichotomy.
\newblock {\em Stochastic Process. Appl.}, 129(9):3360--3375, 2019.

\bibitem{Ball_2017}
F.~Ball, T.~Britton, and P.~Trapman.
\newblock An epidemic in a dynamic population with importation of infectives.
\newblock {\em The Annals of Applied Probability}, 27(1):242--274, 2017.

\bibitem{ball_britton_rewire}
Frank Ball and Tom Britton.
\newblock Epidemics on networks with preventive rewiring.
\newblock {\em Random Structures \& Algorithms}, 57(4):760--805, 2020.

\bibitem{Bartlett_1949}
M.~Bartlett.
\newblock Some evolutionary stochastic processes.
\newblock {\em Journal of the Royal Statistical Society. Series B
  (Methodological)}, 2(11):211–229, 1949.

\bibitem{Bass_1969}
F.~M. Bass.
\newblock A new product growth for model consumer durables.
\newblock {\em Management Science}, 5(15):215–227, 1969.

\bibitem{Bastani_2021}
H.~Bastani, K.~Drakopoulos, V.~Gupta, I.~Vlachogiannis, C.~Hadjichristodoulou,
  P.~Lagiou, G.~Magiorkinis, D.~Paraskevis, and S.~Tsiodras.
\newblock Efficient and targeted covid-19 border testing via reinforcement
  learning.
\newblock {\em Nature}, 7883(599):108--113, 2021.

\bibitem{Mandjes_2019}
R.~Bekker, M.~Mandjes, P.~Spreij, and N.~Starreveld.
\newblock Dynamic erd\H{o}s-r{\'e}nyi graphs.
\newblock pages 123--140, 2019.

\bibitem{Benjamini2001}
I.~Benjamini and O.~Schramm.
\newblock Recurrence of distributional limits of finite planar graphs.
\newblock {\em Electron. J. Probab.}, 6:no. 23, 13, 2001.

\bibitem{Bernoulli_1766}
D.~Bernoulli.
\newblock Essai d’une nouvelle analyse de la mortalite causee par la petite
  verole.
\newblock {\em Mem. Math. Phy. Acad. Roy. Sci. Paris}, 1766.
\newblock (English translation entitled ‘An attempt at a new analysis of the
  mortality caused by smallpox and of the advantages of inoculation to prevent
  it’, In: \textit{Smallpox Inoculation: An Eighteenth Century Mathematical
  Controversy}, Bradley L. Adult Education Department: Nottingham, 1971, 21).

\bibitem{Billingsley2013}
P.~Billingsley.
\newblock {\em Convergence of probability measures}.
\newblock John Wiley \& Sons, Inc., New York, 1999.

\bibitem{Birge_2022}
J.~R. Birge, O.~Candogan, and Y~Feng.
\newblock Controlling epidemic spread: Reducing economic losses with targeted
  closures.
\newblock {\em Management Science}, 5(68):3175–3195, 2022.

\bibitem{Blackburn2009}
S.~R. Blackburn and S.~Gerke.
\newblock Connectivity of the uniform random intersection graph.
\newblock {\em Discrete Mathematics}, 309.16:5130–5140, 2009.

\bibitem{Bloznelis2010}
M.~Bloznelis.
\newblock Component evolution in general random intersection graphs.
\newblock {\em SIAM Journal on Discrete Mathematics}, 24.2:639–654, 2010.

\bibitem{Bloznelis2013}
M.~Bloznelis.
\newblock Degree and clustering coefficient in sparse random intersection
  graphs.
\newblock {\em The Annals of Applied Probability}, 23.3:1254–1289, 2013.

\bibitem{Bloznelis2017}
M.~Bloznelis.
\newblock Degree-degree distribution in a power law random intersection graph
  with clustering.
\newblock {\em Internet Mathematics}, 2017.

\bibitem{Bloznelis_Jaworski_Kurauskas2013}
M.~Bloznelis and J.~Damarackas.
\newblock Assortativity and clustering of sparse random intersection graphs.
\newblock {\em Electronic Journal of Probability}, 18(38, 24.), 2013.

\bibitem{Bloznelis_Damarackas2013}
M.~Bloznelis and J.~Damarackas.
\newblock Degree distribution of an inhomogeneous random intersection graph.
\newblock {\em Electronic Journal of Combinatorics}, Paper 3, 13., 2013.

\bibitem{Bollobas84}
B.~Bollob{\'a}s.
\newblock The evolution of random graphs.
\newblock {\em Transactions of the American Mathematical Society},
  286(1):257--274, 1984.

\bibitem{Britton_2016}
T.~Britton, D.~Juher, and J.~Saldaña.
\newblock A network epidemic model with preventive rewiring: comparative
  analysis of the initial phase.
\newblock {\em Bulletin of Mathematical Biology}, 12(78):2427--2454, 2016.

\bibitem{Britton_2019}
T.~Britton, E.~Pardoux, F.~Ball, C.~Laredo, D.~Sirl, and V.~C Tran.
\newblock {\em Stochastic epidemic models with inference}, volume 2255 of {\em
  Lecture Notes in Mathematics}.
\newblock Springer, 2019.

\bibitem{Lelarge2015}
E.~Coupechoux and M.~Lelarge.
\newblock Contagions in random networks with overlapping communities.
\newblock {\em Advances in Applied Probability}, 47(4), 2015.

\bibitem{Croccolo_2020}
F.~Croccolo and H.~E. Roman.
\newblock Spreading of infections on random graphs: A percolation-type model
  for covid-19.
\newblock {\em Chaos, Solitons \& Fractals}, 139:110077, 2020.

\bibitem{Deijfen2009}
M.~Deijfen and W.~Kets.
\newblock Random intersection graphs with tunable degree distribution and
  clustering.
\newblock {\em Probability in the engineering and informational sciences},
  23:4:661--674, 2009.

\bibitem{Dimitrov_2010}
N.~B. Dimitrov and L.~A. Meyers.
\newblock Mathematical approaches to infectious disease prediction and control.
\newblock {\em Risk and Optimization in an Uncertain World}, pages 1--25, 2010.
\newblock INFORMS.

\bibitem{dynweaklimit2023}
L.~Dort and E.~Jacob.
\newblock Local weak limit of dynamical inhomogeneous random graphs.
\newblock 2023.

\bibitem{ER_1960}
P.~Erd\H{o}s and A.~R{\'e}nyi.
\newblock On the evolution of random graphs.
\newblock {\em Magyar Tud. Akad. Mat. Kutat\'{o} Int. K\"{o}zl.}, 5:17--61,
  1960.

\bibitem{Eubank_2004}
S.~Eubank, H.~Guclu, V.~Anil~Kumar, M.~V. Marathe, A.~Srinivasan, Z.~Toroczkai,
  and N~Wang.
\newblock Modelling disease outbreaks in realistic urban social networks.
\newblock {\em Nature}, 6988(429):180--184, 2004.

\bibitem{Fill2000}
J.~A. Fill, E.~R. Scheinerman, and K.~B. Singer-Cohen.
\newblock Random intersection graphs when $m = \omega(n)$: an equivalence
  theorem relating the evolution of the $g(n, m, p)$ and $g(n, p)$ models.
\newblock {\em Random Structures and Algorithms}, page 156–176, 2000.

\bibitem{Fransson_2019}
C.~Fransson and P.~Trapman.
\newblock Sir epidemics and vaccination on random graphs with clustering.
\newblock {\em Journal of Mathematical Biology}, 78:2369–2398, 2019.

\bibitem{Gilbert_1959}
E.~N. Gilbert.
\newblock Random graphs.
\newblock {\em The Annals of Mathematical Statistics}, 30(4):1141 -- 1144,
  1959.

\bibitem{Godehardt_Jaworski}
E.~Godehardt and J.~Jaworski.
\newblock Two models of random intersection graphs for classification.
\newblock {\em Exploratory Data Analysis in Empirical Research: Proceedings of
  the 25th Annual Conference of the Gesellschaft f{\'u}r Klassifikation
  e.V.,University of Munich,March 14–16, 2001}, page 67–81, 2003.
\newblock Edited by M. Schwaiger and O. Opitz. Berlin, Heidelberg: Springer
  Berlin Heidelberg.

\bibitem{Hofstad2016}
R.~van~der {Hofstad}.
\newblock {\em Random graphs and complex networks Volume 1}.
\newblock Cambridge University Press, 2016.

\bibitem{Hofstad2023}
R.~van~der {Hofstad}.
\newblock {\em Random graphs and complex networks. Volume 2.}
\newblock Cambridge University Press, 2024.

\bibitem{Hofstad2018}
R.~van~der {Hofstad}, J.~Komj{\'a}thy, and V.~Vadon.
\newblock Random intersection graphs with communities.
\newblock {\em Adv. in Appl. Probab.}, 53(4):1061–1089, 2021.

\bibitem{Hofstad2022}
R.~van~der {Hofstad}, J.~Komj{\'a}thy, and V.~Vadon.
\newblock Phase transition in random intersection graphs with communities.
\newblock {\em Random Structures \& Algorithms}, 60(3):406--461, 2022.

\bibitem{Milewska2023}
R.~van~der {Hofstad}, M.~Milewska, and B.~Zwart.
\newblock Dynamic random intersection graph: dynamic local convergence and
  giant structure.
\newblock {\em Random Structures \& Algorithms}, 66(1), 2024.

\bibitem{JLR2000}
S.~Janson, T.~Łuczak, and A.~Ruciński.
\newblock {\em Random graphs}.
\newblock Wiley-Interscience, 2000.

\bibitem{kallenberg2002foundations}
O.~Kallenberg.
\newblock {\em Foundations of modern probability}.
\newblock Springer-Verlag, New York, 2002.

\bibitem{Karonski1999}
M.~Karonski, E.~R. Scheinerman, and K.~B. Singer-Cohen.
\newblock On random intersection graphs: The subgraph problem.
\newblock {\em Combinatorics, Probability and Computing}, 8.1 $\&$ 2:131--159,
  1999.

\bibitem{Kendall_1956}
D.~G. Kendall.
\newblock {\em Deterministic and stochastic epidemics in closed populations}.
\newblock University of California Press, 1956.

\bibitem{McKendrick_1927}
W.~O. Kermack and A.~G. McKendrick.
\newblock A contribution to the mathematical theory of epidemics.
\newblock {\em Proceedings of the Royal Society of London. Series A, Containing
  Papers of a Mathematical and Physical Character}, 772(115):700–721, 1927.

\bibitem{Larson_2007}
R.~C. Larson.
\newblock Simple models of influenza progression within a heterogeneous
  population.
\newblock {\em Operations research}, 3(55):399--412, 2007.

\bibitem{Lashari_2018}
A.~A. Lashari and P.~Trapman.
\newblock Branching process approach for epidemics in dynamic partnership
  network.
\newblock {\em Journal of Mathematical Biology}, 1-2(76):265--294, 2018.

\bibitem{Lloyd-Smith_2009}
J.~O. Lloyd-Smith, D.~George, K.~M. Pepin, V.~E. Pitzer, J.~R. Pulliam, A.~P.
  Dobson, P.~J. Hudson, and B.~T. Grenfell.
\newblock Epidemic dynamics at the human-animal interface.
\newblock {\em Science}, 5958(316):1362--1367, 2009.

\bibitem{Mamani_2013}
H.~Mamani, S.~E. Chick, and D.~Simchi-Levi.
\newblock A game-theoretic model of international influenza vaccination
  coordination.
\newblock {\em Management Science}, 7(59):1650--1670, 2013.

\bibitem{Mandjes_2024}
M.~Mandjes and J.~Wang.
\newblock Estimation of on- and off-time distributions in a dynamic erd{\H
  o}s-r{\'e}nyi random graph.
\newblock 2024.

\bibitem{Manshadi_2020}
V.~Manshadi, S.~Misra, and S.~Rodilitz.
\newblock Diffusion in random networks: Impact of degree distribution.
\newblock {\em Operations research}, 6:1722–1741, 2020.

\bibitem{Milewska2025}
M.~Milewska.
\newblock {\em Mathematical insights into epidemics: from overdispersion to
  dynamic local convergence}.
\newblock PhD thesis, Eindhoven University of Technology, 2025.

\bibitem{Newman2003}
M.~E.~J. Newman.
\newblock Properties of highly clustered networks.
\newblock {\em Physical Review E}, 68.2:131--159, 2003.

\bibitem{Rath_2024}
B.~R{\'a}th, M.~Sz{\H{o}}ke, and L.~Warnke.
\newblock Local limit of the random degree constrained process.
\newblock 2024.

\bibitem{Trapman2016_dynER}
S.~Rosengren and P.~Trapman.
\newblock A dynamic erd\H{o}s-r\'enyi graph model.
\newblock 2016.

\bibitem{Ross_1916}
R.~Ross.
\newblock An application of the theory of probabilities to the study of a
  priori pathometry.—- part i.
\newblock {\em Proceedings of the Royal Society of London. Series A, Containing
  papers of a mathematical and physical character}, 638(92):204--230, 1916.

\bibitem{Ross_1917}
R.~Ross and H.~P. Hudson.
\newblock An application of the theory of probabilities to the study of a
  priori pathometry.—- part iii.
\newblock {\em Proceedings of the Royal Society of London. Series A, Containing
  papers of a mathematical and physical character}, 650(93):225--240, 1917.

\bibitem{Rybarczyk2011}
K.~Rybarczyk.
\newblock Diameter, connectivity, and phase transition of the uniform random
  intersection graph.
\newblock {\em Discrete Mathematics}, 311.17:1998--2019, 2011.

\bibitem{Singer_thesis}
K.~B. Singer.
\newblock {\em Random intersection graphs}.
\newblock 1996.
\newblock Thesis (Ph.D.)–The Johns Hopkins University. ProQuest LLC, Ann
  Arbor, MI.

\bibitem{Sousi_2020}
P.~Sousi and S.~Thomas.
\newblock Cutoff for random walk on dynamical erd\H{o}s-r{\'e}nyi graph.
\newblock {\em Annales de l’Institut Henri Poincar{\'e}, Probabilit{\'e}s et
  Statistiques}, 56(4), 2020.

\end{thebibliography}

\end{document}